\documentclass[11pt]{amsart}
\usepackage[T1]{fontenc}
\usepackage{ifpdf}
\usepackage[utf8]{inputenc}
\usepackage[english]{babel}
\usepackage{url}
\usepackage{graphicx}
\usepackage{silence,lmodern}
\usepackage[style=alphabetic, backend=biber, giveninits=true, url=false]{biblatex}
\DeclareLabelalphaTemplate{
  \labelelement{
    \field[final]{shorthand}
    \field{label}
    \field[strwidth=2, strside=left, ifnames=1]{labelname}
    \field[strwidth=1, strside=left]{labelname}
    }
        }
\DeclareFieldFormat{extraalpha}{#1}
\DeclareNameAlias{author}{given-family}
\renewbibmacro{in:}{}

\DeclareFieldFormat*{title}{\mkbibitalic{#1}}
\DeclareFieldFormat*{citetitle}{\mkbibitalic{#1}}
\DeclareFieldFormat[article]{journaltitle}{#1}
\DeclareFieldFormat{citejournal}{#1}
\DeclareFieldFormat*[article]{journaltitle}{#1}
\DeclareFieldFormat[article]{citejournal}{#1}
\DefineBibliographyExtras{english}{%
}
\usepackage{xpatch}
\xpatchbibdriver{article}{%
\newunit\newblock
\printfield{edition}%
}
{%
\setunit{\addcomma\space}\newblock
\printfield{edition}%
}{}{}

\makeatletter
\newcommand*\rel@kern[1]{\kern#1\dimexpr\macc@kerna}
\newcommand*\widebar[1]{%
  \begingroup
  \def\mathaccent##1##2{%
    \rel@kern{0.8}%
    \overline{\rel@kern{-0.8}\macc@nucleus\rel@kern{0.2}}%
    \rel@kern{-0.2}%
  }%
  \macc@depth\@ne
  \let\math@bgroup\@empty \let\math@egroup\macc@set@skewchar
  \mathsurround\z@ \frozen@everymath{\mathgroup\macc@group\relax}%
  \macc@set@skewchar\relax
  \let\mathaccentV\macc@nested@a
  \macc@nested@a\relax111{#1}%
  \endgroup
}
\makeatother

\usepackage{capt-of}
\usepackage{eucal}
\usepackage[letterpaper, left=2.5cm, right=2.5cm, top=2.5cm,
bottom=2.5cm,dvips]{geometry} 

\bibliography{ms}
\usepackage{geometry}
\usepackage[toc,page]{appendix}
\usepackage{extarrows}

\usepackage{amsmath,amsthm}
\usepackage{ stmaryrd }
\usepackage{mathtools}
\usepackage{hyperref}
\usepackage{enumerate}
\usepackage{indentfirst}
\usepackage{color,soul}
\usepackage{centernot}
\usepackage{ stmaryrd }
\usepackage[export]{adjustbox}
\usepackage{tikzit}
\tikzstyle{new edge style 0}=[<-]
\tikzstyle{new edge style 1}=[->]

\usepackage{turnstile}
\usepackage{amssymb,amscd,amsthm, verbatim,amsmath,color,fancyhdr, mathrsfs}
\usepackage{array}
\usepackage{enumerate}
\usepackage{empheq}
\usepackage{tikz-cd}
\usepackage[most]{tcolorbox}
\usepackage{xcolor}
\usepackage[all]{xy}
\usepackage[notransparent]{svg}
\usepackage{relsize}
\usepackage{pdfpages}

\usepackage{csquotes}
\usepackage{pxfonts}
\usepackage{bbm}

\usepackage{scalerel,stackengine}
\stackMath
\newcommand\reallywidehat[1]{%
\savestack{\tmpbox}{\stretchto{%
  \scaleto{%
    \scalerel*[\widthof{\ensuremath{#1}}]{\kern-.6pt\bigwedge\kern-.6pt}%
    {\rule[-\textheight/2]{1ex}{\textheight}}%WIDTH-LIMITED BIG WEDGE
 }{\textheight}% 
}{0.5ex}}%
\stackon[1pt]{#1}{\tmpbox}%
}
\usepackage{todonotes}
\newcommand{\uss}{\text{uss}}
\newcommand{\id}{\text{id}}

\newcommand{\Hom}{\text{Hom}}
\def\R{\mathbb{R}}

\def\Z{\mathbb{Z}}

\def\XB{\text{XB}}
\def\PD{\text{PD}}

\def\pd{\text{PD}}
\def\B{\mathbb{B}}
\def\R{\mathbb{R}}
\def\XP{\mathbb{XP}}
\def\P{\mathbb{P}}
\def\N{\mathbb{N}}
\def\Alg{\text{Alg}_{\Bbbk}^2}
\def\Alge{\text{Alg}_{\Bbbk}}

\def\F{\mathsf{F}}

\def\Tr{{\text{Tr}}}
\def\one{\imath}
\def\<{{\langle}}
\def\>{{\rangle}}
\def\d{{\partial}}

\newcommand{\leftelbow}{\includegraphics{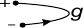}}
\newcommand{\rightelbow}{\includegraphics{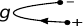}}
\newcommand{\arti}{\includegraphics[trim=0 -0.3mm 0 0]{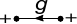}}

\newcommand{\linartg}{\includegraphics[trim=0 +0.7mm 0 0]{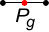}}
\newcommand{\linarte}{\includegraphics[trim=0 +0.7mm 0 0]{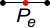}}
\newcommand{\eksi}{\includegraphics[trim=0 -0.3mm 0 0]{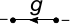}}

\newcommand{\lineksg}{\includegraphics[trim=0 +0.7mm 0 0]{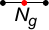}}
\newcommand{\linekse}{\includegraphics[trim=0 +0.7mm 0 0]{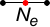}}
\newcommand{\foldcupg}{\includegraphics[trim=0 +0.7mm 0 0]{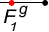}}
\newcommand{\foldcapg}{\includegraphics[trim=0 +0.7mm 0 0]{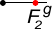}}
\newcommand{\nokta}{\includegraphics[trim=0 -0.1mm 0 0]{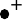}}

\newcommand{\noktab}{\includegraphics[trim=0 -0.1mm 0 0]{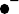}}

\newcommand{\noktac}{\includegraphics[trim=0 -0.3mm 0 0]{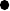}}
\newcommand{\leftelbov}{\includegraphics{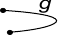}}
\newcommand{\leftelbowe}{\includegraphics{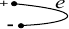}}
\newcommand{\leftelbote}{\includegraphics{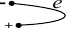}}

\newcommand{\leftelbovg}{\includegraphics[trim=0 0.3mm 0 0]{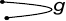}}
\newcommand{\leftelbovt}{\includegraphics[trim=0 0.1mm 0 0]{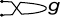}}
\newcommand{\leftelbor}{\includegraphics[trim=0 0.4mm 0 0]{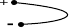}}
\newcommand{\rightelbov}{\includegraphics{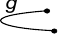}}
\newcommand{\rightelbovt}{\includegraphics[trim=0 -0.1mm 0 0]{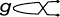}}

\newcommand{\rightelbos}{\includegraphics[trim=0 0.4mm 0 0]{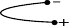}}
\newcommand{\rightelbowe}{\includegraphics{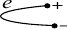}}
\newcommand{\rightelbote}{\includegraphics{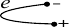}}
\newcommand{\bos}{\includegraphics[trim=0 -0.4mm 0 0]{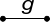}}

\newcommand{\cuspa}{\includegraphics[trim=0 0.6mm 0 0]{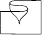}}
\newcommand{\cuspb}{\includegraphics[trim=0 0.6mm 0 0]{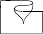}}
\newcommand{\commua}{\includegraphics[trim=0 0.4mm 0 0]{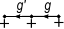}}
\newcommand{\commub}{\includegraphics[trim=0 0.2mm 0 0]{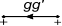}}
\newcommand{\commuc}{\includegraphics[trim=0 1.5mm 0 0]{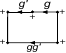}}
\newcommand{\syma}{\includegraphics[trim=0 2.1mm 0 0]{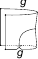}}
\newcommand{\gcirc}{\includegraphics[trim=0 0mm 0 0]{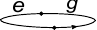}}
\newcommand{\teknokta}{\includegraphics[trim=0 0mm 0 0]{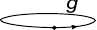}}

\makeatletter
\newcommand\footnoteref[1]{\protected@xdef\@thefnmark{\ref{#1}}\@footnotemark}
\makeatother

\newcommand{\G}{\mathcal{G}}
\newcommand{\bp}{\mathtt{bp}}

\usepackage{tikz}

\theoremstyle{definition}
\newtheorem{definition}{Definition}[section]
\newtheorem*{remark}{Remark}
\newtheorem{example}{Example}[section]

\theoremstyle{plain}
\newtheorem{theorem}{Theorem}[section]
\newtheorem{cor}{Corollary}[theorem]
\newtheorem{lemma}{Lemma}[section]
\newtheorem{prop}{Proposition}[section]

\newcommand{\op}{\text{op}}
\newcommand{\ori}{\text{or}}
\newcommand{\un}{\text{un}}

\newcommand{\Bord}{\text{Bord}}

\newcommand{\X}{\text{X}}
\usetikzlibrary{decorations.markings, decorations.fractals,calc,trees,positioning,arrows,fit,shapes,decorations.text,decorations.pathmorphing}
\graphicspath{{images/}}
\usepackage{wrapfig}
\title{Two-Dimensional Extended Homotopy Field Theories}
\author{K\"urşat S\"ozer}
\newcommand{\Address}{{\bigskip \footnotesize
K\"urşat S\"ozer, \textsc{Laboratoire Paul Painlev\'e, Universit\'e de Lille,  Villeneuve d'Ascq, Lille, France 59655} \par\nopagebreak \textit{E-mail address:} \textit{kursat.sozer@univ-lille.fr}}}
\begin{document}
\begin{abstract}
    We give another definition of two-dimensional extended homotopy field theories (E-HFTs) with aspherical targets and classify them. When the target of E-HFT is chosen to be a $K(G,1)$-space, we classify E-HFTs taking values in the symmetric monoidal bicategory of algebras, bimodules, and bimodule maps by certain Frobenius $G$-algebras called quasi-biangular $G$-algebras. As an application, for any discrete group $G$, we verify a special case of the $(G \times SO(2))$-structured cobordism hypothesis due to Lurie. 
\end{abstract}
\maketitle
\tableofcontents
\section{Introduction}
Extended topological field theories (E-TFTs) are generalizations of topological field theories (usually called TQFTs or TFTs) to manifolds with corners and higher categories (\cite{freedext1}, \cite{lawrence}, \cite{baezdolan}). A different generalization of TFTs is obtained by considering manifolds equipped with principal $G$-bundles. When $G$ is a discrete group, such a generalization was introduced by Turaev~\cite{ilkhqft} who called them homotopy (quantum) field theories (HFTs). These theories are defined by applying axioms of TFTs to manifolds and cobordisms endowed with maps to a fixed target space. In this paper, we combine E-TFTs and HFTs in dimension $2$. More precisely, we define $2$-dimensional extended homotopy field theories (E-HFTs) with aspherical targets and classify them. 

\subsection{Main results} To define a $2$-dimensional E-HFT with target $X\simeq K(G,1)$ we introduce an X-cobordism bicategory $X\Bord_2$. The objects of $X\Bord_2$ are compact oriented $0$-dimensional manifolds and the $1$-morphisms are oriented cobordisms between such manifolds equipped with homotopy classes of maps to $X$. The $2$-morphisms of $\X\Bord_2$ are equivalence classes of pairs $(S,P)$, where $S$ is a certain type of oriented surface with corners and $P$ is a homotopy class of a map from $S$ to $X$. The equivalence relation is given by diffeomorphisms respecting $P$ and restricting to the identity on the boundary. The disjoint union operation turns $X\Bord_2$ into a symmetric monoidal bicategory and \textit{a $2$-dimensional E-HFT with target $X$ (extended X-HFT)} is defined as a symmetric monoidal $2$-functor from $X\Bord_2$ to any other symmetric monoidal bicategory.

For a given symmetric monoidal bicategory $\mathcal{C}$, our classification of $\mathcal{C}$-valued $2$-dimensional extended X-HFTs is comprised of two steps. Firstly, we define certain combinatorial diagrams in $I=[0,1]$, $I^2$, and $I^3$, called $G$-linear, $G$-planar, and $G$-spatial diagrams, respectively. These diagrams generalize the ones in \cite{schommer} and they possess the same information as the morphisms of $\X\Bord_2$. As the second step, we define a symmetric monoidal bicategory $\XB^{\PD}$ whose $1$- and $2$-morphisms are defined using these diagrams. This bicategory is equivalent to $\X\Bord_2$ and has a convenient description in terms of generators and relations. Figures \ref{fig:generators} and \ref{fig:relations} show the corresponding list of generators and relations for $\X\Bord_2$. Then, the classification of $2$-dimensional extended X-HFTs reduces to an application of the cofibrancy theorem (\cite{schommer}), which is a coherence theorem for symmetric monoidal $2$-functors explained below. 

For a computadic symmetric monoidal bicategory $\mathsf{F}(\P)$ constructed from a list of generators and relations $\P$, let $\text{SymMon}(\mathsf{F}(\P),\mathcal{C})$ denote the bicategory of symmetric monoidal $2$-functors, transformations, and modifications. The cofibrancy theorem gives an equivalence $\text{SymMon}(\mathsf{F}(\P),\mathcal{C}) \simeq \P(\mathcal{C})$ of bicategories, where $\P(\mathcal{C})$ is the bicategory called \textit{$\P$-data in $\mathcal{C}$} whose objects are assignments of generators in $\P$ to the objects, $1$-morphisms, and $2$-morphisms of $\mathcal{C}$ subject to relations (see Section \ref{cofibrancy_section}). Applying this theorem to the list of generators and relations $\XP$ of $\XB^{\PD}$ along with the equivalence $\X\Bord_2 \simeq \XB^{\PD}$, gives the following classification theorem.

\newtheorem*{thm:genelsinif}{Theorem \ref{thm:genelsinif}}
\begin{thm:genelsinif}
For any symmetric monoidal bicategory $\mathcal{C}$, there is an equivalence of bicategories $$\text{SymMon}(\X\Bord_2,\mathcal{C})\simeq \XP(\mathcal{C}).$$ 
\end{thm:genelsinif}
Next, we consider a specific target bicategory $\Alg$ of $\Bbbk$-algebras, bimodules, and bimodule maps for a commutative ring $\Bbbk$ with unit. The following notions are the main ingredients of our result on $\Alg$-valued $2$-dimensional extended X-HFTs. For a discrete group $G$ with identity element $e$, a strongly graded $G$-algebra is a $G$-graded associative $\Bbbk$-algebra $A= \oplus_{g \in G} A_g$ with unity such that $A_g A_{g'} = A_{gg'}$ for all $g,g' \in G$. The opposite $G$-algebra of $A$ is $A^{op}= \oplus_{g \in G} A_{g^{-1}}$ where the order of multiplication is reversed. 

A Frobenius $G$-algebra is a pair $(A,\eta)$, where $A=\oplus_{g \in G}A_g$ is a $G$-algebra such that each $A_g$ is a finitely generated projective $\Bbbk$-module, and $\eta: A \otimes A \to \Bbbk$ is a nondegenerate bilinear form satisfying $\eta(ab,c)=\eta(a,bc)$ for any $a,b,c \in A$. A quasi-biangular $G$-algebra is a strongly graded Frobenius $G$-algebra $(A,\eta)$ whose identity component $A_e$ is separable and $\eta$ satisfies certain conditions (see Section \ref{alg_or_sinif}). We also need $G$-graded Morita contexts between $G$-algebras, which were introduced by Boisen~\cite{gmorita}. We recall their definition and introduce a notion of compatibility with Frobenius structures in Section \ref{alg_or_sinif}.
\newtheorem*{thm:orclass}{Theorem \ref{thm:orclass}}
\begin{thm:orclass}
Let $\Bbbk$ be a commutative ring and $X$ be a CW-complex which is a $K(G,1)$-space for a discrete group $G$. Then, any $\Alg$-valued $2$-dimensional extended X-HFT $Z: \X\Bord_2 \to \Alg$ whose precomposition $\XB^{\PD} \xrightarrow{\simeq} \X\Bord_2 \xrightarrow{Z} \Alg$ gives a strict symmetric monoidal $2$-functor determines a triple $(A,B,\zeta)$ where $A$ and $B$ are quasi-biangular $G$-algebras and $\zeta$ is a compatible $G$-graded Morita context between $A$ and $B^{\op}$. Conversely, for any such triple $(A,B,\zeta)$ there exists an $\Alg$-valued $2$-dimensional extended X-HFT.
\end{thm:orclass}
Theorem \ref{orclass} generalizes Schommer-Pries' classification of $\Alg$-valued $2$-dimensional extended TFTs (\cite{schommer}) in terms of separable symmetric Frobenius algebras. Theorem \ref{genelsinif} suggests studying the bicategory $\XP(\Alg)$ to understand $\text{SymMon}(\X\Bord_2,\Alg)$, more specifically to answer the question of which triples yield equivalent extended X-HFTs. This study leads us to define a bicategory, $\text{Frob}^G$, which has quasi-biangular $G$-algebras as objects, compatible $G$-graded Morita contexts as $1$-morphisms, and equivalences of such Morita contexts as $2$-morphisms (see Section \ref{alg_or_sinif}). 
\newtheorem*{thm:equivbicats}{Theorem \ref{thm:equivbicats}}
\begin{thm:equivbicats}
Under the assumptions of Theorem \ref{orclass}, there is an equivalence of bicategories 
\begin{align*}
    \text{SymMon}(\X\Bord_2,\Alg)\simeq \text{Frob}^G.
\end{align*}
\end{thm:equivbicats}
On the level of objects this equivalence maps a triple $(A,B,\zeta)$ to $A$. Consequently, Theorem \ref{thm:equivbicats} implies the following corollary.
\newtheorem*{cor:impliedequiv}{Corollary \ref{impliedequiv}}
\begin{cor:impliedequiv}
Under the assumptions of Theorem \ref{orclass}, two triples $(A_1,B_1,\zeta_1)$ and $(A_2,B_2,\zeta_2)$ produce equivalent $2$-dimensional extended X-HFTs if and only if there exists a compatible $G$-graded Morita context between $A_1$ and $A_2$.
\end{cor:impliedequiv} 

A different approach to the classification of $2$-dimensional E-HFTs with $K(G,1)$-targets is given by the $(G \times SO(2))$-structured cobordism hypothesis due to Lurie~\cite{lurie}. This hypothesis states a classification of such E-HFTs in terms of homotopy $(G \times SO(2))$-fixed points (see Section \ref{gso2sch}). Davidovich \cite{davidovich} computed these fixed points in $\Alg$ when $\Bbbk$ is an algebraically closed field of characteristic zero. By comparing Theorem \ref{equivbicats} with Davidovich's results, we verify a special case of the $(G \times SO(2))$-structured cobordism hypothesis as follows.
\newtheorem*{cor:Gcobhyp}{Corollary \ref{Gcobhyp}}
\begin{cor:Gcobhyp}
For any discrete group $G$ and any algebraically closed field $\Bbbk$ of characteristic zero, the $(G \times SO(2))$-structured cobordism hypothesis for $\Alg$-valued oriented E-HFTs with target $X\simeq K(G,1)$ holds true. 
\end{cor:Gcobhyp}
In the definition of $\X\Bord_2$ we use oriented manifolds. By using unoriented manifolds, we define the unoriented $X$-cobordism bicategory $\X\Bord_2^{\un}$ and provide a list of generators and relations. Then, parallel to oriented case, we classify $2$-dimensional extended unoriented HFTs and verify a special case of the $(G \times O(2))$-structured cobordism hypothesis. %\\[.2cm]
\subsection{Related works} Our main reference is Schommer-Pries' thesis \cite{schommer} on the classification of $2$-dimensional extended TFTs. In addition to detailed classification of oriented and unoriented extended TFTs, Schommer-Pries also sketched the classification of $2$-dimensional structured extended TFTs (see Section 3.5 of \cite{schommer}). In this approach the structured cobordism bicategory is defined using topological stacks. In particular, a stack corresponding to principal $G$-bundle and orientation structures provides an alternative formulation for extended $(G \times SO(2))$-structured TFTs, or equivalently $2$-dimensional extended HFTs with $K(G,1)$-targets. 

When defining $2$-dimensional extended TFTs, Schommer-Pries \cite{schommer} defined the cobordism bicategory in all dimensions not only in dimension 2. Schweigert and Woike \cite{woike_schweigert_ehqfts_orbifoldization} extended this cobordism bicategory to the setting of homotopy field theories and defined extended HFTs in all dimensions. Similarly, M\"uller and Szabo \cite{muller_szabo_indexthy} constructed a geometric cobordism bicategory $\text{Cob}^{\mathcal{F}}_{n,n-1,n-2}$ where $\mathcal{F}$ is a general stack which encodes the arbitrary background fields in the corresponding quantum field theory (see also \cite{lukasmuller_thesis}). In their work \cite{muller_szabo_discretegauge} when the background fields $\mathcal{F}$ are chosen for $2$-dimensional Dijkgraaf-Witten theories, namely principal bundles with finite structure group $G$ and orientations, the symmetric monoidal bicategory $\text{Cob}^{\mathcal{F}}_{2,1,0}$ is equivalent to $\X\Bord_2$ described above. Using the $n$-dimensional version $\text{Cob}^{\mathcal{F}}_{n,n-1,n-2}$ of this bicategory, they defined an $n$-dimensional extended homotopy field theory for all $n \geq 2$. Moreover, M\"uller and Woike \cite{muller_woike_higherflatgerbes_ehqfts} constructed an $n$-dimensional extended HFT from a flat $(n-1)$-gerbe on a target space represented by $U(1)$-valued singular cocycle (see also \cite{lukasmuller_thesis} and \cite{gerbes_hqfts}). This construction is generalized to unoriented extended HFTs by Young \cite{young_orientation_twistedHQFTs}. 

A state sum approach to both $2$-dimensional extended TFTs and HFTs was taken by Davidovich \cite{davidovich}. As mentioned above Davidovich \cite{davidovich} also classified $\Alg$-valued $2$-dimensional extended $(G \times SO(2))$-structured TFTs following the Cobordism Hypothesis. \\[.3cm]
\textbf{Conventions.}
Throughout the paper, $G$ is a discrete group with identity element $e$ and the target space is a pointed aspherical CW-complex $(X,x)$ with $\pi_1(X,x)=G$. All manifolds are assumed to be smooth and all algebras are unital. By a closed manifold we mean a compact manifold without boundary. For smooth manifolds $M$ and $N$ the space of smooth maps $C^{\infty}(M,N)$ is provided with the Whitney $C^{\infty}$-topology. For subsets $K \subset M$ and $L \subset N$ the notation $[(M,K),(N,L)]$ stands for the set of relative homotopy classes of maps between pairs. \\[.3cm]    
\textbf{Acknowledgments.} I would like to thank my advisor Vladimir Turaev for introducing this problem to me and his support throughout this project. I would also like to thank Noah Snyder for fruitful and enlightening discussions on extended field theories and the cobordism hypothesis. I am grateful to Patrick Chu for helpful discussions and to Alexis Virelizier for his comments on the earlier version of this paper. I would like to thank the referee for valuable comments and suggestions. This work was supported by NSF grant DMS-1664358.

\section{The $2$-dimensional X-cobordism Bicategory}
In his study of HFTs, Turaev \cite{ilkhqft} introduced notions of $X$-manifold and $X$-cobordism using pointed manifolds where $X$ is a connected CW-complex with a specified point $x\in X$. In this paper, $X$ is always a $K(G,1)$-space. In this case, a \textit{pointed manifold} is a manifold with a basepoint on each connected component. We denote the set of basepoints of a pointed manifold $M$ by $\mathtt{bp}_M$. An \textit{$n$-dimensional $X$-manifold} is a pair $(M,\mathtt{g})$ consisting of a closed pointed $n$-manifold $M$ and a homotopy class $\mathtt{g} \in [(M,\bp_M),(X,x)]$ called the \textit{characteristic map}. An \textit{$X$-cobordism} between $X$-manifolds $(M,\mathtt{g})$ and $(M',\mathtt{g}')$ is a pair $(W,\mathtt{P})$ consisting of a cobordism $W$ between $M$ and $M'$ and a homotopy class $\mathtt{P} \in [(W, \bp_{M} \cup \bp_{M'}),(X,x)]$ restricting to $\mathtt{g}$ and $\mathtt{g}'$ on the corresponding boundary components. 

For an aspherical space X, a $2$-dimensional extended X-HFT is a symmetric monoidal $2$-functor from the $2$-dimensional X-cobordism bicategory $\X \Bord_2$ to another symmetric monoidal bicategory. Therefore, this bicategory plays a key role in the definition of $2$-dimensional extended X-HFTs. 

%\begin{remark}
There are existing definitions of structured or equivariant cobordism bicategories related to $\X\Bord_2$. These include the homotopy bicategory of the $(\infty,2)$-category of cobordisms with $(G \times SO(2))$-structures (see \cite{lurie}, \cite{bordn}), the structured cobordism bicategory $\Bord_2^{\mathcal{F}}$ introduced in \cite{schommer} with a topological topological stack $\mathcal{F}$ corresponding to principal $G$-bundle and orientation structures, the $G$-equivariant cobordism bicategory $G$-$\text{Cob}(2,1,0)$ introduced in \cite{woike_schweigert_ehqfts_orbifoldization}, and the structured cobordism bicategory $\text{Cob}^{\mathcal{F}}_{2,1,0}$ introduced in \cite{muller_szabo_indexthy} with an appropriate choice of a stack $\mathcal{F}$ (see \cite{muller_szabo_discretegauge}, \cite{lukasmuller_thesis}). It can be shown that these symmetric monoidal bicategories are equivalent to $\X \Bord_2$ and hence the corresponding extended HFTs with aspherical targets are equivalent.  
%\end{remark}

We start this section with a descriptive definition of $\X \Bord_2$ to motivate the types of X-manifolds and structures on them, which form the $1$- and $2$-morphisms of this bicategory. Then we provide the complete definition of $\X \Bord_2$. 
\subsection{Surfaces with corners} Roughly, the objects of $\X \Bord_2$ are compact oriented $0$-manifolds, $1$-morphisms are $1$-dimensional X-cobordisms, and $2$-morphisms are X-homeomorphism classes of $2$-dimensional X-cobordisms between those X-cobordisms. This hints that underlying manifold of a $2$-morphism must be a surface with corners. Recall that a surface with corners $M$ is a $2$-dimensional topological manifold whose coordinate charts are of the form $\varphi : U \to \R_{+}^2$ where $U \subset M$ is open and $\R_{+}^2 = [0 , \infty) \times [0 , \infty)$. Compatibility of charts is given by diffeomorphisms; that is, two charts $(U,\varphi)$ and $(U', \varphi')$ with $U \cap U' \neq \emptyset$ are compatible if the composition $\varphi' \circ \varphi^{-1} : \varphi(U \cap U') \to \varphi'(U \cap U')$ is a diffeomorphism. Here when $\varphi(U \cap U') \cap \d \R^2_{+} \neq \emptyset$, the map $\varphi' \circ \varphi^{-1}$ is a diffeomorphism if it is a restriction of a diffeomorphism defined on an open set containing $\varphi(U \cap U')$.

The composition of $2$-morphisms in the X-cobordism bicategory is given by gluing surfaces with corners along their common boundaries. Recall that the first step of gluing construction is to choose collar neighborhoods. Nevertheless, not every surface with corners admits collar neighborhoods. Following \cite{schommer}, we use $\<2\>$-surfaces which are special cases of $\<n\>$-manifolds defined in \cite{laures}. These surfaces admit collar neighborhoods (see \cite{laures}) and they are certain type of surfaces with faces.

A surface with faces $M$ is a surface with corners such that any point $m \in M$ belongs to $\text{index}(m)$ different connected faces. Here, the index of a point $m$ is the number of zeros in $\varphi(m) \in \R^2_{+}$ where $(U,\varphi)$ is a chart with $m \in U$ and a connected face of a surface with corners $M$ is the closure of a component of $\{ m \in M  \ | \ \text{index}(m)=1\}$. A face is a disjoint union of connected faces. 
\begin{definition}
A \textit{$\<2\>$-surface} is a $2$-dimensional compact manifold with faces $S$ equipped with two submanifolds with faces $\d_h S$ and $\d_v S$ called \textit{horizontal} and \textit{vertical faces} respectively such that $\d S= \d_h S \cup \d_v S$ and $\d_h S \cap \d_v S$ is either empty or a face of both. A $\<2\>$-surface $S$ is \textit{pointed} if it is equipped with a finite set $R \subset \d S$ such that $\d_h S \cap \d_v S \subset R$, $\d_v S \cap R = \d_h S \cap \d_v S$, and every connected component of $\d_h S$ contains at least $2$ elements of $R$. 
\end{definition}
%\begin{figure}[ht]
%    \centering
%    \includesvg{2xsurface3withcollars}
%    \caption{Examples of cobordism type $\<2\>$-X-surfaces and their compositions}
%    \label{fig:2xsurface3}
%\end{figure}

\begin{figure}[ht]
    \centering
    \def\svgwidth{\columnwidth}
   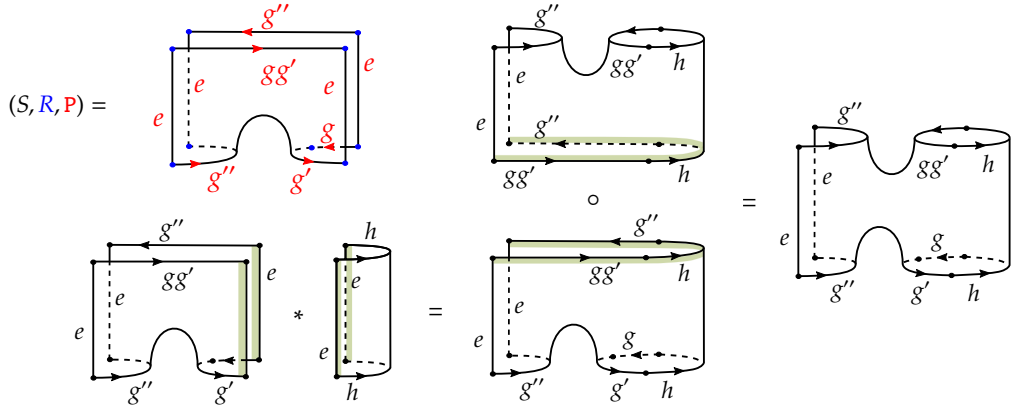
    \caption{Examples of cobordism type $\<2\>$-X-surfaces and their compositions}
   \label{fig:2xsurface3}
\end{figure}
\begin{definition}\label{2Xdefn}
A \textit{$\<2\>$-X-surface} is a triple $(S,R,\mathtt{P})$ where $(S,R)$ is a pointed oriented $\<2\>$-surface and $\mathtt{P} \in [(S,R), (X,x)]$ is a homotopy class. A $\<2\>$-X-surface $(S,R,\mathtt{P})$ is said to be \textit{cobordism type} if $\d_v S$ is diffeomorphic to a product X-manifold with a constant characteristic map i.e. $(\d_v S, \mathtt{P}|_{\d_v S}) \cong (M \times I,\mathtt{P}|_{M \times I})$ where $I=[0,1]$, $(M,\mathtt{P}|_M)$ is a $0$-dimensional X-manifold, and the restriction of $\mathtt{P}|_{M \times I} \in [(M \times I, \d(M \times I)) ,(X,x)]$ to each connected component is the constant homotopy class.
\end{definition}
Figure \ref{fig:2xsurface3} shows an example of a cobordism type $\<2\>$-X-surface $(S,R,\mathtt{P})$ where we encode the data of relative homotopy class $\mathtt{P}$ by arrows and $G$-labels are determined uniquely by $\mathtt{P}$ and arrows. Observe that the horizontal boundary of a $\<2\>$-X-surface $(S',R',\mathtt{P}')$ is not a $1$-dimensional X-cobordism if $R' \neq \d (\d_h S')$. Since we regard $2$-morphisms of $\X\Bord_2$ to be X-cobordisms between $1$-morphisms, this observation implies that $1$-morphisms are more general than $1$-dimensional X-cobordisms in that there are possibly extra points in the interior of the underlying $1$-dimensional compact manifold. 
\begin{definition}
A \textit{$1$-dimensional marked X-manifold} is a triple $(M,T,\mathtt{g})$ where $M$ is an oriented compact $1$-manifold, $T \subseteq M$ is a finite set with $\d M \subset T$ and each connected component of $M$ contains at least 2 elements of $T$, and $\mathtt{g} \in [(M,T),(X,x)]$.
\end{definition}
According to the arguments above the bicategory $\X \Bord_2$ is expected to have compact oriented $0$-manifolds as objects, $1$-dimensional marked X-manifolds as $1$-morphisms, and cobordism type $\<2\>$-X-surfaces as $2$-morphisms. In this case, for a given cobordism type $\<2\>$-X-surface, its source and target $1$-morphisms are certain components of the horizontal boundary. However, the composition of morphisms is a delicate issue, especially the composition of $1$-morphisms. 

When we glue two manifolds along their common boundary, the smooth structure on the resulting topological manifold depends on the choice of collar neighborhoods (see \cite{hcobordism}). Equivalently, different choices of collars give different smooth structures. However, different choices give diffeomorphic smooth manifolds and diffeomorphisms are non-canonical. Therefore, gluing operation on smooth manifolds is not well defined on the nose, but up to a non-canonical diffeomorphism. 

The same results continue to hold for $\<2\>$-X-surfaces. Let $(S,R,\mathtt{P})$ be a $\<2\>$-X-surface and $(N, N \cap R, \mathtt{P}|_N)$ be a face with the inclusion map $\iota  : (N, N \cap R) \hookrightarrow (S,R)$. For a collar neighborhood $U_N \subset S$ of $N$, \textit{collar} is a diffeomorphism $\Psi_N: U_N \to N \times \R_+$ (see Figure \ref{fig:2xsurface3}). The following proposition implies that $2$-morphisms of $\X\Bord_2$ must be (relative) diffeomorphism classes of cobordism type $\<2\>$-X-surfaces in order to have well-defined horizontal and vertical compositions of $2$-morphisms.   
\begin{prop}
Let $(S,R,\mathtt{P})$ and $(S',R',\mathtt{P}')$ be cobordism type $\<2\>$-X-surfaces with faces and $(N,T,\mathtt{g})$ be a $1$-dimensional marked X-manifold together with inclusions $\iota: (N, T,\mathtt{g}) \hookrightarrow (S,R,\mathtt{P})$ and $\iota': (N,T,\mathtt{g}) \hookrightarrow (S',R', \mathtt{P}')$ realizing $(N,T,\mathtt{g})$ as a face of both $\<2\>$-X-surfaces. Then $S \cup_{N} S'$ is a topological manifold with boundary. If in addition we are given collars $\Psi_+ : N \times \R_+ \to S$ and $\Psi_- : N \times \R_+ \to S'$, then there exists a canonical smooth structure on $S \cup_N S'$ which is compatible with the smooth structures on $S$ and $S'$. Moreover, different choices of collars produce non-canonically diffeomorphic cobordism type $\<2\>$-surfaces.  
\end{prop}
The proof of this theorem follows from the proofs of Proposition 3.1 and Theorem 3.3 in \cite{schommer}. Note that gluing $\<2\>$-X-surfaces vertically along their common horizontal boundary components does not yield a cobordism type $\<2\>$-X-surface. One needs to choose a diffeomorphism $I \cup I \cong I$ and omit the points on the faces through which $\<2\>$-X-surfaces are glued. Figure \ref{fig:2xsurface3} shows examples of vertical and horizontal compositions of $\<2\>$-X-surfaces denoted by $\circ$ and $\ast$, respectively. 
Following \cite{schommer}, we solve the problem of composition of $1$-morphisms by equipping manifolds with germs of neighborhoods. The notion of a germ of neighborhoods was made precise by Schommer-Pries (see Section 3.2.3 in \cite{schommer}) using halations which are formulated as maps of pro-manifolds.

\subsection{Pro-X-manifolds and X-halations}\label{halation_subsectionu}
Recall that a directed set is a tuple $(D,\leq)$ where $D$ is a nonempty set and $\leq$ is a reflexive and transitive binary relation such that for any $x,y \in D$ there exists $z \in D$ with $x \leq z$ and $y \leq z$. We think of a directed set $(D, \leq)$ as a category $\mathcal{D}$ whose objects are elements of $D$ and morphisms are given by the relation $\leq$. Let $\text{Man}^X$ be the category of smooth X-manifolds and smooth pointed maps commuting with characteristic maps. A \textit{pro-X-manifold} is a pair $(\mathcal{D},A)$ where $\mathcal{D}$ is a directed set and $A: \mathcal{D} \to \text{Man}^X$ is a functor. 

Two directed sets play an important role in describing germs of neighborhoods of X-manifolds: the first one is trivial one $\mathcal{D}_{\bullet}= \{ \bullet \}$ and the second one is associated to an embedding of X-manifolds as follows. Let $(M,\mathtt{g})$ and $(N,\mathtt{h})$ be X-manifolds possibly with boundary or corners, and let $\iota :(M,\mathtt{g}) \hookrightarrow (N,\mathtt{h})$ be an embedding of X-manifolds i.e. $\iota(\bp_M)=\bp_N$ and $\mathtt{g} = \mathtt{h} \circ [\iota]$ as elements of $[(M,\bp_M),(X,x)]$. Then the directed set $\mathcal{D}_N$ is given by codimension zero closed X-submanifolds of $N$ containing $\iota(M)$ and the relation is inclusion. 

For a given X-manifold $(M,\mathtt{g})$ we denote the pro-X-manifolds corresponding to these directed sets with $(M,\mathtt{g})$ and $(\hat{M} \subset N,\hat{\mathtt{g}})$ respectively. There is an obvious inclusion $\imath_M: (M,\mathtt{g}) \hookrightarrow (\hat{M} \subset N,\hat{\mathtt{g}})$ of pro-X-manifolds and an \textit{X-halation} is an inclusion of pro-X-manifolds isomorphic to $\imath_M$. Here morphisms of pro-X-manifolds are the morphisms in $\text{Man}^X$ of the corresponding limits and colimits of the diagrams. More precisely, the set of morphisms between two pro-X-manifolds discussed above is 
\begin{align*}
    \Hom_{\text{pro-}\text{Man}^X}(\mathcal{D}_{\bullet}, \mathcal{D}_N) = \lim_p \underset{q}{\text{colim } }\Hom_{\text{Man}^X}(\mathcal{D}_{\bullet}(p),\mathcal{D}_N(q))
\end{align*}
where the limit and colimit are taken in the category of sets. The codimension of an X-halation is the codimension of the embedding. We denote an X-halation by $(M,\hat{M},\hat{\mathtt{g}})$ and call an X-manifold equipped with an X-halation an X-\textit{haloed manifold}. A map between X-haloed manifolds $(A,\hat{A},\hat{\mathtt{a}})$ and $(B,\hat{B},\hat{\mathtt{b}})$ is a pair of pro-X-manifold morphisms $A \to B$ and $\hat{A} \to \hat{B}$ such that the diagram involving inclusions $A \hookrightarrow \hat{A}$ and $B\hookrightarrow \hat{B}$ commutes. The category of pro-objects in a category $\mathcal{C}$ is generally defined using cofiltered diagrams instead of directed sets. Here we use the results in \cite{schommer} to simplify arguments and refer reader to Sections 3.2.1 and 3.2.2 of \cite{schommer} for a more detailed exposition on halations. 

The solution to the problem of composition of $1$-morphisms is to use compatible X-haloed manifolds. The compatibility is given by choosing an orientation for the normal bundle of the embedding which defines X-halation. Such an X-halation is called \textit{co-oriented}. Now assume that $(M_0,T_0,\mathtt{g}_0)$ and $(M_1,T_1,\mathtt{g}_1)$ are $1$-dimensional marked X-manifolds equipped with co-oriented codimension one X-halations $(\hat{M}_0 \subset N_0,\hat{\mathtt{g}}_0)$ and $(\hat{M}_1 \subset N_1, \hat{\mathtt{g}}_1)$ respectively. Using the tubular neighborhood theorem and the co-orientations of the source and target objects, we write these X-halations as $(M_i \subset M_i \times \R \cup_{\d M_i \times \R \times \{0 \} } \d M_i \times (\R \times \R_{+}))$ for $i=0,1$. Using Lemma 3.25 in \cite{schommer} we refine the index of both pro-X-manifolds to natural numbers $\N$ as
\begin{align*}
    i &\mapsto M_0 \times \{0 \} \cup_{\d_t M_0} \d_t M_0 \times \{0\} \times \bigg( -\dfrac{1}{i},0 \bigg] \cup_{\d_s M_0} \d_s M_0 \times \{0 \} \times \bigg[ 0 , \dfrac{1}{i} \bigg) \\
    i &\mapsto M_1 \times \{0 \} \cup_{\d_s M_1} \d_s M_1 \times \{0\}  \times \bigg[ 0 , \dfrac{1}{i} \bigg) \cup_{\d_t M_1} \d_t M_1 \cup \{0 \} \times \bigg( -\dfrac{1}{i},0 \bigg] \\
    i &\mapsto Y \times \bigg( - \dfrac{1}{i}, \dfrac{1}{i} \bigg)
\end{align*}
where $\d_s$ and $\d_t$ denote the source and target boundary components and $Y \cong \d_t M_0 \cong \d_s M_1$. For each $i \in \N$ the pushout exists and it is the smooth manifold $M_1 \cup_Y M_0$ whose smooth structure is determined by the embedding of $Y \times \bigg( -\dfrac{1}{i}, \dfrac{1}{i} \bigg)$, and by the smooth structures of $M_0$ and $M_1$. The pushout as a co-oriented codimension one X-haloed manifold, i.e. as an $1$-morphism of $\X\Bord_2$, exists by the results in \cite{isaksen} on commutativity of finite colimits and (cofiltered) limits. Here note that the colimit is taken over the pushout diagram $K = \bullet \leftarrow \bullet \rightarrow \bullet$ which is clearly finite (see Section 3.2.3 of \cite{schommer} for details).
\subsection{The X-cobordism bicategory $\X\Bord_2$}
We are now equipped with the necessary information to define the $2$-dimensional X-cobordism bicategory.
\begin{definition}
\textit{The $2$-dimensional X-cobordism bicategory $\X\Bord_2$} has
\begin{enumerate}
    \item triples $\big\{ \big((M,\mathtt{g}), (M, \hat{M}_1,\hat{\mathtt{g}}_1), (M, \hat{M}_2,\hat{\mathtt{g}}_2)\big)\big\}$ as objects where $(M,\mathtt{g})$ is a compact oriented $0$-manifold, $(M,\hat{M}_1,\hat{\mathtt{g}}_1)$ and $(M, \hat{M}_2,\hat{\mathtt{g}}_2)$ are co-oriented codimension one and codimension two X-halations respectively with inclusions $(M,\mathtt{g}) \hookrightarrow (M,\hat{M}_1,\hat{\mathtt{g}}_1) \hookrightarrow (M,\hat{M}_2,\hat{\mathtt{g}}_2)$. For brevity we denote such an object with $(M,\hat{M}_1,\hat{M}_2,\hat{\mathtt{g}}_2)$,
   % \begin{figure}[ht]
%    \centering
 %   \includesvg{halation21}
 %   \caption{Co-oriented X-halations and an X-haloed $1$-cobordism}
 %   \label{fig:haloed1}
%\end{figure}
\begin{figure}[ht]
    \centering
    \def\svgwidth{\columnwidth}
   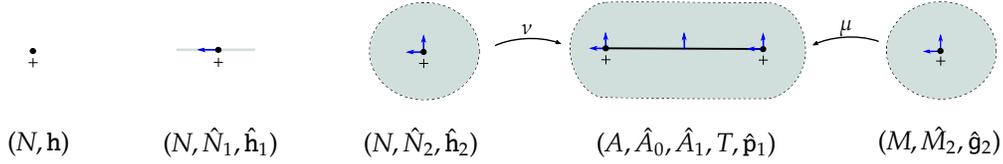
    \caption{Co-oriented X-halations and an X-haloed $1$-cobordism}
    \label{fig:haloed1}
\end{figure}
    \item X-haloed $1$-dimensional X-cobordisms as $1$-morphisms; an X-haloed $1$-dimensional X-cobordism $(A,\hat{A}_0,\hat{A}_1,T,\hat{\mathtt{p}}_1)$ from $(M,\hat{M}_1,\hat{M}_2,\hat{\mathtt{g}}_2)$ to $(N,\hat{N}_1,\hat{N}_2,\hat{\mathtt{h}}_2)$ consists of 
    \begin{itemize}
        \item a $1$-dimensional marked X-manifold $(A,T,\mathtt{p})$, 
        \item a codimension zero X-halation $(A,\hat{A}_0,\hat{\mathtt{p}}_0)$ and a co-oriented codimension one X-halation $(A,\hat{A}_1, \hat{\mathtt{p}}_1)$ with 
        inclusions $(A,\mathtt{p}) \hookrightarrow (A,\hat{A}_0,\hat{\mathtt{p}}_0) \hookrightarrow (A, \hat{A}_1, \hat{\mathtt{p}}_1)$,
        \item a decomposition of the boundary of $(A,T,\mathtt{p})$ as $\d A = \d_{in} A \amalg \d_{out} A$ with isomorphisms of X-halations preserving co-orientations (see Figure \ref{fig:haloed1})
        \begin{align*}
             (M,\hat{M}_1,\hat{M}_2,\hat{\mathtt{g}}_2) &\xrightarrow[\cong]{\mu}  (\d_{in}A ,\hat{A}_0|_{\d_{in}}, \hat{A}_1|_{\d_{in}},\hat{\mathtt{p}}_1)\\
             (N,\hat{N}_1,\hat{N}_2,\hat{\mathtt{h}}_2)  &\xrightarrow[\cong]{\nu}  (\d_{out}A, \hat{A}_0|_{\d_{out}}, \hat{A}_1|_{\d_{out}},\hat{\mathtt{p}}_1).
        \end{align*}
    where $\hat{A}_0|_{\d_{in}}$ is co-oriented by an inward pointing normal vector and $\hat{A}_0|_{\d_{out}}$ is co-oriented by an outward pointing normal vector.
    \end{itemize}
     \item isomorphism classes of X-haloed $2$-dimensional X-cobordisms as $2$-morphisms; an X-haloed $2$-dimensional X-cobordism $(S,\hat{S},R,\hat{\mathtt{F}})$ from $(A,\hat{A}_0,\hat{A}_1,T,\hat{\mathtt{p}}_1)$ to $(B,\hat{B}_0,\hat{B}_1,Q,\hat{\mathtt{q}}_1)$ consists of a cobordism type $\<2\>$-X-surface $(S,R,\mathtt{F})$ together with a codimension zero X-halation $(S,\hat{S},\hat{\mathtt{F}})$, and isomorphisms of X-halations (see Figure \ref{fig:Sdecomp})
   %  \begin{figure}[ht]
%    \centering
%   \includesvg[]{bicategorysaddle2}
 %   \caption{An example of decomposition of a $2$-morphism in $\X\Bord_2$}
 %   \label{fig:Sdecomp}
 %   \end{figure}
    \begin{figure}[ht]
    \centering
    \def\svgwidth{\columnwidth}
   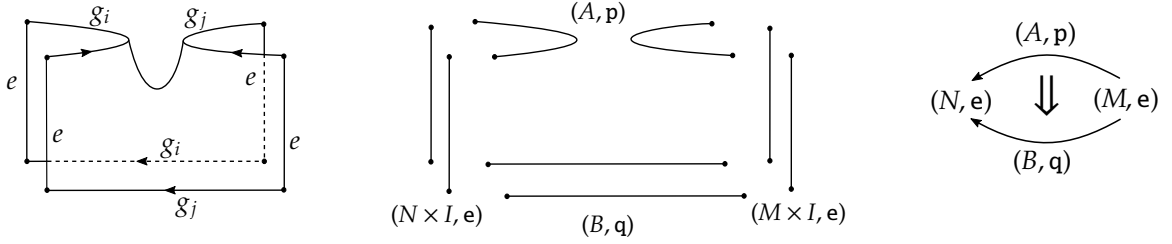
     \caption{An example of decomposition of a $2$-morphism in $\X\Bord_2$}
    \label{fig:Sdecomp}
\end{figure}
\begin{align*}
    (A,\hat{A}_1, \hat{\mathtt{p}}_1) \amalg (B, \hat{B}_1, \hat{\mathtt{q}}_1) \xrightarrow[\cong]{\theta} & (\d_h S , \hat{S}|_{\d_h S},\hat{\mathtt{F}}|_{\d_h S})\\
     (M \times I,\reallywidehat{M \times I^2}, \hat{\mathtt{e}}) \amalg (N \times I , \reallywidehat{N \times  I^2}&, \hat{\mathtt{e}})  \xrightarrow[\cong]{\eta}
     (\d_v S, \hat{S}|_{\d_v S},\hat{\mathtt{F}}|_{\d_v S})
\end{align*}
where $(A,\hat{A}_1,\hat{\mathtt{p}})$ is co-oriented by an inward pointing normal vector and $(B,\hat{B}_1,\hat{\mathtt{q}})$ is co-oriented by an outward pointing normal vector. The X-halations of $M \times I$ and $N \times I$ are induced by their embeddings into $M \times I^2$ and $N \times I^2$ with constant homotopy class $\hat{\mathtt{e}}$. Co-orientations are given by an inward pointing normal vector for $(M \times I , \reallywidehat{M \times I^2}, \hat{\mathtt{e}})$ and an outward pointing normal vector for $(N \times I , \reallywidehat{N \times I^2}, \hat{\mathtt{e}})$. For such an X-haloed $2$-cobordism we have the following notations; $\theta(A)= \d_{h,in}S$, $\theta(B)=\d_{h,out}S$, $\eta(M \times I)= \d_{v,in}S$, and $\eta(N \times I)= \d_{v,out}S$.

Two X-haloed $2$-cobordisms $(S_0,\hat{S}_0,R_0,\hat{\mathtt{F}}_0)$ and $(S_1,\hat{S}_1,R_1,\hat{\mathtt{F}}_1)$ are \textit{isomorphic} if there is an isomorphism of X-halations $\xi : (S_0, \hat{S}_0, \hat{\mathtt{F}}_0) \to (S_1, \hat{S}_1,\hat{\mathtt{F}}_1)$ which restricts isomorphisms  
\begin{align*}
    (\d_{h,in}S_0 ,(\hat{S}_0)|_{\d_{h,in}S_0}, (\hat{\mathtt{F}}_0)|_{\d_{h,in}S_0}) &\to (\d_{h,in}S_1, (\hat{S}_1)|_{\d_{h,in} S_1} ,(\hat{\mathtt{F}}_1)|_{\d_{h,in}S_1}) \\
    (\d_{h,out}S_0 ,(\hat{S}_0)|_{\d_{h,out}S_0}, (\hat{\mathtt{F}}_0)|_{\d_{h,out}S_0}) &\to (\d_{h,out}S_1, (\hat{S}_1)|_{\d_{h,out} S_1}, (\hat{\mathtt{F}}_1)|_{\d_{h,out}S_1} )\\
    (\d_{v,in}S_0 ,(\hat{S}_0)|_{\d_{v,in}S_0}, (\hat{\mathtt{F}}_0)|_{\d_{v,in}S_0}) &\to (\d_{v,in}S_1, (\hat{S}_1)|_{\d_{v,in} S_1} ,(\hat{\mathtt{F}}_1)|_{\d_{v,in}S_1} )\\
    (\d_{v,out}S_0 ,(\hat{S}_0)|_{\d_{v,out}S_0}, (\hat{\mathtt{F}}_0)|_{\d_{v,out}S_0}) &\to (\d_{v,out}S_1, (\hat{S}_1)|_{\d_{v,out} S_1}, (\hat{\mathtt{F}}_1)|_{\d_{v,out}S_1}) 
\end{align*}
such that $\xi|_{\d S_0}$ is identity, $\xi \circ \eta = \eta'$ and $\xi \circ \theta = \theta'$ where $\theta'$ and $\eta'$ are isomorphisms of co-oriented X-halations corresponding to the decomposition $\d S_1 = \d_h S_1 \amalg \d_v S_1$.
\end{enumerate}
\end{definition}

\begin{lemma}
\label{thm:bicat}
The bicategory $\X\Bord_2$ is a symmetric monoidal bicategory under disjoint union.
\end{lemma}
We skip the proof which is given in \cite{tezim} using a method developed by Shulman \cite{shulman}. Recall that two major goals of this paper are to define $2$-dimensional extended homotopy field theories and classify them. The following definition addresses to the first one. 
\begin{definition}
Let $\mathcal{C}$ be a symmetric monoidal bicategory. A \textit{$\mathcal{C}$-valued $2$-dimensional extended homotopy field theory with target $X$} is a symmetric monoidal $2$-functor from $\X\Bord_2$ to $\mathcal{C}$.
\end{definition}

\section{The \textit{G}-planar Decompositions}\label{G_planar_section}
\subsection{\textit{G}-linear diagrams}
Linear diagrams, introduced by Schommer-Pries~\cite{schommer}, represent $1$-dimensional compact manifolds equipped with a Morse function to $[0,1]$. Briefly speaking, a linear diagram is a triple formed by the set of critical values of a Morse function on a compact $1$-manifold, an open cover of $[0,1]$, and combinatorial data describing preimages of a Morse function on open sets. By labeling critical values with cup or cap instead of their indices (see Figure \ref{fig:cupcap}) the first ingredient of a linear diagram is defined as follows.
\begin{definition}
A \textit{$1$-dimensional graphic} $\Psi$ (\cite{schommer}) is a finite subset of $(0,1)$ where each point is labeled with either cup or cap.
\end{definition}
%\begin{figure}[ht]
%    \centering
%    \includesvg{1dsing}
%    \caption{Singularities of a Morse function on a $1$-manifold and their images in $\R$}
%    \label{fig:cupcap}
%\end{figure}
\begin{figure}[ht]
    \centering
    \def\svgwidth{\columnwidth}
   %% Creator: Inkscape 1.0 (4035a4fb49, 2020-05-01), www.inkscape.org
%% Accompanies image file '1dsing_1.pdf' (pdf, eps, ps)
\begingroup%
  \makeatletter%
  \providecommand\color[2][]{%
    \errmessage{(Inkscape) Color is used for the text in Inkscape, but the package 'color.sty' is not loaded}%
    \renewcommand\color[2][]{}%
  }%
  \providecommand\transparent[1]{%
    \errmessage{(Inkscape) Transparency is used (non-zero) for the text in Inkscape, but the package 'transparent.sty' is not loaded}%
    \renewcommand\transparent[1]{}%
  }%
  \providecommand\rotatebox[2]{#2}%
  \newcommand*\fsize{\dimexpr\f@size pt\relax}%
  \newcommand*\lineheight[1]{\fontsize{\fsize}{#1\fsize}\selectfont}%
  \ifx\svgwidth\undefined%
    \setlength{\unitlength}{440.85766486bp}%
    \ifx\svgscale\undefined%
      \relax%
    \else%
      \setlength{\unitlength}{\unitlength * \real{\svgscale}}%
    \fi%
  \else%
    \setlength{\unitlength}{\svgwidth}%
  \fi%
  \global\let\svgwidth\undefined%
  \global\let\svgscale\undefined%
  \makeatother%
  \begin{picture}(1,0.17336333)%
    \lineheight{1}%
    \setlength\tabcolsep{0pt}%
    \put(0,0){\includegraphics[width=\unitlength,page=1]{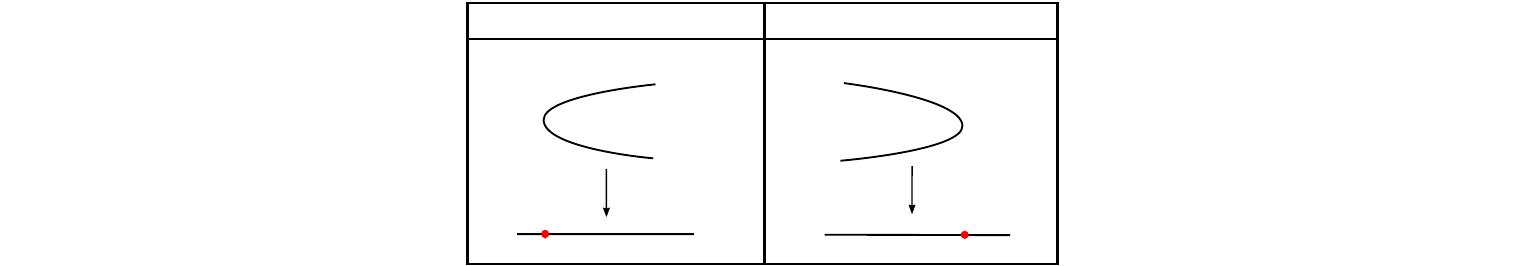}}%
    \put(0.37808701,0.15523729){\color[rgb]{0,0,0}\makebox(0,0)[lt]{\lineheight{1.25}\smash{\begin{tabular}[t]{l}\scriptsize$\mathsf{cup}$\end{tabular}}}}%
    \put(0.57715512,0.15523731){\color[rgb]{0,0,0}\makebox(0,0)[lt]{\lineheight{1.25}\smash{\begin{tabular}[t]{l}\scriptsize$\mathsf{cap}$\end{tabular}}}}%
    \put(0,0){\includegraphics[width=\unitlength,page=2]{1dsing_1.pdf}}%
  \end{picture}%
\endgroup%

     \caption{Singularities of a Morse function on a $1$-manifold and their images in $\R$}
    \label{fig:cupcap}
\end{figure}
For a given $1$-dimensional graphic $\Psi$, an open cover $\mathcal{U}=\{U_{\alpha}\}_{\alpha \in J}$ of $[0,1]$ having at most double intersections is said to be \textit{$\Psi$-compatible} if each $U_{\alpha}$ contains at most one element from $\Psi$ and double intersections are disjoint from $\Psi$. It is not hard to find such open covers and the second ingredient of a linear diagram is defined as follows. 
\begin{definition}(\cite{schommer})
Let $\Psi$ be a $1$-dimensional graphic. A \textit{chambering set} $\Gamma$ for $\Psi$ is a set of isolated points in $(0,1)$ disjoint from $\mu$. \textit{Chambers} of $\Gamma$ are the connected components of $[0,1]  \backslash (\Gamma \cup \mu )$. A chambering set $\Gamma$ is said to be subordinate to an open cover $\mathcal{U}=\{U_{\alpha}\}_{\alpha \in J}$ of $[0,1]$ if each chamber is a subset of at least one $U_{\alpha}$.
\end{definition}
\begin{example}\label{1dex}
Figure \ref{fig:1dimex} shows an example of a $1$-dimensional oriented compact manifold $M$ equipped with a Morse function $f: (M,\d M) \to ([0,1],\{0,1\})$. Critical values of $f$ defines a $1$-dimensional graphic $\Psi$ (see Figure \ref{fig:cupcap}). An open cover $\mathcal{U} =\{U_i\}_{i=1}^4$ of $[0,1]$ is a $\Psi$-compatible open cover and turquoise points form a chambering set subordinate to $\mathcal{U}$.
\end{example}
%\begin{figure}[h!]
%    \centering
%    \includesvg{pdexample3s}
%    \caption{Induced $1$-dimensional graphic and a chambering set}
%    \label{fig:1dimex}
%\end{figure}
\begin{figure}[h!]
    \centering
    \def\svgwidth{\columnwidth}
    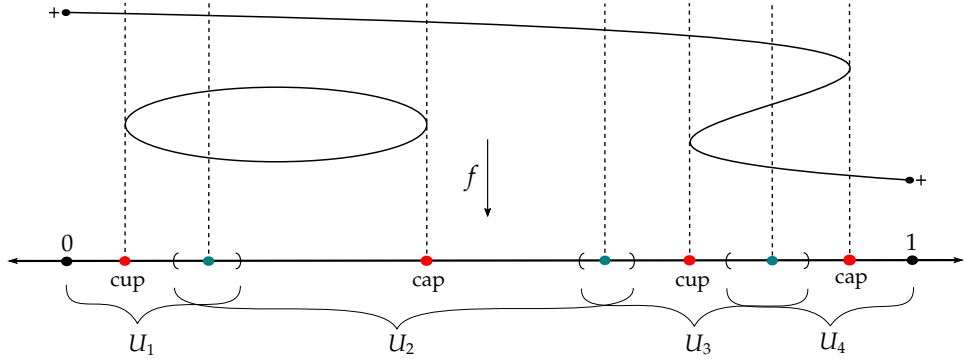
    \caption{Induced $1$-dimensional graphic and a chambering set}
    \label{fig:1dimex}
\end{figure}
For an oriented compact $1$-manifold $M$, we call a Morse function of the form $f :(M,\d M) \to ([0,1],\{ 0,1\})$ whose critical values are distinct and lie in $(0,1)$ a \textit{generic map}. Let $\Psi$ be a $1$-dimensional graphic induced from a pair $(M,f)$ of an oriented compact $1$-manifold equipped with a generic map. Let $\Gamma$ be a chambering set subordinate to a $\Psi$-compatible open cover $\mathcal{U}$. Since $f$ is a Morse function and chambers are disjoint from $\mu$, the preimage $f^{-1}(V)$ of a chamber $V$ consists of disjoint union of arcs (possibly empty) each mapping diffeomorphically onto $V$ under $f$. A \textit{trivialization} of $V$ is an identification of $f^{-1}(V)$ with $\mathbb{N}_{\leq N}\times V$ for some $N \in \N$ where $\mathbb{N}_{\leq N}= \{ a \in \N \ | \ 0 < a \leq N\}$ if $f^{-1}(V)$ is nonempty and identification with empty set otherwise. In this case, each $\{i\} \times V$ is called a \textit{sheet} and each sheet is equipped with an orientation.

Trivializations of two neighboring chambers have the same number of sheets if chambers are separated by a point in $\Gamma$. If a point in $\mu$ separates chambers, then by the Morse lemma the number of sheets differ by two (see Figure \ref{fig:cupcap}). A \textit{sheet data} $\mathcal{S}$ (\cite{schommer}) for a pair $(\Psi,\Gamma)$ consists of a trivialization of each chamber and an injection or a permutation between trivializations of neighboring chambers preserving orientations and describing how sheets are glued.
\begin{definition}(Definition 3.45, \cite{schommer})
A \textit{linear diagram} is a triple $(\Psi, \Gamma, \mathcal{S})$ consisting of a $1$-dimensional graphic $\Psi$, a chambering set $\Gamma$ subordinate to a $\Psi$-compatible open cover $\mathcal{U}= \{U_{\alpha}\}_{\alpha \in J}$ of $[0,1]$, and a sheet data $\mathcal{S}$ associated to the pair $(\Psi,\Gamma)$.
\end{definition}
Any linear diagram yields an oriented compact $1$-manifold and a generic map to $[0,1]$. Our goal is to add extra data of X-manifolds to linear diagrams so that these diagrams produce oriented $1$-dimensional marked X-manifolds. Recall that a $1$-dimensional marked X-manifold is a triple $(M,T,\mathtt{g})$ where $M$ is an oriented compact $1$-manifold, $T \subseteq M$ is a finite set with $\d M \subset T$ and each closed connected component of $M$ contains at least 2 elements of $T$, and $\mathtt{g} \in [(M,T),(X,x)]$. 

We describe the extra data on linear diagrams on an example. Let $(M,T,\mathtt{g})$ be a $1$-dimensional marked X-manifold, shown in Figure \ref{fig:1d_marked_mnfd}, whose underlying manifold $M$ is the $1$-dimensional oriented compact manifold considered in Example \ref{1dex}. We consider the same generic map $f: (M,\d M) \to (I,\d I)$ and chambering set in Example \ref{1dex} giving the linear diagram $(\Psi,\Gamma,\mathcal{S})$. 
%\begin{figure}[ht]
%    \centering
%    \includesvg{pdexample7s}
%    \caption{$1$-dimensional marked X-manifold $(M,T,\mathtt{g})$}
%    \label{fig:1d_marked_mnfd}
%\end{figure}
\begin{figure}[ht]
    \centering
    \def\svgwidth{\columnwidth}
    %% Creator: Inkscape 1.0 (4035a4fb49, 2020-05-01), www.inkscape.org
%% Accompanies image file 'pdexample7s_1.pdf' (pdf, eps, ps)
\begingroup%
  \makeatletter%
  \providecommand\color[2][]{%
    \errmessage{(Inkscape) Color is used for the text in Inkscape, but the package 'color.sty' is not loaded}%
    \renewcommand\color[2][]{}%
  }%
  \providecommand\transparent[1]{%
    \errmessage{(Inkscape) Transparency is used (non-zero) for the text in Inkscape, but the package 'transparent.sty' is not loaded}%
    \renewcommand\transparent[1]{}%
  }%
  \providecommand\rotatebox[2]{#2}%
  \newcommand*\fsize{\dimexpr\f@size pt\relax}%
  \newcommand*\lineheight[1]{\fontsize{\fsize}{#1\fsize}\selectfont}%
  \ifx\svgwidth\undefined%
    \setlength{\unitlength}{476.18504802bp}%
    \ifx\svgscale\undefined%
      \relax%
    \else%
      \setlength{\unitlength}{\unitlength * \real{\svgscale}}%
    \fi%
  \else%
    \setlength{\unitlength}{\svgwidth}%
  \fi%
  \global\let\svgwidth\undefined%
  \global\let\svgscale\undefined%
  \makeatother%
  \begin{picture}(1,0.18199458)%
    \lineheight{1}%
    \setlength\tabcolsep{0pt}%
    \put(0,0){\includegraphics[width=\unitlength,page=1]{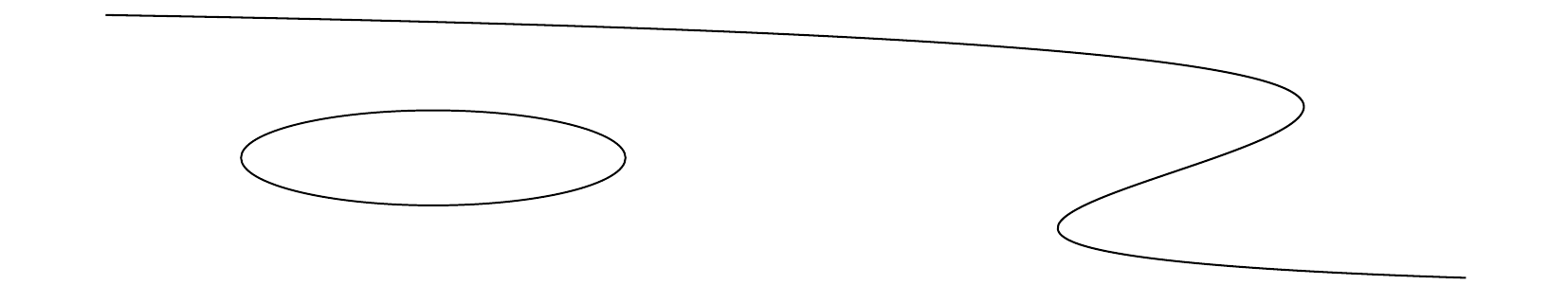}}%
    \put(0.40355507,0.07509603){\color[rgb]{0,0,0}\makebox(0,0)[lt]{\lineheight{1.25}\smash{\begin{tabular}[t]{l}\small$g'$\end{tabular}}}}%
    \put(0.2642576,0.1761845){\color[rgb]{0,0,0}\makebox(0,0)[lt]{\lineheight{1.25}\smash{\begin{tabular}[t]{l}\small$g''$\end{tabular}}}}%
    \put(0,0){\includegraphics[width=\unitlength,page=2]{pdexample7s_1.pdf}}%
    \put(0.83846402,0.1072526){\color[rgb]{0,0,0}\makebox(0,0)[lt]{\lineheight{1.25}\smash{\begin{tabular}[t]{l}\small$g'''$\end{tabular}}}}%
    \put(0,0){\includegraphics[width=\unitlength,page=3]{pdexample7s_1.pdf}}%
    \put(0.23013513,0.03691182){\color[rgb]{0,0,0}\makebox(0,0)[lt]{\lineheight{1.25}\smash{\begin{tabular}[t]{l}\small$e$\end{tabular}}}}%
    \put(0,0){\includegraphics[width=\unitlength,page=4]{pdexample7s_1.pdf}}%
  \end{picture}%
\endgroup%

    \caption{$1$-dimensional marked X-manifold $(M,T,\mathtt{g})$}
    \label{fig:1d_marked_mnfd}
\end{figure}
First we label elements of $\Gamma$. A point in $\Gamma$ is labeled with $\beta^{\sigma}$ where $\sigma \in S_N$ is the permutation coming from the sheet data of this point. We then add the elements of $f(T)$ to $\Gamma$, which means there are possibly new chambers. Each new chamber has the induced trivialization from the larger chamber which splits into two. We do not label these added points. After that we equip the boundary components of every sheet, except the critical points of $f$, with oriented points using the orientation of $M$ (brown points in Figure \ref{fig:modifpdex1}) and label each sheet with a group element using the characteristic map $\mathtt{g}$ as shown in Figure \ref{fig:modifpdex1}.    
%\begin{figure}[ht]
%    \centering
%    \includesvg{pdexample_6s_tez}
%    \caption{Example of $G$-linear diagram without sheet data}
%    \label{fig:modifpdex1}
%\end{figure}
\begin{figure}[ht]
    \centering
    \def\svgwidth{\columnwidth}
    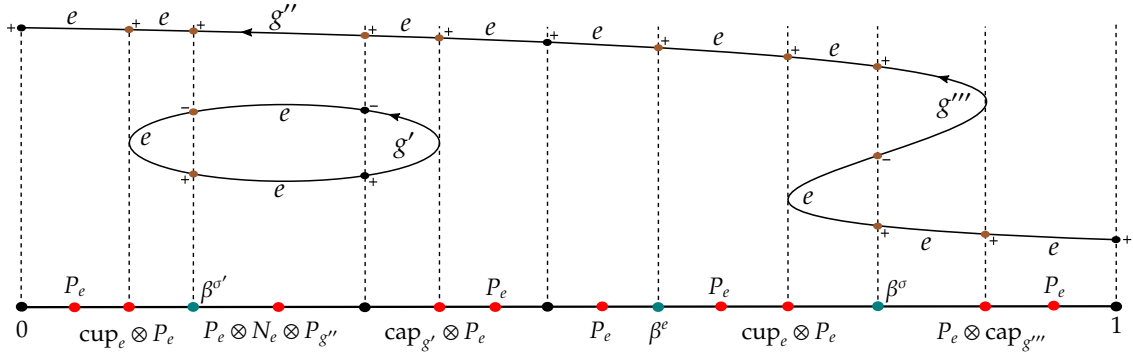
    \caption{Example of $G$-linear diagram without sheet data}
    \label{fig:modifpdex1}
\end{figure}

Next, we add labeled points to $[0,1]$ as follows. If the preimage of a chamber does not have any singularity then the midpoint of that chamber is added. The label of this point is given as follows. We assign $P_{g_1}$ for such a sheet with + boundary points and $g_1$-label, and assign $N_{g_2}$ for a such a sheet with - boundary points and $g_2$-label. Then the label of the added point is given by $A_1 \otimes A_2 \otimes \dots \otimes A_n$ where $A_i$ is the assigned label of the $i$-th sheet according to the trivialization of the chamber for $i=1,\dots,n$. Similarly, we modify labels $\text{cup}$ and $\text{cap}$ according to assignments to sheets which are in the same trivialization with these singularities (see Figure \ref{fig:modifpdex1}). Lastly, on sheet data trivialization of sheets involves labeling each sheet with a group element with an arrow as described above and each (non-labeled) point is lifted to boundary of a sheet. We denote this modified linear diagram using the extra data of the $1$-dimensional marked X-manifold with $(\Psi^G, \Gamma^G, \mathcal{S}^G)$ and call a \textit{$G$-linear diagram}. 

It is clear that any $G$-linear diagram $(\Psi^G, \Gamma^G, \mathcal{S}^G)$ produces a pair $((M',T',\mathtt{g}'),f')$. For any such pair by choosing a compatible chambering set we obtain a new pair. These two pairs are related by the following notion. An \textit{X-homeomorphism} between (marked) X-manifolds is a pointed orientation preserving diffeomorphism commuting with characteristic maps. An X-homeomorphism between such pairs is called \textit{over $[0,1]$} if it commutes with the fixed generic maps.
\begin{prop} \label{lindecomp}
Let $(\Psi^G, \Gamma^G, \mathcal{S}^G)$ be a $G$-linear diagram induced from a pair $((M,T,\mathtt{g}),f)$ of $1$-dimensional oriented marked X-manifold, a generic map $f: (M,\d M) \to ([0,1],\{0,1\})$, and a chambering set $\Gamma$ for a $\Psi$-compatible open cover $\mathcal{U}$ of $[0,1]$. If the pair $((M',T',\mathtt{g}'),f')$ is constructed from $(\Psi^G,\Gamma^G,\mathcal{S}^G)$, then there exists an X-homeomorphism $F:M \to M'$ over $[0,1]$.
\end{prop}
\begin{proof}
The diffeomorphism $F$ maps inverse images of chambers to corresponding trivializations. Since corresponding connected components have the same $G$-labels and both $f$ and $f' \circ F$ restrict to the same map on $f^{-1}(V)$ for any chamber $V$, $F$ is an X-homeomorphism over $[0,1]$.   
\end{proof}

\subsection{\textit{G}-planar diagrams}
Planar diagrams, introduced by Schommer-Pries~\cite{schommer}, represent cobordism type $\<2\>$-surfaces equipped with a generic map to $I^2=[0,1]\times[0,1]$. Here generic maps refer to Schommer-Pries stratification of jet spaces described below. Parallel to linear diagrams, a planar diagram consists of a graphic of a generic map, an open cover of $I^2$, and a combinatorial data describing preimages of a generic map on open sets. 

In his classification of $2$-dimensional extended TFTs, Schommer-Pries \cite{schommer} studied maps from cobordism type $\<2\>$-surfaces to $I^2$ and refined the Thom-Boardman stratification of jet spaces. Figure \ref{fig:2dimsing} shows the singularities of Schommer-Pries stratification in normal coordinates and their graphics in $I^2$. Here by a \textit{graphic} we mean the image of a singularity under a generic map. In this context by a generic map we mean a map whose jet sections are transversal to each strata. On Figure \ref{fig:2dimsing} generic maps are projections to the page. The numbers on singularity names indicate the indices of a singularity. By an index of a singularity, we mean a symmetry of either a singularity or its graphic. For example, Fold-1 is obtained from Fold-2 by changing the folding direction. Similarly, Cap, Cup, Saddle-1, and Saddle-2 are different indices of the Morse singularity. Observe that cusp singularity has $4$ indices. 

%\begin{figure}[b]
%    \centering
%     \includesvg{2dsingularities}
%     \caption{Singularities of Schommer-Pries stratification and their graphics in $\R^2$}
%    \label{fig:2dimsing}
%\end{figure}
\begin{figure}[b]
    \centering
    \def\svgwidth{\columnwidth}
    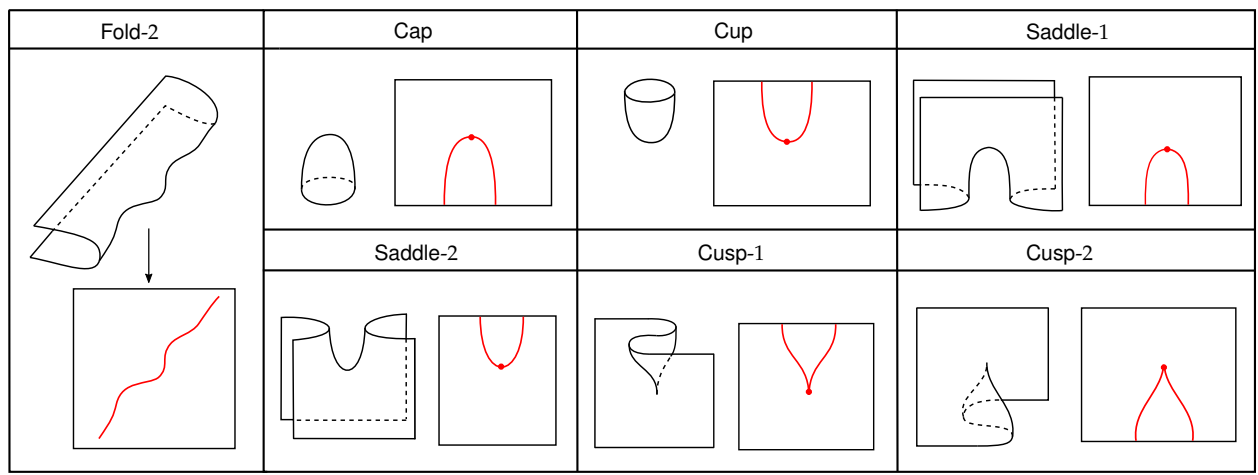
    \caption{Singularities of Schommer-Pries stratification and their graphics in $\R^2$}
    \label{fig:2dimsing}
\end{figure}
For a given cobordism type $\<2\>$-surface $\Sigma$, a generic map for Schommer-Pries stratification has the form $f:(\Sigma, \d_v \Sigma, \d_h \Sigma ) \to (I^2, \d I \times I, I \times \d I)$. Transversality theorems (see \cite{schommer}, \cite{stable-maps}) imply that the set of generic maps is dense in $C^{\infty}((\Sigma, \d_v \Sigma, \d_h \Sigma),(I^2, \d I \times I, I \times \d I))$. The properties of Schommer-Pries stratification are listed in the following definition. In particular, the graphic of a generic map for this stratification is a $2$-dimensional graphic which is defined as follows. 
\begin{definition}[Definition 1.29, \cite{schommer}]\label{2dimgraphic}
A \textit{2-dimensional graphic} $\Phi=(\eta,\mu)$ is a diagram in $I^2$ consisting of a finite number of embedded labeled curves ($\eta$) and a finite number of labeled points ($\mu$) satisfying the following conditions:
\begin{enumerate}[(i)]
    \item Elements of $\eta$ can only have transversal intersections and no three or more elements intersect at a point. Each element of $\eta$ is labeled with either Fold-$1$ or Fold-$2$.
    \item Elements of $\eta$ are disjoint from $\d I \times I$ and intersect transversely with $I \times \d I$. Labeling each of these intersection points on $I \times \d I$ with cup for Fold-$1$ labeled curves and with cap for Fold-$2$ labeled curves produces $1$-dimensional graphics on $I \times \{0\}$ and $I \times \{1\}$. 
    \item Projections of elements of $\eta$ to the last coordinate of $I^2$ are local diffeomorphisms.
    \item Elements of $\mu$ are isolated and disjoint from $\d (I \times I)$. Each element is labeled with one of the Cup, Cap, Saddle-$1$, Saddle-$2$, Cusp-$i$ for $i=1,2,3,4$. 
    \item Each element in $\mu$ has a neighborhood in which two elements of $\eta$ form one of the Cup, Cap, Saddle-$1$, Saddle-$2$, Cusp-$i$ graphic for $i=1,2,3,4$ (see Figure \ref{fig:2dimsing}).
\end{enumerate}
\end{definition}
%\begin{figure}[b]
%    \centering
%    \includesvg{newdiagrams2d}
%    \caption{Graphics of singularities with X-manifold data}
%    \label{fig:newdiagrams}
%\end{figure}
\begin{figure}[b]
    \centering
    \def\svgwidth{\columnwidth}
    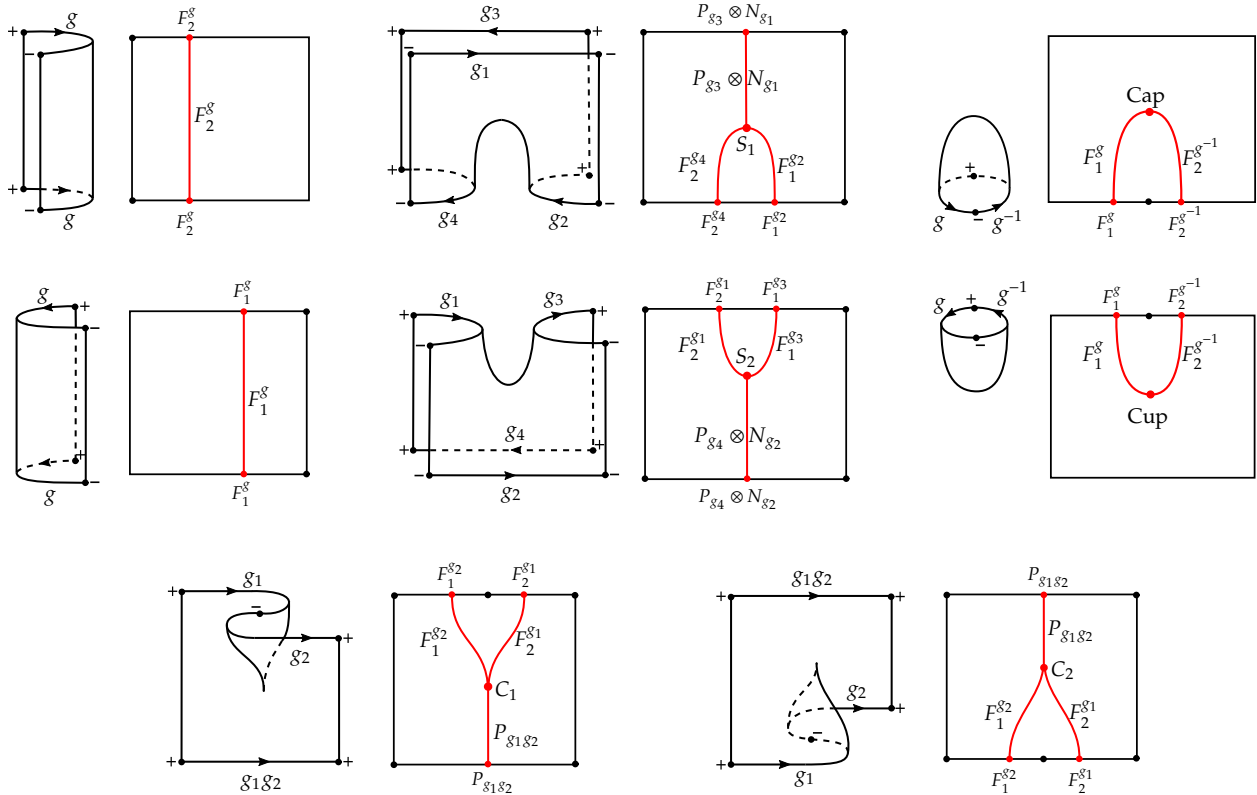
    \caption{Graphics of singularities with X-manifold data}
    \label{fig:newdiagrams}
\end{figure}
We want to extend this definition to cobordism type $\<2\>$-X-surfaces so that a $2$-dimensional graphic additionally contains the X-manifold data. First we consider $\<2\>$-X-surfaces whose underlying manifolds are singularities of Schommer-Pries stratification in normal coordinates (see Figure \ref{fig:2dimsing}). Figure \ref{fig:newdiagrams} shows their graphics with the X-manifold data. Note that we abbreviate Fold-i label to $F_i$, Saddle-i to $S_i$, and Cusp-i to $C_i$ for $i=1,2$. Also observe that Fold-1 and Fold-2 singularities are paths of cup and cap singularities in the previous section. For this reason, henceforth, on any $G$-linear diagram we replace cup and cap labels with $F_1$ and $F_2$ labels respectively. This implies that the restriction of each diagram in Figure \ref{fig:newdiagrams} to $I \times \d I$ yields two partial $G$-linear diagrams. Later we complete them to $G$-linear diagrams by adding chambering sets and sheet data.   

Compared to graphics of singularities in Figure \ref{fig:2dimsing}, there are additional arcs connecting graphics of Morse\footnote{Similar to the graphics of saddles, one can add arcs to the graphics of Cup and Cap and label them with $\emptyset$.} and cusp singularities to the (red) points of the boundary $G$-linear diagram. The reason behind the addition of these arcs is the connection between these diagrams and string diagrams, which is the content of Theorem \ref{computadic}. We call $\<2\>$-X-surfaces in Figure \ref{fig:newdiagrams} \textit{elementary $\<2\>$-X-surfaces} since they are building blocks of cobordism type $\<2\>$-X-surfaces under horizontal and vertical gluing operations. However, this list is not complete. The rest of the elementary $\<2\>$-X-surfaces and their graphics are given in Figure \ref{fig:newdigrams2}. 
%\begin{figure}[h!]
%    \centering
%    \includesvg{newdiagrams2d_2}
%    \caption{Remaining elementary $\<2\>$-X-surfaces and their graphics}
%    \label{fig:newdigrams2}
%\end{figure}
\begin{figure}[h!]
    \centering
    \def\svgwidth{\columnwidth}
    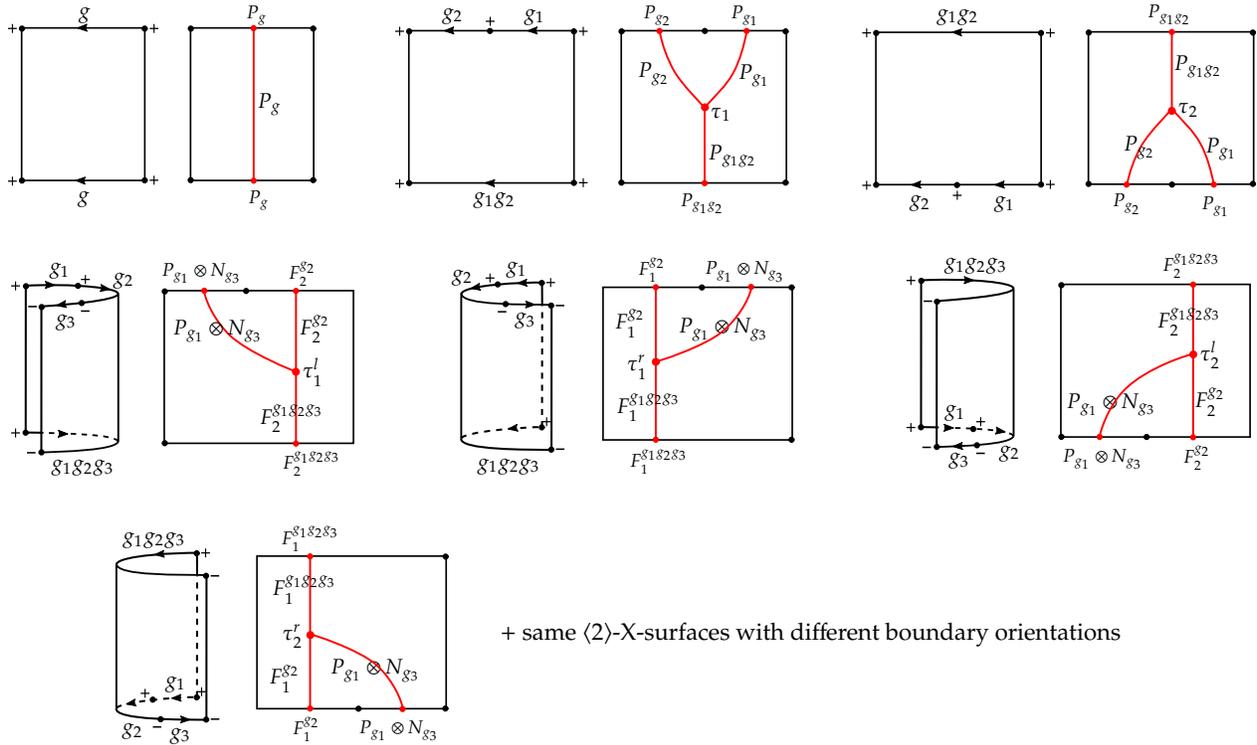
    \caption{Remaining elementary $\<2\>$-X-surfaces and their graphics}
    \label{fig:newdigrams2}
\end{figure}

We now know the extra data on $2$-dimensional graphics of elementary $\<2\>$-X-surfaces. Using these modified diagrams, for any generic map on a cobordism type $\<2\>$-X-surface $(\Sigma,R,\mathtt{P})$ we add the X-manifold data $(R,\mathtt{P})$ to the graphic of the generic map in two steps. First we decompose $(\Sigma,R,\mathtt{P})$ into horizontal and vertical compositions of elementary $\<2\>$-X-surfaces. This is always possible by the nature of Schommer-Pries stratification and above arguments. Using $\mathtt{P}$ we choose $G$-labels on each elementary $\<2\>$-X-surface in the decomposition. We then consider the modified diagrams of these $\<2\>$-X-surfaces in $I^2$ as described above. Figure \ref{fig:newdiagramex} shows an example of this process where the generic map is projection to the page. For a given $2$-dimensional graphic $\Phi=(\eta,\mu)$, we denote the union of $\eta$ and arcs encoding the X-manifold data with $\eta^G$ and similarly $\mu^G$ denotes the union of $\mu$ and additional labeled points. We call such a $2$-dimensional graphic $\Phi$ equipped with X-manifold data a \textit{$2$-dimensional $G$-graphic} and denote it with $\Phi^G=(\eta^G,\mu^G)$. 

Let $\Phi^G=(\eta^G,\mu^G)$ be a $2$-dimensional $G$-graphic, an open cover $\mathcal{U}=\{U_{\alpha}\}_{\alpha \in J}$ of $I^2$ with at most triple intersections is said to be \textit{$\Phi$-compatible} (\cite{schommer}) if each triple intersection is disjoint from $\mu$, each double intersection is disjoint from $\eta \cup \mu$ or contains a single element from $\eta$, and the open covers $\{U_{\alpha} \cap (I \times \{i\})\}_{\alpha \in J}$ of $I \times \{i\}$ for $i=0,1$ are compatible with corresponding $1$-dimensional graphics obtained from $\Phi^G$. Knowing the fact that $\R^2$ has covering dimension two and sets $\eta$ and $\mu$ are finite it is not hard to find $\Phi$-compatible open covers for a given graphic $\Phi$. 
\theoremstyle{definition}
\begin{definition}[Definition 1.42, \cite{schommer}]
Let $\Phi^G=(\eta^G,\mu^G)$ be a $2$-dimensional $G$-graphic. A \textit{chambering graph} $\Gamma$ for $\Phi^G$ is a smoothly embedded graph in $I^2$ satisfying the following conditions. Vertices of $\Gamma$ are disjoint from elements of $\Phi^G$ and have degree either one or three. Edges of $\Gamma$ are disjoint from $\d I \times I$ and transverse to $\Phi^G$ and $I \times \d I$. Furthermore, projection of each edge to the last coordinate is a local diffeomorphism and around each trivalent vertex one of the edges projects to the opposite side of the projection of other two edges with respect to the image of the vertex.
\end{definition}
\begin{definition}
Let $\Gamma$ be a chambering graph for $\Phi^G=(\eta^G,\mu^G)$. \textit{Chambers} of $\Gamma$ are the connected components of $I^2 \backslash(\Gamma \cup \eta \cup \mu)$. A chambering graph $\Gamma$ is said to be subordinate to an open cover $\mathcal{U}=\{U_{\alpha}\}_{\alpha \in J}$ of $I^2$ if each chamber is a subset of at least one $U_{\alpha}$ with $\alpha \in J$ and the chambering sets $\Gamma \cap (I \times \{i\})$ are compatible with the restricted open covers $\{U_{\alpha} \cap (I \times \{i\})\}_{\alpha \in J}$ for $i=0,1$.
\end{definition}
\begin{example}
Figure \ref{fig:newdiagramex} shows an example of a chambering graph $\Gamma$ where each colored region is a chamber. Note that the chambering graph leads to new red points on $I \times \d I$ forming two partial $G$-linear graphs. In this example all new points are labeled with $P_e \otimes N_e$.
\end{example}
%\begin{figure}[h!]
%    \centering
%    \includesvg{Gexample4}
%    \caption{Adding X-manifold data to a graphic and an example of a chambering graph}
%    \label{fig:newdiagramex}
%\end{figure}
\begin{figure}[h!]
    \centering
    \def\svgwidth{\columnwidth}
    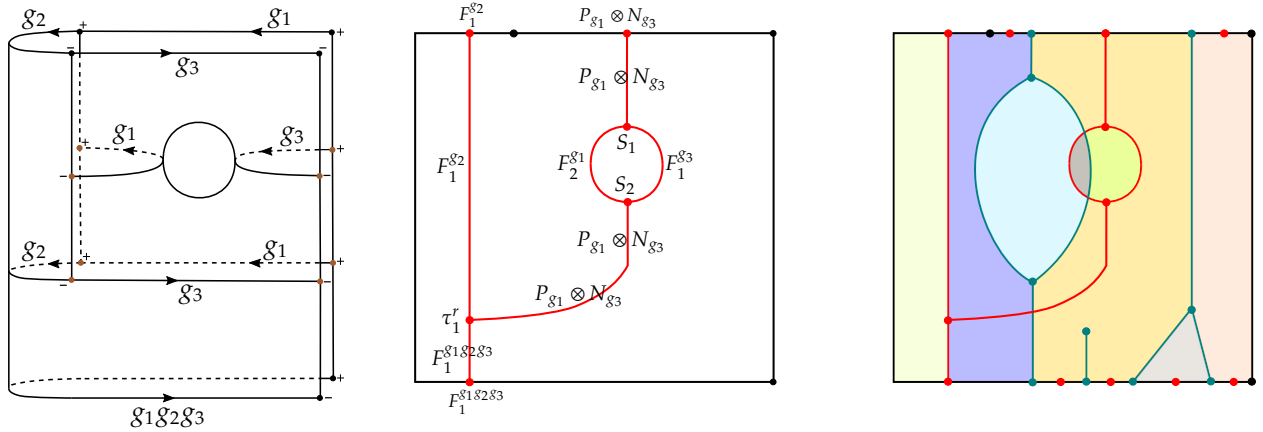
    \caption{Adding X-manifold data to a graphic and an example of a chambering graph}
    \label{fig:newdiagramex}
\end{figure}
\begin{prop}
Let $\Phi^G$ be a $2$-dimensional $G$-graphic in $I^2$ and let $\mathcal{U}= \{U_{\alpha}\}_{\alpha \in J}$ be a $\Phi$-compatible open cover of $I^2$. Then there exists a chambering graph $\Gamma$ for $\Phi^G$ subordinate to $\mathcal{U}$.
\end{prop}
The $2$-dimensional graphic version of this proposition was proven in \cite{schommer} (see Proposition 1.46). This version follows from that only using transversality arguments. From now on we assume that all chambering graphs are subordinate to some compatible open cover.  

Next, we recall sheet data associated to a pair $(\Phi,\Gamma)$. We know sheet data on the components of $I \times \d I$ from the previous section. For the other boundary component $\d I \times I$, sheet data is similar and indeed simpler since vertical boundary components are all identical. Therefore, we consider the open subsets of chambers by removing boundary components. That is, for any chamber $U_{\beta}$ which intersects with $\d (I \times I)$ we consider $U_{\beta}'=U_{\beta}- (U_{\beta} \cap \d (I \times I))$. Since $f$ is generic the preimage $f^{-1}(U_{\beta}')$ consists of disjoint union of open sets (possibly empty) each mapping diffeomorphically onto $U_{\beta}'$. A \textit{trivialization} of $U_{\beta}'$ is an identification of $f^{-1}(U_{\beta}')$ with $\mathbb{N}_{\leq N} \times U_{\beta}'$ for some $N \in \N$ if $f^{-1}(U_{\beta}')$ is nonempty and identification with the empty set otherwise. In this case, each $\{i\} \times U_{\beta}'$ is called a \textit{sheet} and each sheet is oriented. By requiring the same trivializations on $U_{\beta}'$ and $\d U_{\beta}$ we extend identifications to $\mathbb{N}_{\leq N} \times U_{\beta}$.

Similar to $1$-dimensional case, trivializations of two neighboring chambers have the same number of sheets if chambers are separated by an edge of $\Gamma$ (see Figure \ref{fig:sheetdata}). If an element in $\eta$ separates chambers then the number of sheets differ by two because it is a fold graphic (see Figure \ref{fig:2dimsing}). A \textit{sheet data} $\mathcal{S}$ (\cite{schommer}) for a pair $(\Phi,\Gamma)$ consists of a trivialization of each chamber and an injection or a permutation between trivializations of neighboring chambers preserving orientations and describing how sheets are glued (see Figure \ref{fig:sheetdata}). Gluing description of sheets requires certain conditions on permutations and injections. For example, if three chambers are separated by edges of a trivalent vertex of $\Gamma$ then the circular composition of permutations must be identity. Also, the permutation corresponding to the edge of a univalent vertex must be the trivial permutation.
\begin{figure}[ht]
\centering
\begin{minipage}[t]{.52 \textwidth}
 \centering
  \def\svgwidth{\columnwidth}
    \scalebox{0.6}{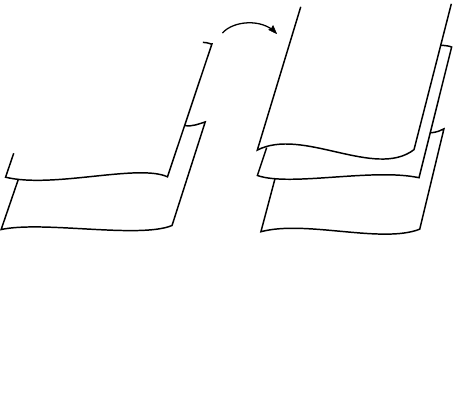}
  % \includesvg[width=.6\linewidth]{sheetdata1234}
  \captionof{figure}{}
  \label{fig:sheetdata}
\end{minipage}%
  \begin{minipage}[t]{.53\textwidth}
  \centering
   \def\svgwidth{\columnwidth}
    \scalebox{0.85}{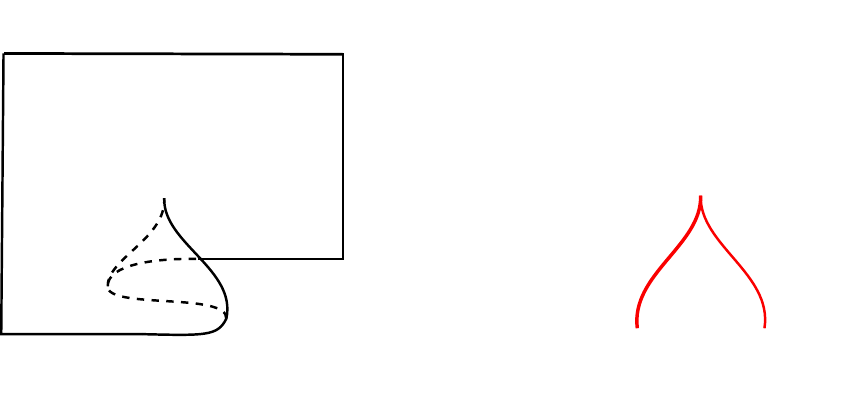}
  %\includesvg[width=.7\linewidth]{cuspsheet_buyuk}
  \captionof{figure}{}
  \label{fig:cuspsheet}
\end{minipage}
\end{figure}
The only nontrivial sheet data is that of a cusp graphic, which we briefly describe. Consider the Cusp-$2$ labeled point in Figure \ref{fig:cuspsheet}. Let $\mathbb{N}_{\leq N+3} \times U_{\beta_1}$ and $\mathbb{N}_{\leq N+1} \times U_{\beta_2}$ be the trivializations such that sheets $\bigcup_{i=N+1}^{N+3} i \times U_{\beta_1}$ and $(N+1) \times U_{\beta_2}$ belong to cusp singularity as shown in Figure \ref{fig:cuspsheet}. In this case, restriction of injections to cusp singularity gives $\sigma_1(N+1)=N+1$ and $\sigma_2(N+1)=N+3$.

\theoremstyle{definition}
\begin{definition}[Definition 1.48, \cite{schommer}]
A \textit{planar diagram} is a triple $(\Phi,\Gamma, \mathcal{S})$ consisting of a $2$-dimensional graphic $\Phi$, a chambering graph $\Gamma$ for $\Phi$ subordinate to a $\Phi$-compatible open cover $\mathcal{U}= \{U_{\alpha}\}_{\alpha \in J}$ of $I^2$, and a sheet data $\mathcal{S}$ associated to the pair $(\Phi,\Gamma)$. 
\end{definition}
Any planar diagram $(\Phi, \Gamma,\mathcal{S})$ produces a cobordism type $\<2\>$-surface $\Sigma$ with a generic map $f :(\Sigma,\d_v \Sigma, \d_h \Sigma) \to (I^2,\d I \times I, I \times \d I)$. In the case of a tuple $(\Phi^G, \Gamma)$, the associated sheet data can be improved to produce a $\<2\>$-X-surface $(\Sigma,R,\mathtt{P})$. We call such sheet data carrying X-manifold data to sheets \textit{$G$-sheet data} and denote it with $\mathcal{S}^G$. Then generalizing a planar diagram, we define a \textit{$G$-planar diagram} as a triple $(\Phi^G,\Gamma, \mathcal{S}^G)$. As an extension of $G$-linear diagrams we label edges of a chambering graph with $\beta^{\sigma}$ where $\sigma$ is the permutation coming from sheet data. If both sheets are trivialized by the empty set, then the separating edge is labeled with $\beta^{\imath}$. We also label vertices of chambering graph as follows. For a fixed trivalent vertex, if two of the edges direct upwards the vertex is labeled with $X^{\sigma,\sigma'}$ and if these two of the edges downward then it is labeled with $(X^{\sigma,\sigma'})^{-1}$. Here $\sigma$ and $\sigma'$ are the permutations corresponding to the sheet data of these edges. A univalent vertex is labeled with $X^e$ if its edge directs upward and it is labeled with $(X^e)^{-1}$ if its edge directs downward. We also label intersection of edges of chambering graph and $2$-dimensional $G$-graphic. For such an intersection, if an edge of chambering graph is labeled with $\beta^{\sigma}$ and an arc of $2$-dimensional $G$-graphic is labeled with $A$ then the intersection is labeled with $\beta^{\sigma}_{A}$.     
%\begin{figure}[ht]
%    \centering
%    \includesvg{unbiasedex5_tez}
%    \caption{Example of a cobordism type $\<2\>$-X-surface and its $G$-planar diagram}
%    \label{fig:stringdiagex}
%\end{figure}
\begin{figure}[ht]
    \centering
    \def\svgwidth{\columnwidth}
    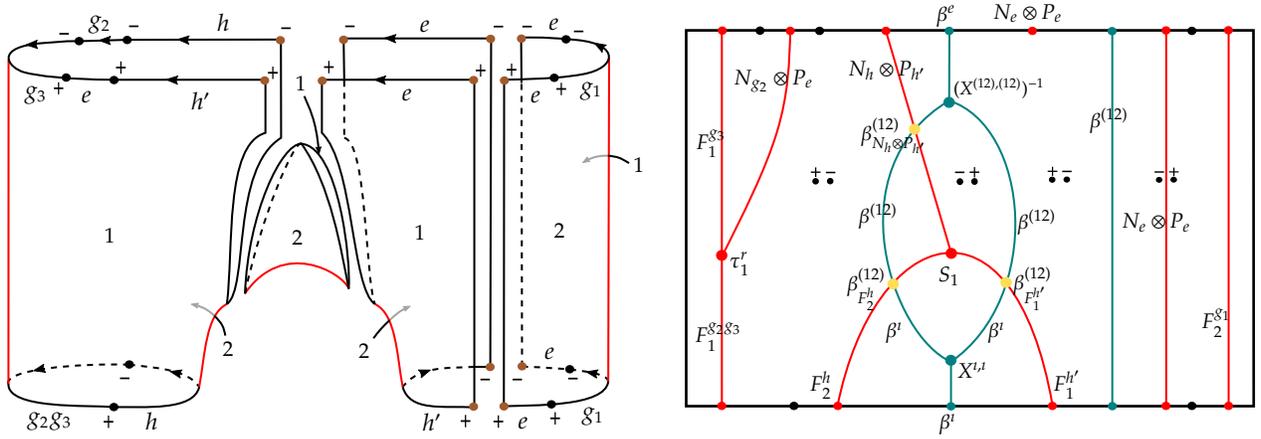
    \caption{Example of a cobordism type $\<2\>$-X-surface and its $G$-planar diagram}
    \label{fig:stringdiagex}
\end{figure}
\begin{example}
Figure \ref{fig:stringdiagex} shows an example of a cobordism type $\<2\>$-X-surface $(\Sigma, R,\mathtt{P})$ and its $G$-planar diagram with respect to the projection map. We denote the trivializations of sheets with numbers on $\Sigma$. Correspondingly, we encode this data to the $G$-planar diagram by labeling chambers with ordered signed points. Here signs come from the sign of the points on the corner of the corresponding sheet. 
\end{example}
Let $(\Sigma,R,\mathtt{P})$ and $(\Sigma',R',\mathtt{P}')$ be $\<2\>$-X-surfaces endowed with generic maps $f$ and $f'$ respectively. An X-homeomorphism $F: (\Sigma,R,\mathtt{P}) \to (\Sigma',R',\mathtt{P}')$ is said to be \textit{over $I^2$} if it commutes with the fixed generic maps, i.e. $f' \circ F =f$.
\begin{prop}\label{lem2}
Let $f:(\Sigma,\d_v \Sigma,\d_h \Sigma) \to (I^2, \d I \times I,I \times \d I)$ be a generic map on a cobordism type $\<2\>$-X-surface $(\Sigma,R,\mathtt{P})$ inducing a $2$-dimensional graphic $\Phi^G$. Let $\Gamma$ be a chambering graph for $\Phi^G$ subordinate to a $\Phi$-compatible open cover giving a $G$-planar diagram $(\Phi^G,\Gamma, \mathcal{S}^G)$. If the pair $((\Sigma',R',\mathtt{P}'),f')$ is constructed from $(\Phi^G,\Gamma, \mathcal{S}^G)$ then there exists an X-homeomorphism $F: \Sigma \to \Sigma'$ over $I^2$.
\end{prop}
\begin{proof}
The diffeomorphism $F:\Sigma \to \Sigma'$ maps inverse images of chambers to corresponding trivializations. Since $F(R)=R'$, $[\mathtt{P}' \circ F]= \mathtt{P}$, and both $f$ and $f' \circ F$ restrict to the same map on $f^{-1}(U_{\beta})$ for any chamber $U_{\beta}$, $F$ is an X-homeomorphism over $I^2$.
\end{proof}

\subsection{\textit{G}-spatial diagrams}
Schommer-Pries \cite{schommer} introduced spatial diagrams to identify planar diagrams which produce diffeomorphic cobordism type $\<2\>$-surfaces. We extend them to $G$-spatial diagrams which identify those $G$-planar diagrams producing X-homeomorphic cobordism type $\<2\>$-X-surfaces. Then, using these identifications we define an equivalence relation among $G$-planar diagrams, and prove the $G$-planar decomposition theorem. 

Different generic maps on a fixed cobordism type $\<2\>$-surface yield different graphics just as different Morse functions yield different critical values. In the latter case, Cerf theory relates different sets of critical values in terms of isotopies or birth and death of critical values. Similarly, Schommer-Pries \cite{schommer} related different graphics obtained from different generic maps in terms of isotopies and certain local moves of graphics, called \textit{movie-moves} (see Figure \ref{fig:movie}). These movie-moves are obtained from the singularities of certain stratification of jet spaces $J^r((\Sigma \times I, \d_h \Sigma \times I , \d_v \Sigma \times I), (I^2 \times I, I \times \d I \times I, \d I \times I \times I))$ for a cobordism type $\<2\>$-surface $\Sigma$. Note that in general the map $F$ is not of the form $F(x,t)=(f_t(x),t)$. For our purposes we consider the subspace of $J^r(\Sigma \times I, I^2 \times I)$ consisting of paths of functions. 

\begin{figure}
    \centering
    \includegraphics{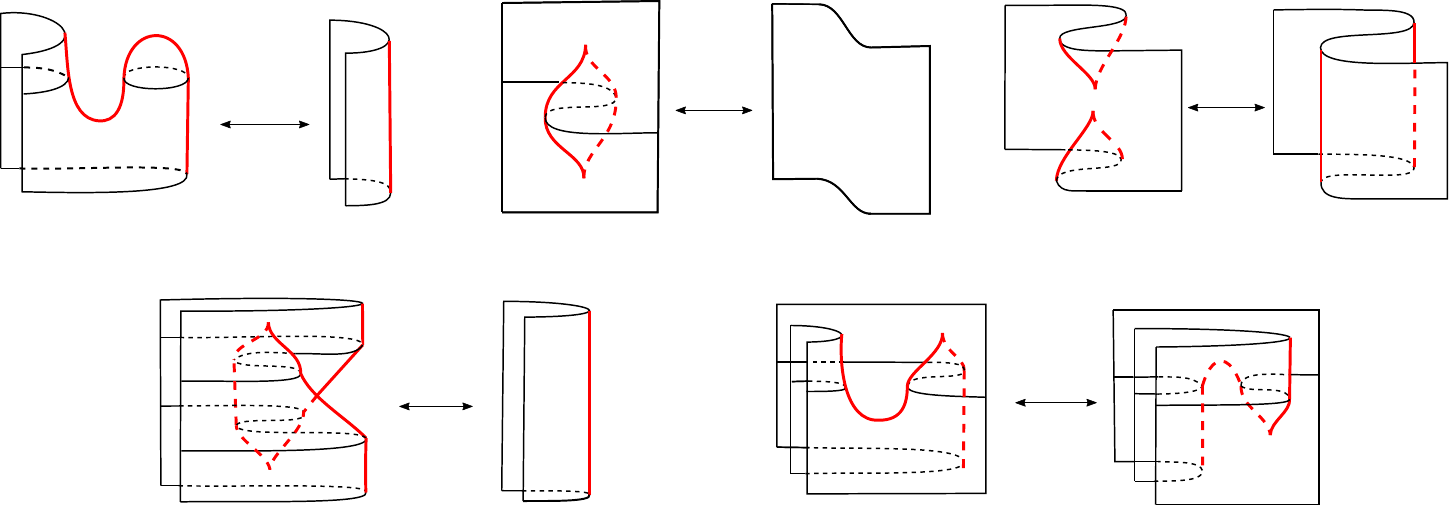}
    \caption{Movie-moves coming from codimension 3 singularities}
    \label{fig:movie}
\end{figure}

Figure \ref{fig:3dimsing} shows the graphics of singularities for the Schommer-Pries stratification in normal coordinates. Observe the relation between movie-moves in Figure \ref{fig:movie} and the horizontal boundary components of the new graphics. The remaining movie-moves coming from this stratification are shown in Figure \ref{fig:movie1}. The properties of this stratification are listed in in the following definition. In particular, the graphic of a generic map for this stratification is a $3$-dimensional graphic which is defined as follows.

%\begin{figure}[ht]
%    \centering
%    \includesvg{3dsing}
%    \caption{Graphics of the singularities in $I^2 \times I$}
%    \label{fig:3dimsing}
%\end{figure}
\begin{figure}[ht]
    \centering
    \def\svgwidth{\columnwidth}
    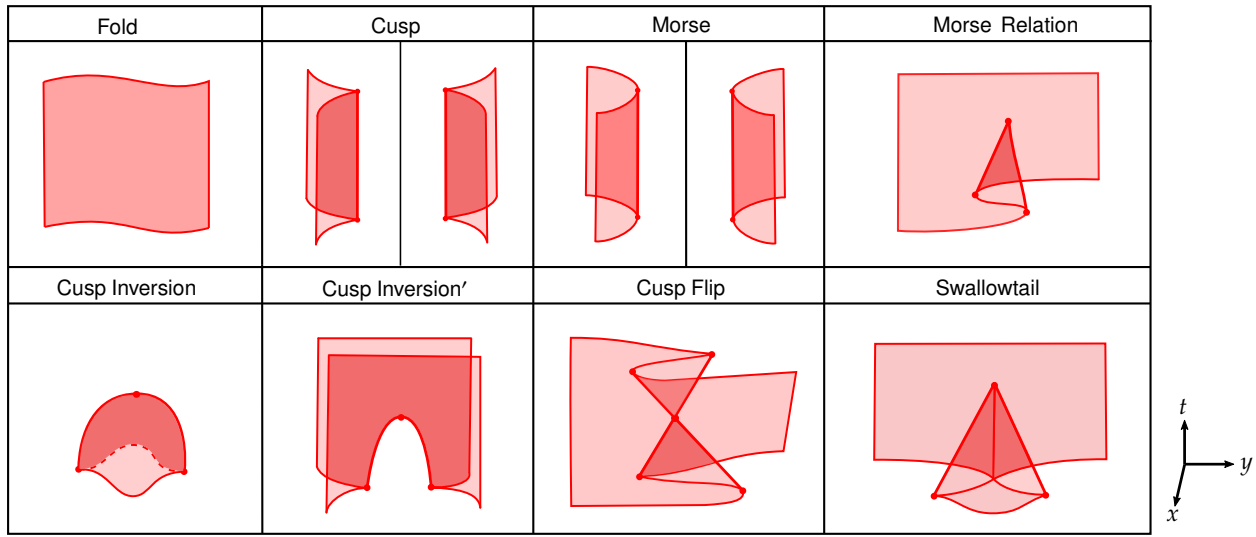
    \caption{Graphics of the singularities in $I^2 \times I$}
    \label{fig:3dimsing}
\end{figure}
\begin{definition}[Definition 1.30, \cite{schommer}]\label{3dimgraph}
A \textit{3-dimensional graphic} $\Delta=(\delta ,\eta,\mu)$ is a diagram in $I^2 \times I$ consisting of a finite number of embedded compact labeled surfaces $(\delta)$, a finite number of embedded labeled curves $(\eta)$, and a finite number of embedded labeled points $(\mu)$ satisfying the following conditions:
\begin{enumerate}[(i)]
    \item Projections of elements of $\delta$ to the last two coordinates are local diffeomorphisms. Elements of $\delta$ intersect with $I \times \{i\} \times I$ along vertical line segments $\{(x,i,s)\}_{s \in [0,1]}$ for some $x \in (0,1)$ and for $i=0,1$. Each element of $\delta$ is labeled with either Fold-$1$ or Fold-$2$. 
    \item Projections of elements of $\eta$ to the last coordinate are local diffeomorphisms.  Every element of $\eta$ has a neighborhood in which two elements of $\delta$ form either Morse or Cusp graphic (see Figure \ref{fig:3dimsing}). Each element of $\eta$ is labeled with either Morse-$i$\footnote{Morse singularities are paths of Cap, Cup, Saddle-$1$, and Saddle-$2$ singularities.} or Cusp-$i$ where $i = 1,2,3,4$ indicates the indices.
    \item Each element of $\mu$ has a neighborhood in which some elements of $\delta$ and $\eta$ form one of the following graphics: Morse relation-$i$, Cusp inversion-$j$, Cusp inversion$'$-$j$, Cusp flip-$j$, and Swallowtail-$i$ where $i=1,2,3,4$ and $j=1,2$ indicate graphics of different indices. 
    \item Elements of $\mu$ are labeled with one of the following singularities: Morse relation-$i$, Cusp inversion-$j$, Cusp inversion$'$-$j$, Cusp flip-$j$, and Swallowtail-$j$ where $1 \leq i \leq 8$ and $j =1,2,3,4$ indicate the indices.
    \item The restriction of the graphic to the components of $I^2 \times \d I$ gives $2$-dimensional graphics.
    \item Elements of $\delta,\eta,$ and $\mu$ are disjoint from $\d I \times I^2$. They are transversal with respect to each other and to $I^2 \times \d I$. Moreover, when two surfaces intersect along an arc there can only be finitely many points on the arc with tangent space lying in $\<\d_x,\d_y\>$ where $(x,y,t)$ is the coordinate for $I^2 \times I$.
\end{enumerate}
\end{definition}
\begin{figure}[h!]
    \centering
    \includegraphics{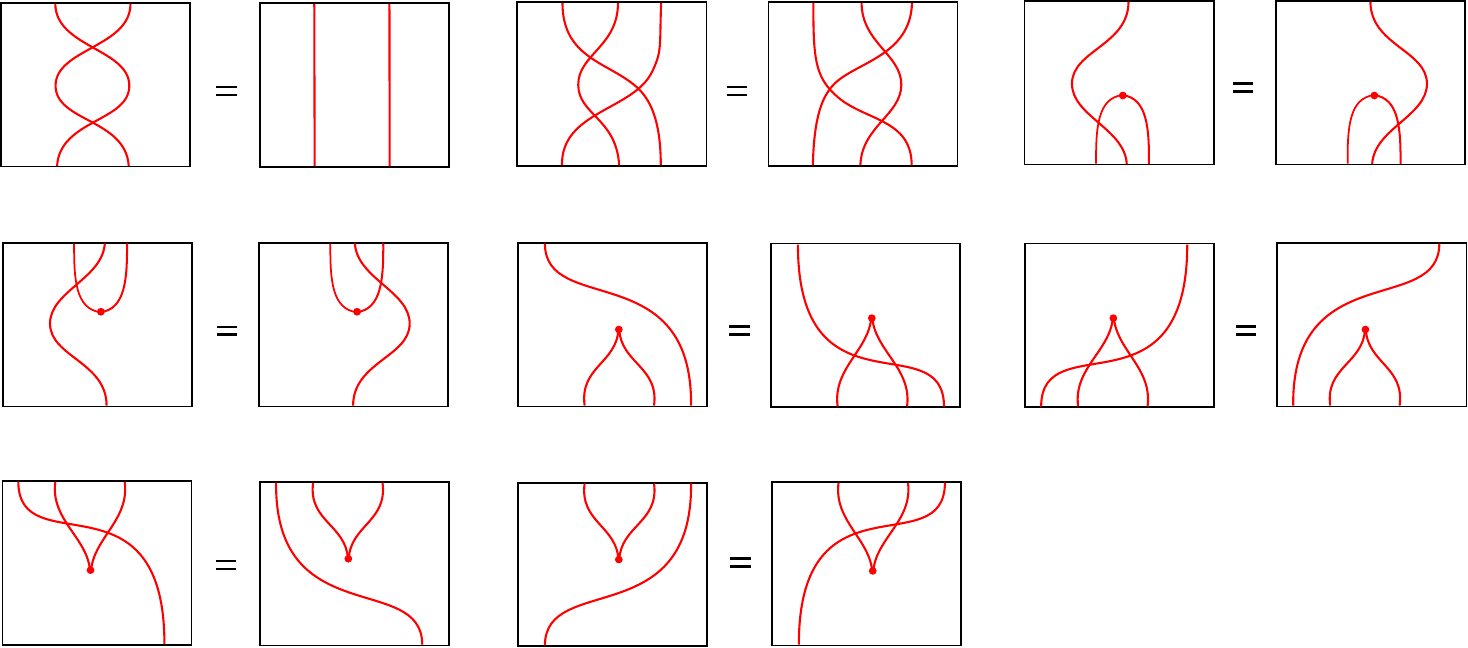}
    \caption{Movie-moves coming from intersection of codimension 1 and 2 singularities}
    \label{fig:movie1}
\end{figure}

%\begin{figure}
%    \centering
%    \includesvg{eskileri_relationlarin_yenilenmesi}
%    \caption{Some of the generalized movie-moves}
%    \label{fig:eskileri_relationlarin_yenilenmesi}
%\end{figure}
\begin{figure}
    \centering
    \def\svgwidth{\columnwidth}
    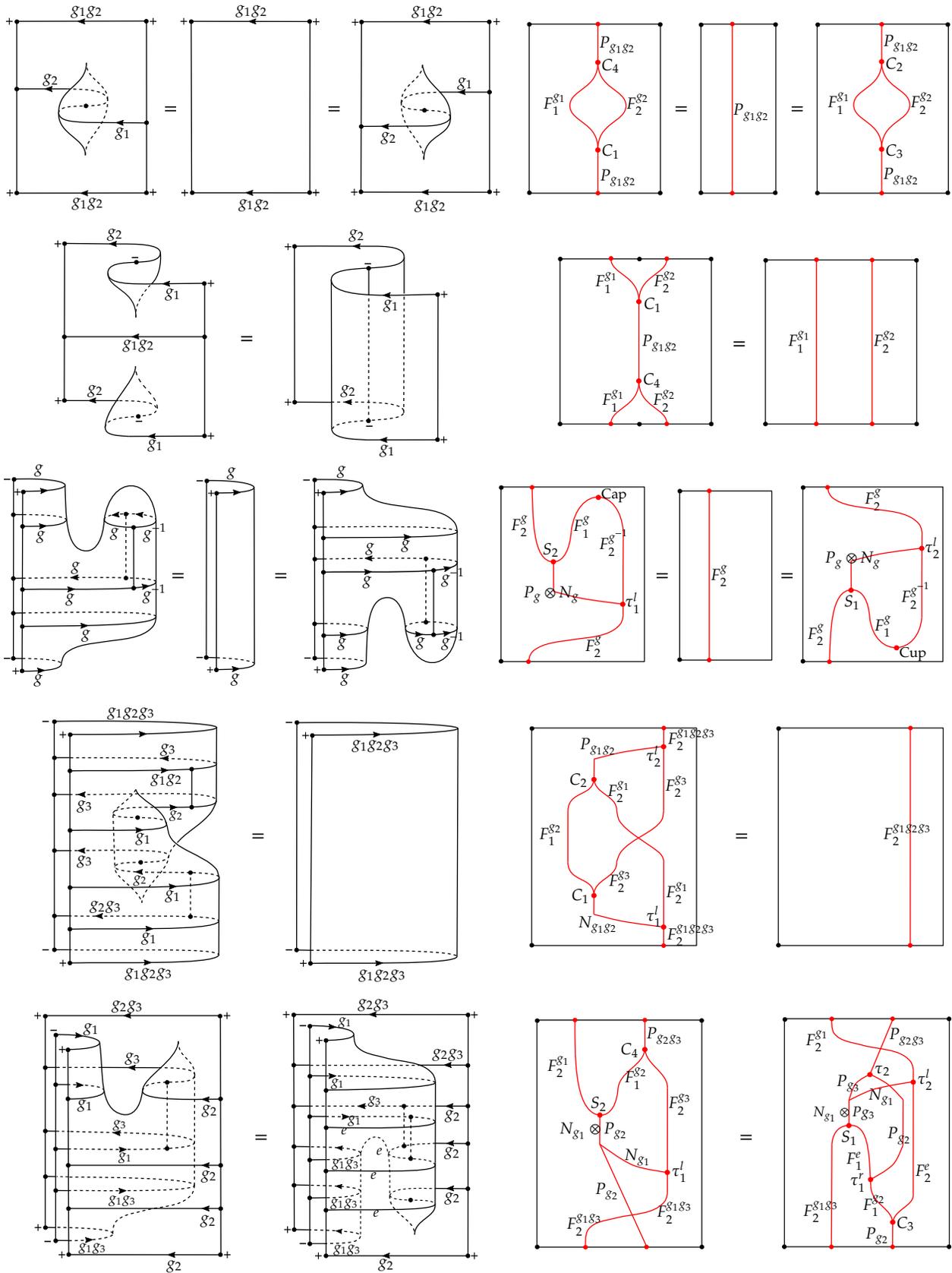
    \caption{Some of the generalized movie-moves}
    \label{fig:eskileri_relationlarin_yenilenmesi}
\end{figure}
%\begin{figure}
%    \centering
%    \includesvg{new_relations}
%    \caption{Some of the generalized movie-moves obtained by gluing elementary $\<2\>$-X-surfaces}
%    \label{fig:new_relations}
%\end{figure}
\begin{figure}
    \centering
    \def\svgwidth{\columnwidth}
    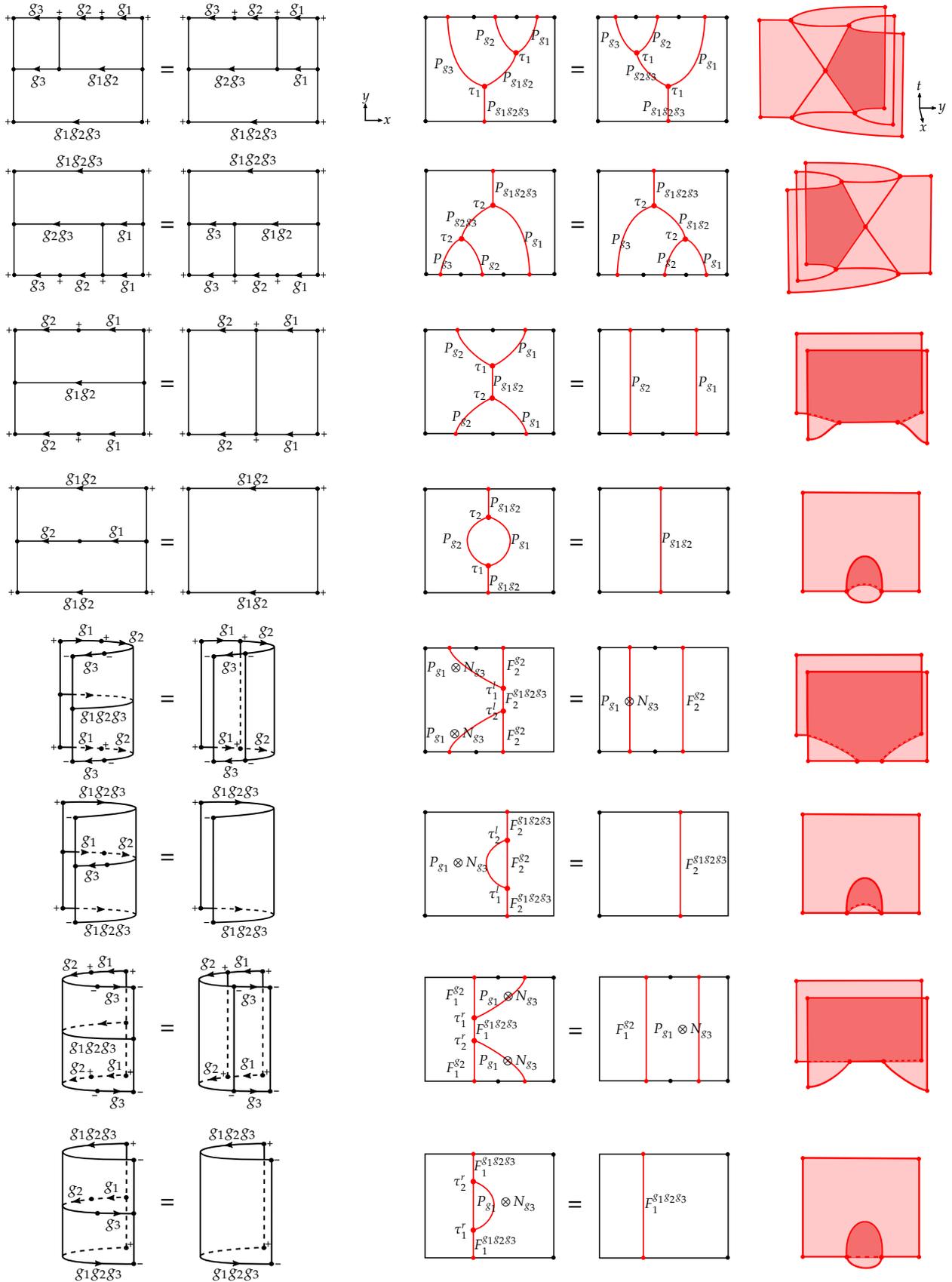
    \caption{Some of the generalized movie-moves obtained by gluing elementary $\<2\>$-X-surfaces}
    \label{fig:new_relations}
\end{figure}
Let $\Sigma$ be a cobordism type $\<2\>$-surface and $F:(\Sigma \times I, \d_h \Sigma \times I ,\d_v \Sigma \times I)\to (I^2 \times I, I \times \d I \times I, \d I \times I \times I)$ be a generic map. We know that the restriction of $F$ to boundary components $\Sigma \times \{0\}$ and $\Sigma \times \{1\}$ gives two generic maps on $\Sigma$. The converse is also true. That is, for any given two generic maps $f_1,f_2 : (\Sigma, \d_v \Sigma, \d_h \Sigma) \to (I^2 , \d I \times I, I \times \d I)$ there exists a generic map $F$ on $\Sigma \times I$ with $F|_{\Sigma \times \{0\}}= f_1$ and $F|_{\Sigma \times \{1\}} =f_2$. Therefore, $3$-dimensional graphics are designed to identify $2$-dimensional graphics obtained from generic maps on $\<2\>$-surfaces which are diffeomorphic relative to boundary. Using $2$-dimensional $G$-graphics, we extend this result to $\<2\>$-X-surfaces which are X-homeomorphic relative to their boundary. This is useful since in the cobordism bicategory $\X\Bord_2$, the $2$-morphisms are X-homeomorphism classes of cobordism type $\<2\>$-X-surfaces relative to their boundary. 

Movie-moves are local relations on $2$-dimensional graphics generating the identification of $2$-dimensional graphics induced from different generic maps. Figure \ref{fig:eskileri_relationlarin_yenilenmesi} and \ref{fig:new_relations} show some of the generalized movie-moves relating $2$-dimensional $G$-graphics. The remaining generalized movie-moves involves singularities with different indices, orientations, and decomposition into elementary $\<2\>$-X-surfaces. Similarly, movie-moves given in Figure \ref{fig:movie1} are generalized and their possible versions (different indices, orientations, decompositions) form generalized movie-moves. From now on, whenever we refer to these figures we mean the complete list of movie-moves. 

Similar to previous two sections, we generalize a $3$-dimensional graphic to a $G$-graphic by adding (labeled) surfaces, arcs, and points. For every new movie-move, the corresponding graphic is shown at the right hand side of Figure \ref{fig:new_relations}. Then, we define \textit{a $3$-dimensional $G$-graphic} $\Delta^G=(\delta^G,\eta^G,\mu^G)$ as a $3$-dimensional graphic $\Delta=(\delta,\eta,\mu)$ along with a $2$-dimensional locally conical stratified space of compact type which is transverse to the graphic and whose local models are given in Figure \ref{fig:new_relations}. The notion of locally conical stratified space is also the main ingredient of Definition \ref{chambering_foam_defn} below. Its definition can be found in \cite{schommer} (Definition 1.41). Reader can think of a locally conical stratified space as a space constructed from given local models just as a manifold is built locally from disks.

Let $\Delta^G=(\delta^G,\eta^G,\mu^G)$ be a $3$-dimensional $G$-graphic. An open cover $\mathcal{U}=\{U_{\alpha}\}_{\alpha \in J}$ of $I^3$ with at most $4$-fold intersections is said to be \textit{$\Delta$-compatible} (\cite{schommer}) if each $4$-fold intersection is disjoint from $\delta \cup \eta \cup \mu$, each $3$-fold intersection is disjoint from $\mu \cup \eta$ and contains at most a single component of surfaces in $\delta$, each double intersection is disjoint from points in $\mu$, and the open covers $\{ U_{\alpha} \cap (I^2 \times \{i\}) \}_{\alpha \in J}$ of $I^2 \times \{i \}$ for $i=0,1$ are compatible with the corresponding $2$-dimensional graphics obtained from $\Delta^G$. Since $I^3$ has covering dimension $3$ and there are only finitely many elements in $\delta$, $\eta$, and $\mu$, $\Delta$-compatible open covers exist. 

%\begin{figure}[h!]
%    \centering
%    \includesvg{kucukchamb}
%    \caption{Local models for a chambering foam}
%    \label{fig:cfoam}
%\end{figure}
\begin{figure}[h!]
    \centering
    \def\svgwidth{\columnwidth}
    %% Creator: Inkscape 1.0 (4035a4fb49, 2020-05-01), www.inkscape.org
%% Accompanies image file 'kucukchamb_1.pdf' (pdf, eps, ps)
\begingroup%
  \makeatletter%
  \providecommand\color[2][]{%
    \errmessage{(Inkscape) Color is used for the text in Inkscape, but the package 'color.sty' is not loaded}%
    \renewcommand\color[2][]{}%
  }%
  \providecommand\transparent[1]{%
    \errmessage{(Inkscape) Transparency is used (non-zero) for the text in Inkscape, but the package 'transparent.sty' is not loaded}%
    \renewcommand\transparent[1]{}%
  }%
  \providecommand\rotatebox[2]{#2}%
  \newcommand*\fsize{\dimexpr\f@size pt\relax}%
  \newcommand*\lineheight[1]{\fontsize{\fsize}{#1\fsize}\selectfont}%
  \ifx\svgwidth\undefined%
    \setlength{\unitlength}{459.91005725bp}%
    \ifx\svgscale\undefined%
      \relax%
    \else%
      \setlength{\unitlength}{\unitlength * \real{\svgscale}}%
    \fi%
  \else%
    \setlength{\unitlength}{\svgwidth}%
  \fi%
  \global\let\svgwidth\undefined%
  \global\let\svgscale\undefined%
  \makeatother%
  \begin{picture}(1,0.23597874)%
    \lineheight{1}%
    \setlength\tabcolsep{0pt}%
    \put(0,0){\includegraphics[width=\unitlength,page=1]{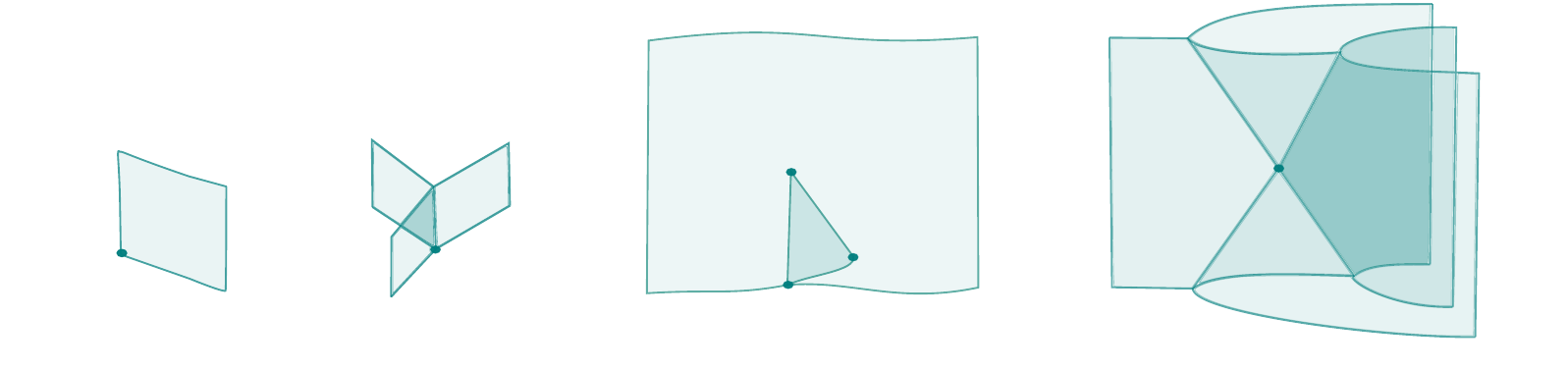}}%
    \put(0.25290315,0.01764399){\color[rgb]{0,0,0}\makebox(0,0)[lt]{\lineheight{1.25}\smash{\begin{tabular}[t]{l}\small $\mathsf{I \times C_3}$\end{tabular}}}}%
    \put(0.50043084,0.0163993){\color[rgb]{0,0,0}\makebox(0,0)[lt]{\lineheight{1.25}\smash{\begin{tabular}[t]{l}\small$\mathsf{CP}$\end{tabular}}}}%
    \put(0.79648277,0.00320976){\color[rgb]{0,0,0}\makebox(0,0)[lt]{\lineheight{1.25}\smash{\begin{tabular}[t]{l}\small$\mathsf{CK}_4$\end{tabular}}}}%
    \put(0.08359014,0.0168705){\color[rgb]{0,0,0}\makebox(0,0)[lt]{\lineheight{1.25}\smash{\begin{tabular}[t]{l}\small $\mathsf{I \times C_1}$\end{tabular}}}}%
  \end{picture}%
\endgroup%

    \caption{Local models for a chambering foam}
    \label{fig:cfoam}
\end{figure}
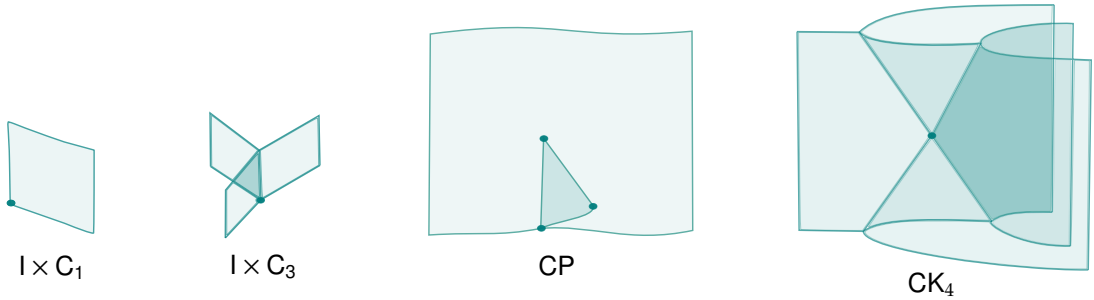
\theoremstyle{definition}
\begin{definition}\label{chambering_foam_defn} (Definition 1.43, \cite{schommer})
Let $\Delta^G=(\delta^G,\eta^G,\mu^G)$ be a $3$-dimensional $G$-graphic. A \textit{chambering foam} $\Gamma$ for $\Delta^G$ is a smooth embedding of a $2$-dimensional locally conical stratified space $\Gamma$ of compact type into $I \times (0,1) \times I$ with the following properties. The space $\Gamma$ is locally conical with respect to the system of local models $I^2$, $I \times C_1$, $I \times C_3, CP$, and $CK_4$ shown in Figure \ref{fig:cfoam}. Vertices are disjoint from $\Delta^G$ and lie in the interior. Edges can only intersect with a surface from $\delta^G$. Faces can only intersect with surfaces from $\delta^G$ and arcs from $\eta^G$. All intersections are transversal and $\Gamma$ additionally satisfies the following conditions:
    \begin{enumerate}[(I)]
        \item Projection $p : \Gamma \to I \times I$ to the last two coordinates has no singularity and projection of faces to the last coordinate has no singularity.
        \item For every $t \in I$ satisfying three conditions; $(I^2 \times \{t\}) \cap \mu^G = \emptyset$, $t$ is not a critical value of projection $\text{pr} : \Gamma \to I$ to the last coordinate, and $(I^2 \times \{t \}) \cap \Gamma$ does not include a vertex of $\Gamma$, the graph $(I^2 \times \{t\}) \cap \Gamma$ forms a chambering graph for the $2$-dimensional $G$-graphic $\Delta^G \cap (I^2 \times \{t\})$.
        \item Projection of each one of four edges in $CK_4$-model connecting at the cone point to the last coordinate is a local diffeomorphism. Additionally, at least one of them must map to downward of the cone point and at least one of them must map to upward of the cone point.
        \item Projection of the two edges in $CP$-model connecting at the cone point to the last coordinate maps both edges to the same direction with respect to the image of the cone point.
    \end{enumerate}
\end{definition}
\begin{definition}
Let $\Delta^G=(\delta^G,\eta^G,\mu^G)$ be a $3$-dimensional $G$-graphic and let $\Gamma$ be a chambering foam for $\Delta^G$. \textit{Chambers} of $\Gamma$ are the connected components of $I^2 \times I \backslash(\Gamma \cup \delta \cup \eta \cup \mu)$. A chambering foam $\Gamma$ is said to be \textit{subordinate to an open cover} $\mathcal{O}= \{O_{\alpha}\}_{\alpha \in J}$ of $I^2 \times I$ if each chamber is a subset of at least one $O_{\alpha}$ with $\alpha \in J$ and the chambering graphs $\Gamma \cap (I^2 \times \{i\})$ are compatible with the restricted open cover $\{ O_{\alpha}\cap I^2 \times \{i\} \}_{\alpha \in J}$ for $i=0,1$. 
\end{definition}
\begin{lemma} \label{compatible}
Let $\Gamma$ be a chambering foam for a $3$-dimensional $G$-graphic $\Delta^G$ inducing $2$-dimensional $G$-graphics and chambering graphs $(\Phi^G_0, \Gamma_0)$ and $(\Phi^G_1,\Gamma_1)$ on $I^2 \times \{0\}$ and $I^2 \times \{1\}$ respectively. Let $\mathcal{O}= \{O_{\alpha}\}_{\alpha \in J}$ be a $\Delta$-compatible open cover of $I^3$ such that $\Gamma_i$ is subordinate to $\mathcal{O}_i=\mathcal{O}|_{I^2 \times \{i\}}$ for $i=0,1$. Then, there exists a chambering foam $\Gamma'$ for $\Delta^G$ subordinate to $\mathcal{O}$ and whose restriction to $I^2 \times \{0\}$ and $I^2 \times \{1\}$ yields $\Gamma_0$ and $\Gamma_1$ respectively.
\end{lemma}
In \cite{schommer} the corresponding statement for a $3$-dimensional graphic was proven (see Corollary 1.47 in \cite{schommer}). In case of nontransversal intersections with new elements encoding X-manifold data, $\Gamma$ can be slightly modified to make all intersections transversal while being compatible with $\mathcal{O}$.

Just as movie-moves in Figures \ref{fig:eskileri_relationlarin_yenilenmesi} and \ref{fig:new_relations} generate relations locally between two $G$-planar diagrams on the boundary of a $G$-spatial diagram, there are movie-moves describing local relations between two chambering graphs on the boundary of a compatible chambering foam. These moves along with corresponding chambering foams are shown in Figure \ref{fig:foam_moves}. Moreover, there are movie-moves coming from an intersection of a $3$-dimensional $G$-graphic with a chambering foam. These local relations are shown in Figure \ref{fig:movies2} in which the labels are omitted. 
%\begin{figure}[h!]
%    \centering
%    \includesvg{movies3_1}
%    \caption{Movie-moves for chambering graphs and the corresponding chambering foams}
%    \label{fig:foam_moves}
%\end{figure}
\begin{figure}[h!]
    \centering
    \def\svgwidth{\columnwidth}
    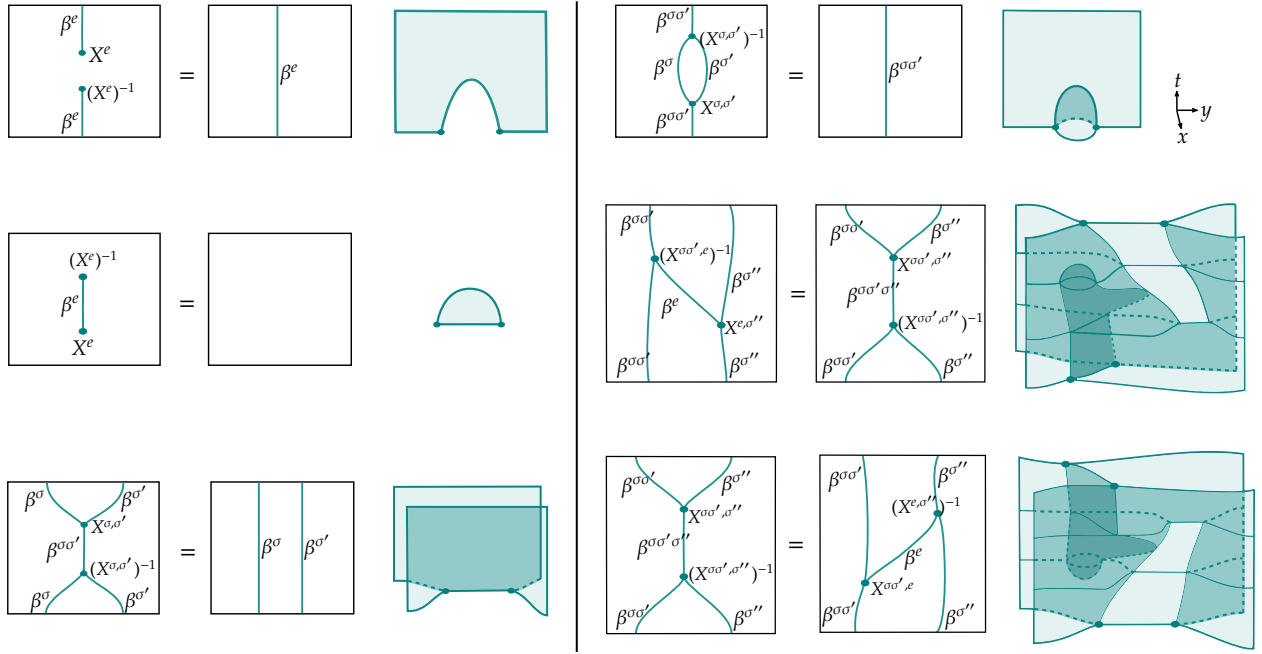
    \caption{Movie-moves for chambering graphs and the corresponding chambering foams}
    \label{fig:foam_moves}
\end{figure}

%\begin{figure}[h]
%    \centering
%    \includesvg{movies2_1}
%    \caption{Generating relations ($\mathcal{X}\mathcal{R}$) from movie-moves of graphics and chambering graphs}
%    \label{fig:movies2}
%\end{figure}
\begin{figure}[h]
    \centering
    \def\svgwidth{\columnwidth}
    %% Creator: Inkscape 1.0 (4035a4fb49, 2020-05-01), www.inkscape.org
%% Accompanies image file 'movies2_1_1.pdf' (pdf, eps, ps)
\begingroup%
  \makeatletter%
  \providecommand\color[2][]{%
    \errmessage{(Inkscape) Color is used for the text in Inkscape, but the package 'color.sty' is not loaded}%
    \renewcommand\color[2][]{}%
  }%
  \providecommand\transparent[1]{%
    \errmessage{(Inkscape) Transparency is used (non-zero) for the text in Inkscape, but the package 'transparent.sty' is not loaded}%
    \renewcommand\transparent[1]{}%
  }%
  \providecommand\rotatebox[2]{#2}%
  \newcommand*\fsize{\dimexpr\f@size pt\relax}%
  \newcommand*\lineheight[1]{\fontsize{\fsize}{#1\fsize}\selectfont}%
  \ifx\svgwidth\undefined%
    \setlength{\unitlength}{463.97529053bp}%
    \ifx\svgscale\undefined%
      \relax%
    \else%
      \setlength{\unitlength}{\unitlength * \real{\svgscale}}%
    \fi%
  \else%
    \setlength{\unitlength}{\svgwidth}%
  \fi%
  \global\let\svgwidth\undefined%
  \global\let\svgscale\undefined%
  \makeatother%
  \begin{picture}(1,0.69727138)%
    \lineheight{1}%
    \setlength\tabcolsep{0pt}%
    \put(0,0){\includegraphics[width=\unitlength,page=1]{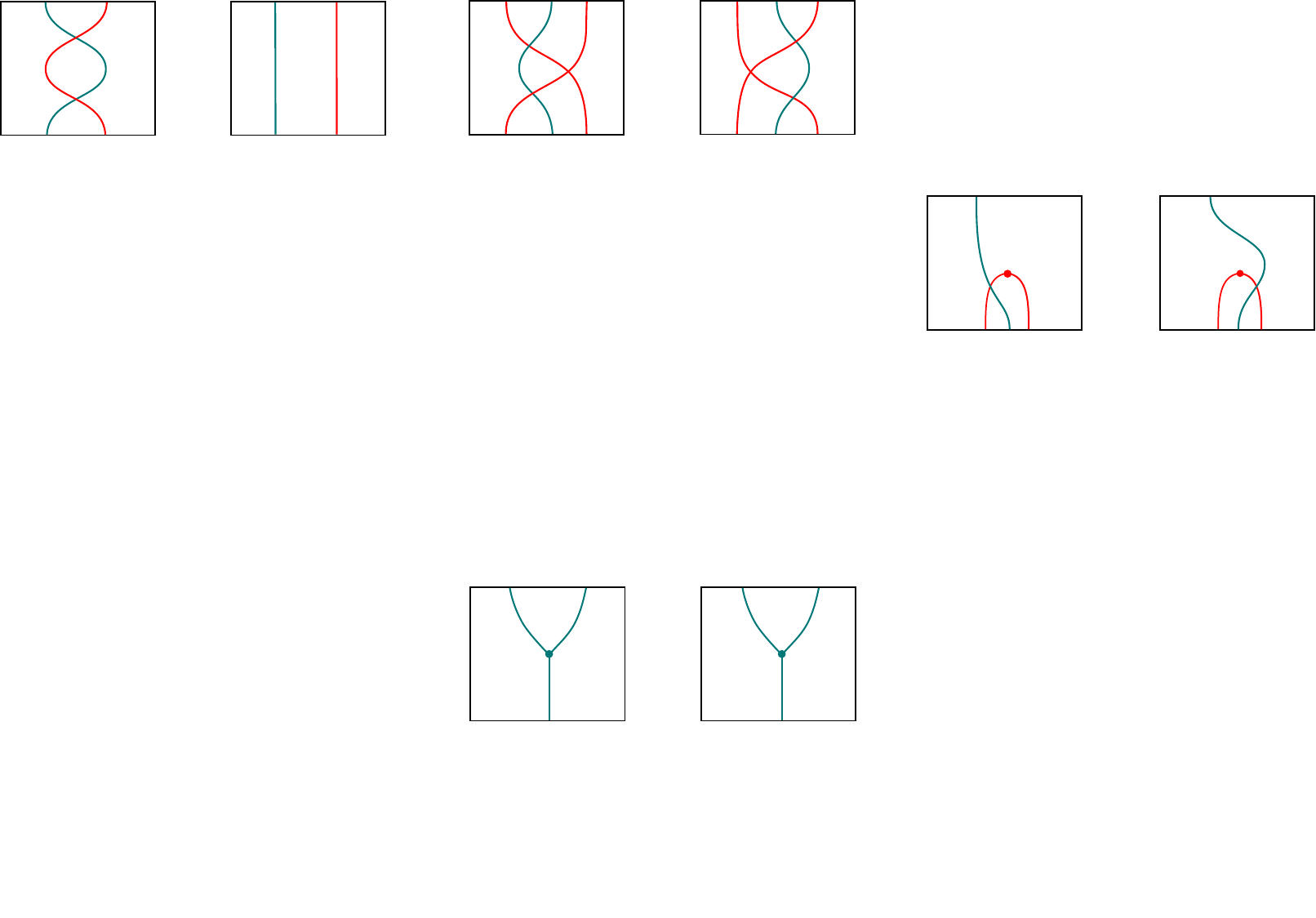}}%
    \put(0.35891054,0.04804323){\color[rgb]{0,0,0}\makebox(0,0)[lt]{\lineheight{1.25}\smash{\begin{tabular}[t]{l}\footnotesize+$\mathrm{reflections}$ $\mathrm{with}$ $\mathrm{respect}$ $\mathrm{to}$ $\mathrm{vertical}$ $\mathrm{and}$ $\mathrm{horizontal}$ $\mathrm{axes}$ \\\end{tabular}}}}%
    \put(0,0){\includegraphics[width=\unitlength,page=2]{movies2_1_1.pdf}}%
  \end{picture}%
\endgroup%

    \caption{Generating relations ($\mathcal{X}\mathcal{R}$) from movie-moves of graphics and chambering graphs}
    \label{fig:movies2}
\end{figure}
Next, we describe the sheet data. Let $\Delta=(\delta, \eta,\mu)$ be a $3$-dimensional graphic induced from a generic map $F: (\Sigma \times I, \d_h \Sigma \times I ,\d_v \Sigma \times I)\to (I^2 \times I, I \times \d I \times I, \d I \times I \times I)$ where $\Sigma$ is a cobordism type $\<2\>$-surface. Let $\Gamma$ be a chambering foam subordinate to a $\Delta$-compatible open cover. Similar to previous section a sheet data associated to $(\Delta,\Gamma)$ extends the sheet data of planar diagrams on faces $I^2 \times \{0,1\}$. Since $F$ is generic the preimage $F^{-1}(O_{\beta})$ of an open chamber consists of disjoint union of open sets each mapping diffeomorphically onto $O_{\beta}$. If a chamber $O_{\beta}$ is not open then we consider $O_{\beta}'= O_{\beta} \backslash (O_{\beta} \cap \d (I^2 \times I))$. Then, a \textit{trivialization} of a chamber is the identification of $F^{-1}(O_{\beta})$ with $\mathbb{N}_{\leq N} \times O_{\beta}$ if $F^{-1}(O_{\beta})$ is nonempty and identification with the empty set otherwise. Each $\{i\} \times O_{\beta}$ is called a \textit{sheet} and each sheet is oriented. For every chamber $O_{\beta}$ which is not open, we extend identifications to $\mathbb{N}_{\leq N} \times O_{\beta}$ by requiring the same trivializations on $O_{\beta}'$ and $O_{\beta} \cap \d (I^2 \times I)$ coming from sheet data of planar diagrams.

Trivializations of two neighboring chambers have the same number of sheets if chambers are separated by a $2$-dimensional strata of $\Gamma$. If an element in $\delta$ separates chambers then the number of sheets differ by two. A sheet data $\mathcal{S}$ (\cite{schommer}) for a pair $(\Delta,\Gamma)$ consists of trivialization of each chamber and an injection or a permutation between trivializations of neighboring chambers preserving orientations and describing how sheets are glued.

Gluing description of sheets requires the following conditions on permutations and injections. In the local models $I \times C_3, CP$, and $CK_4$ circular compositions of three or four permutations must be identity. Since fold, cusp, and Morse graphics are paths of the corresponding graphics in the previous section their sheet data do not change. According to properties of multijet stratification transversal double and triple fold intersections are possible. There are four chambers for the double and eight for the triple fold intersection. In both cases different compositions of injections starting from the chamber with the least number of sheets and ending at the chamber with the maximum number of sheets must be the same. The sheet data for intersection of fold and Morse graphics is the same as double fold intersection and the sheet data for intersection of fold and cusp graphics follows from the sheet data of cusp graphic. Sheet data of Morse relation, cusp inversion, cusp inversion$'$, and cusp flip graphics can be interpreted from the corresponding movie moves (see Figure \ref{fig:movie}). For the details of these sheet data see Section 1.5.2 in \cite{schommer}. 
%\begin{figure}[ht]
%    \centering
%    \includesvg{swallowtail}
%    \caption{The Swallowtail-$1$ sheet data}
%    \label{fig:swallow}
%\end{figure}
\begin{figure}[ht]
    \centering
    \def\svgwidth{\columnwidth}
    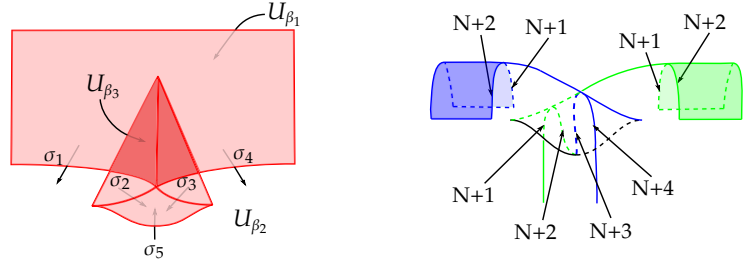
    \caption{The Swallowtail-$1$ sheet data}
    \label{fig:swallow}
\end{figure}

We shortly describe the sheet data of Swallowtail-$1$ graphic shown in Figure \ref{fig:swallow} where two (blue and green) out of three fold singularities form a double fold crossing. Let $\mathbb{N}_{\leq N} \times U_{\beta_1}$, $\mathbb{N}_{\leq N+2}\times U_{\beta_2}$, and $\mathbb{N}_{\leq N+4} \times U_{\beta_3}$ be trivializations of chambers such that sheets $\bigcup_{i=N+1}^{N+2} i  \times U_{\beta_2}$ and $\bigcup_{j=N+1}^{N+4} j \times U_{\beta_3}$ belong to swallowtail singularity as shown in Figure \ref{fig:swallow}. Using the sheet data for cusp singularities, restrictions of injections to these sheets give
\begin{align*}
    \sigma_2(N+1)&= N+3, \hspace{1cm} \sigma_2(N+2)= N+4, \\
    \sigma_3(N+1)&=N+1, \hspace{1cm} \sigma_3(N+2)= N+2, \\
    \sigma_5(N+1)&=N+1, \hspace{1cm} \sigma_5(N+2)= N+4.
\end{align*}
\begin{definition}[Definition 1.49, \cite{schommer}]
A \textit{spatial diagram} is a triple $(\Delta,\Gamma,\mathcal{S})$ consisting of a $3$-dimensional graphic $\Delta$, a chambering foam $\Gamma$ for $\Delta$ subordinate to a $\Delta$-compatible cover $\mathcal{O}= \{ O_{\alpha}\}_{\alpha \in J}$ of $I^3$, and a sheet data $\mathcal{S}$ associated to the pair $(\Delta,\Gamma)$. 
\end{definition}
Any spatial diagram $(\Delta, \Gamma, \mathcal{S})$ produces a compact $3$-dimensional manifold  with corners $M$ with $\d M = \Sigma_1 \sqcup \overline{\Sigma}_2$ where $\Sigma_1$ and $\Sigma_2$ are $\<2\>$-surfaces. Similar to previous two sections, in the case of a tuple $(\Delta^G,\Gamma)$, the associated sheet data can be improved to yield a relative homotopy class from $M$ to $X$ and $\<2\>$-X-surfaces $(\Sigma_1, R_1, \mathtt{P}_1)$ and $(\Sigma_2, R_2, \mathtt{P}_2)$. We call such sheet data carrying X-manifold data to sheets \textit{$G$-sheet data} and denote it with $\mathcal{S}^G$. Then generalizing a spatial diagram, we define a \textit{$G$-spatial diagram} as a triple $(\Delta^G, \Gamma, \mathcal{S}^G)$. 

\begin{prop}\label{independence}
Let $(\Phi_1^G,\Gamma_1,\mathcal{S}^G_1)$ and $(\Phi_2^G, \Gamma_2,\mathcal{S}^G_2)$ be $G$-planar diagrams and let $(\Sigma_1,R_1, \mathtt{P}_1)$ and $(\Sigma_2,R_2, \mathtt{P}_2)$ be the constructed cobordism type $\<2\>$-X-surfaces respectively. Then, $(\Sigma_1,R_1,\mathtt{P}_1)$ is X-homeomorphic to $(\Sigma_2,R_2, \mathtt{P}_2)$ relative to boundary if and only if there exists a $G$-spatial diagram $(\Delta^G,\Gamma,\mathcal{S}^G)$ which restricts to $(\Phi_1^G,\Gamma_1,\mathcal{S}_1^G)$ and $(\Phi_2^G,\Gamma_2,\mathcal{S}_2^G)$ on components of $I^2 \times \d I$. 
\end{prop}
\begin{proof}
$(\Rightarrow)$ For a given such X-homeomorphism $\mathcal{F}$, we obtain a $G$-spatial diagram by first taking a generic map on the mapping cylinder of $\mathcal{F}$ and then choosing a compatible chambering foam. $(\Leftarrow)$ The properties of the stratification imply that the boundary components of the manifold constructed from the $G$-spatial diagram are diffeomorphic relative to boundary and the compatibility of $G$-labels of the arcs on the constructed manifold implies that they are X-homeomorphic.
\end{proof}
We define a relation among $G$-planar diagrams by $(\Phi_1^G, \Gamma_1,\mathcal{S}_1^G)\sim (\Phi_2^G,\Gamma_2,\mathcal{S}_2^G)$ if there exists a $G$-spatial diagram $(\Delta^G,\Gamma,\mathcal{S}^G)$ restricting to the given $G$-planar diagrams on components of $I^2 \times \d I$. It is not hard to see that $\sim$ is an equivalence relation. Since generalized movie-moves provide (nontrivial) local relations on $G$-planar diagrams, the equivalence relation $\sim$ can be described using these moves as follows. 
\begin{prop}
Two $G$-planar diagrams are equivalent if and only if they can be related by a finite sequence of isotopies or movie-moves in Figures\footnote{By referring these figures we mean the list of all generalized movie-moves described in this section.} \ref{fig:movie1}, \ref{fig:eskileri_relationlarin_yenilenmesi}, \ref{fig:new_relations}, \ref{fig:foam_moves}, and \ref{fig:movies2}. 
\end{prop}
Proposition \ref{independence} implies the following theorem which is the first main step towards the classification of $2$-dimensional extended X-HFTs.
\begin{theorem}[$G$-planar decomposition theorem]
The relative X-homeomorphism classes of cobordism type $\<2\>$-X-surfaces are in bijection with the equivalence classes of $G$-planar diagrams.  
\end{theorem}

\section{The Classification of 2-dimensional Extended X-HFTs}
\subsection{New bicategories arising from diagrams}
In this section, we use the $G$-planar decomposition theorem to introduce symmetric monoidal X-cobordism bicategories with diagrams.
\begin{definition}
An object of an X-cobordism bicategory with diagrams $\X\Bord_2^{\pd}$ is a triple  $((M,\hat{M}_1,\hat{M}_2,\hat{\mathtt{g}}_2),\widebar{M},\omega)$ where $(M,\hat{M}_1,\hat{M}_2,\hat{\mathtt{g}}_2)$ is an object of $\X\Bord_2$, $\widebar{M}$ is a finite set of ordered oriented points, and $\omega : M \to \widebar{M}$ is an orientation preserving bijection. 

A $1$-morphism is a triple $((A,\hat{A}_0,\hat{A}_1,T,\hat{\mathtt{p}}_1),L,\nu)$ where $(A,\hat{A}_0,\hat{A}_1,T,\hat{\mathtt{p}}_1)$ is an X-haloed $1$-cobordism, $L=(\Psi^G,\Gamma,\mathcal{S}^G)$ is a $G$-linear diagram, and $\nu: (A,T,p) \to (\widebar{A},\widebar{T},\widebar{\mathtt{p}})$ is an X-homeomorphism over $I$ with $\nu(T)=\widebar{T}$ where $(\widebar{A},\widebar{T},\widebar{\mathtt{p}})$ is the pointed $1$-cobordism constructed from the $G$-linear diagram. The composition of two composable triples is componentwise; the composition of $1$-morphisms in $\X\Bord_2$, the composition of diagrams described below, and the extension of two X-homeomorphisms over $I$, respectively.

A $2$-morphism is a triple $([(S,\hat{S},R,\hat{\mathtt{F}})],P, \kappa)$ where $[(S,\hat{S},R,\hat{\mathtt{F}})]$ is an isomorphism class of an X-haloed $2$-cobordism,  $P=[(\Phi^G,\Gamma,\mathcal{S}^G)]$ is an equivalence class of a $G$-planar diagram, and $\kappa: (S,R,\mathtt{F}) \to (\widebar{S},\widebar{R},\widebar{\mathtt{F}})$ is an X-homeomorphism over $I^2$ where $(\widebar{S},\widebar{R},\widebar{\mathtt{F}})$ is a cobordism type $\<2\>$-X-manifold constructed from a representative $(\Phi^G,\Gamma,\mathcal{S}^G)$. The composition of two composable triples is componentwise, similar to the composition of $1$-morphisms.
\end{definition}
The second bicategory $\XB^{\PD}$ is defined by forgetting X-haloed manifolds and cobordisms in $\X\Bord_2^{\pd}$ and taking isotopy classes of $G$-linear diagrams. In order to define isotopic $G$-linear diagrams, we first need to explain compositions and monoidal products of diagrams.

Horizontal compositions of $G$-linear and $G$-planar diagrams are given by the horizontal concatenation of diagrams where both $G$-sheet data agree and form a new $G$-sheet data. Vertical composition of equivalence classes of $G$-planar diagrams is vertical concatenation of diagrams followed by an isomorphism $I \cup_{\text{pt}}I \cong I$ and forgetting the $G$-linear diagram on the face along which two $G$-planar diagrams are concatenated. Figure \ref{fig:composition} shows an example of horizontal and vertical compositions of $2$-morphisms in $\XB^{\PD}$ whose labels are omitted.

%\begin{figure}[ht]
%    \centering
%    \includesvg{composition2}
%    \caption{Compositions and symmetric monoidal product of $2$-morphisms in $\XB^{\PD}$}
%    \label{fig:composition}
%\end{figure}
\begin{figure}[ht]
    \centering
    \def\svgwidth{\columnwidth}
    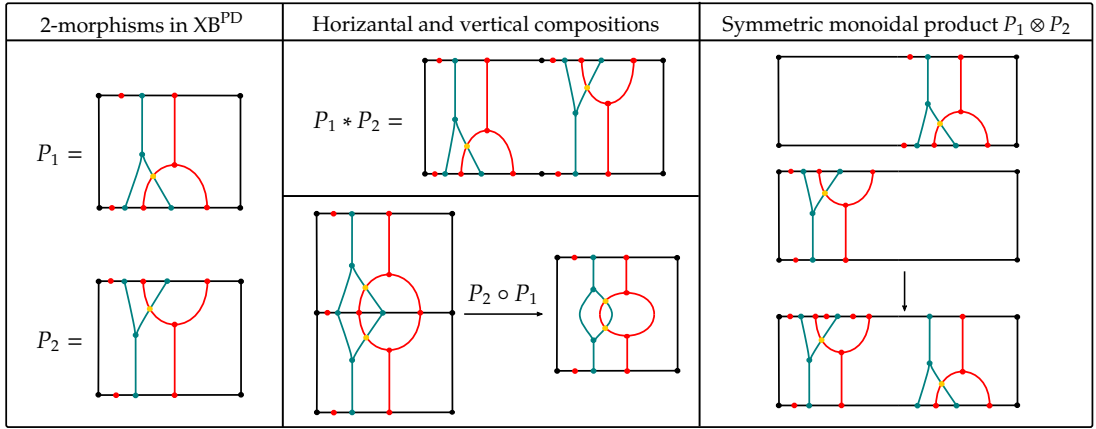
    \caption{Compositions and symmetric monoidal product of $2$-morphisms in $\XB^{\PD}$}
    \label{fig:composition}
\end{figure}
Symmetric monoidal structure on $\XB^{\PD}$ is defined as follows. Let $P_1=(\Phi^G_1,\Gamma_1,\mathcal{S}^G_1)$ and $P_2=(\Phi^G_1,\Gamma_1,\mathcal{S}^G_2)$ be two $G$-planar diagrams on $[m,n] \times I$ and on $[a,b] \times I$ for $m,n,a,b \in \mathbb{Z}$, respectively. Let $V_{\text{left}}$ be the leftmost chamber of $P_1$ and $V_{\text{right}}$ be the rightmost chamber of $P_2$. Then, $P_1 \otimes P_2$ is defined by stretching $V_{\text{left}}$ to the left by $b-a$ units, stretching $V_{\text{right}}$ to the right by $n-m$ units, and joining the stretched diagrams (see Figure \ref{fig:composition}). The $G$-sheet data and the labels of the resulting diagram are modified accordingly. Symmetric monoidal structure on $G$-linear diagrams can be deduced from this description. It is not hard to see that the described symmetric monoidal product of diagrams is compatible with the disjoint union of X-haloed manifolds.

Recall that objects of $\XB^{\PD}$ are finite set of ordered oriented points, $1$-morphisms are isotopy classes of $G$-linear diagrams, and $2$-morphisms are equivalence classes of $G$-planar diagrams. The notion of isotopy between $G$-linear diagrams is generated by the following identifications. Let $L=(\Psi^G,\Gamma,\mathcal{S}^G)$ be any $G$-linear diagram, $\emptyset$ be the empty $G$-linear diagram for the empty $1$-manifold, and $\id_a$ be the identity $G$-linear diagram of the ordered set $a$. Then $L= L \otimes \emptyset= \emptyset \otimes L$ and $L= L \circ \id_{a}=\id_{b} \circ L$ where $L: a \to b$. In this case, it is not hard to see that $\XB^{\PD}$ is a strict $2$-category.  
\begin{lemma}
\label{thm:bicat1}
Both $\X\Bord_2^{\pd}$ and $\XB^{\PD}$ are symmetric monoidal bicategories under disjoint union of X-haloed manifolds and operation $\otimes$ on diagrams.
\end{lemma}
The proof for $\XB^{\PD}$ is very similar to the proof of Lemma \ref{thm:bicat}. The case of $\X\Bord_2^{\pd}$ follows from the compatibility of symmetric monoidal structures. Considering the $G$-planar decomposition theorem, a natural question is whether the symmetric monoidal bicategory $\XB^{\PD}$ defined by using diagrams is symmetric monoidally equivalent to X-cobordism bicategory $\X\Bord_2$. We give a positive answer using the following theorem. 
\begin{theorem}[Whitehead theorem for symmetric monoidal bicategories, Theorem 2.25, \cite{schommer}]\label{whiteadforsym}
Let $\mathcal{B}$ and $\mathcal{C}$ be symmetric monoidal bicategories. A symmetric monoidal $2$-functor $F: \mathcal{B} \to \mathcal{C}$ is a symmetric monoidal equivalence if and only if it is an equivalance of underlying bicategories. That is, $F$ is essentially surjective on objects, essentially full on $1$-morphisms, and fully-faithful on $2$-morphisms.
\end{theorem}
\begin{prop}\label{thm1}
The forgetful $2$-functors $F$ and $G$ given by forgetting X-haloed cobordisms and diagrams respectively
 \begin{align*}
     \XB^{\PD} \xlongleftarrow[\simeq]{\quad F \quad}& \X\Bord_2^{\pd} \xlongrightarrow[\simeq]{\quad G \quad} \X\Bord_2
\end{align*}
are symmetric monoidal equivalences.
\end{prop}
\begin{proof}
For any given finite set $W$ of ordered oriented points or a compact oriented $0$-manifold with co-oriented codimension two X-halation $(Y,\hat{Y}_0,\hat{Y}_1,\hat{\mathtt{g}})$, there exist objects in $\X\Bord_2^{\pd}$ whose images under $F$ and $G$ are isomorphic to $W$ and $(Y,\hat{Y}_0,\hat{Y}_1,\hat{\mathtt{g}})$ respectively. For any given X-haloed $1$-cobordism there exists a Morse function with distinct critical values leading to a $G$-linear diagram and any $G$-linear diagram produces an X-haloed\footnote{Halation can be encoded into a $G$-sheet data by equipping trivializations of chambers with halations.} $1$-cobordism. Thus, by Proposition \ref{lindecomp} each $2$-functor is (essentially) full on $1$-morphisms. Lastly, by the $G$-planar decomposition theorem $2$-functors $F$ and $G$ are fully-faithful on $2$-morphisms. 
\end{proof}
Proposition \ref{thm1} implies that the X-cobordism bicategory $\X\Bord_2$ is symmetric monoidally equivalent to $\mathbbmss{\XB^{\PD}}$. The advantage of $\mathbbmss{\XB^{\PD}}$ is being a computadic unbiased semistrict symmetric monoidal $2$-category. In Appendix \ref{freely_gen_bicat_section}, we provide the definition of computadic unbiased semistrict symmetric monoidal $2$-category and prove this claim whose precise statement is given below (see Theorem \ref{computadic}). This result is an important step of the classification, which we want to describe here.
%\begin{figure}[t]
%    \centering
%      \includesvg{AllGgenrel6}
%   \caption{Generating objects ($\mathcal{XG}_0$), $1$-morphisms ($\mathcal{XG}_1)$, and $2$-morphisms ($\mathcal{XG}_2$)}
%    \label{fig:generators}
%\end{figure}
\begin{figure}[t]
    \centering
    \def\svgwidth{\columnwidth}
    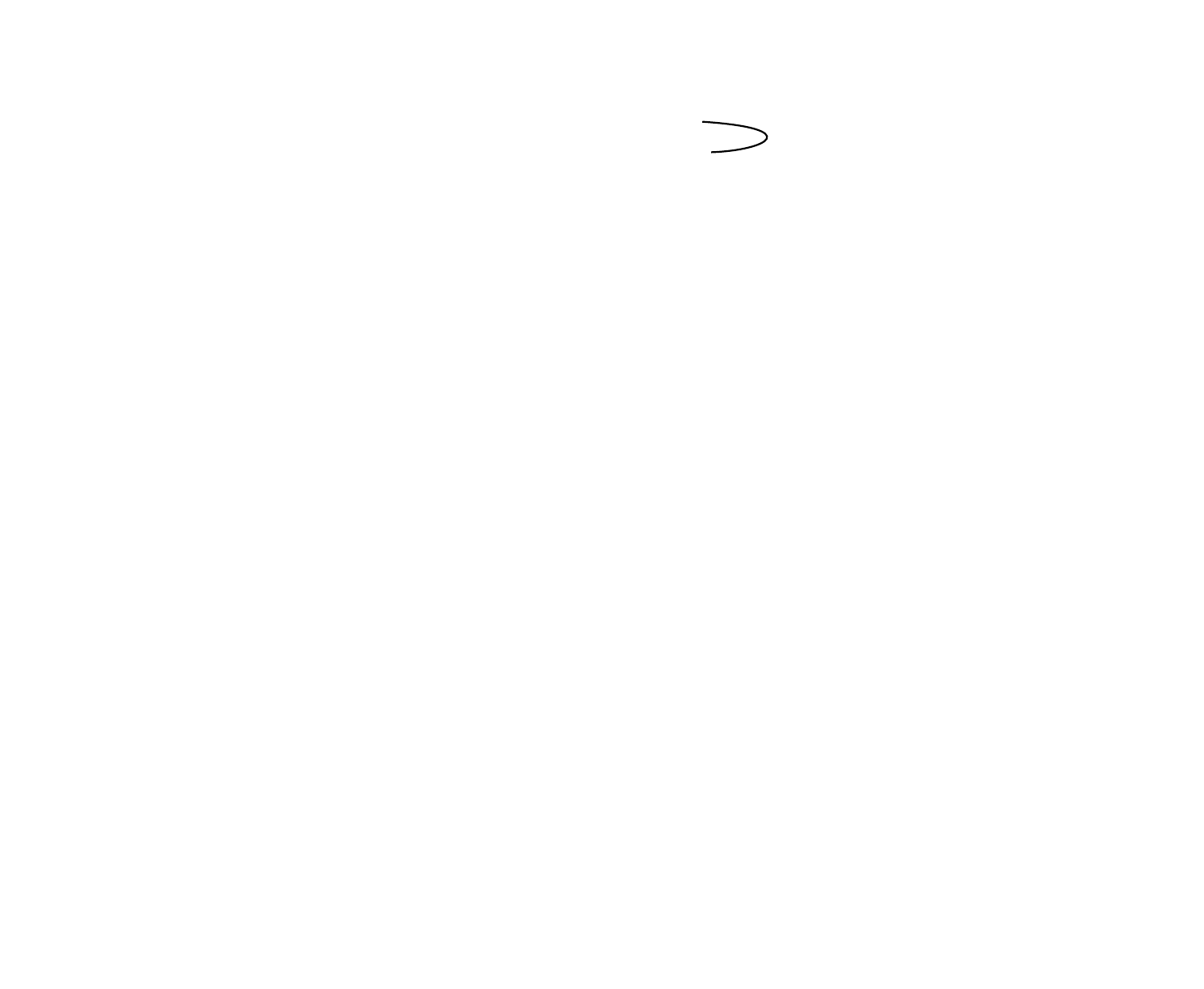
    \caption{Generating objects ($\mathcal{XG}_0$), $1$-morphisms ($\mathcal{XG}_1)$, and $2$-morphisms ($\mathcal{XG}_2$)}
    \label{fig:generators}
\end{figure}    
    
The fact that $\XB^{\PD}$ is a computadic symmetric monoidal bicategory roughly means that there exist four sets namely generating objects $\mathcal{XG}_0$, generating $1$-morphisms $\mathcal{XG}_1$, generating $2$-morphisms $\mathcal{XG}_2$, and generating relations $\mathcal{XR}$ among $2$-morphisms forming the presentation $\XP$ such that there exists an (canonical) isomorphism of symmetric monoidal bicategories $\mathsf{F}_{\uss}(\XP) \to \XB^{\PD}$ where $\mathsf{F}_{\uss}(\XP)$ is constructed from $\XP$. Therefore, we have
\begin{equation}\label{unbiased}
\begin{gathered}
     \xymatrixcolsep{4pc}\xymatrix{ \mathsf{F}_{\text{uss}}(\XP) \ar@{.>}[r]^{\exists }_{\simeq} & \XB^{\PD} & \X\Bord_2^{\pd} \ar[r]^{G}_{\simeq} \ar[l]_{F}^{\simeq} & \X\Bord_2
    }
\end{gathered}
\end{equation}
and the cofibrancy theorem states that symmetric monoidal $2$-functors out of $\mathsf{F}_{\uss}(\XP)$ are determined by the images of generating sets subject to the relations. Thus, the classification of $2$-dimensional extended X-HFTs up to symmetric monoidal equivalence
reduces to understanding images of generators in $\XP$ satisfying relations. The following theorem lists the presentation $\XP=(\mathcal{XG}_0,\mathcal{XG}_1,\mathcal{XG}_2,\mathcal{XR})$ of $\XB^{\PD}$ and its proof is given in Appendix \ref{proof_of_computadic}.
\begin{theorem} \label{computadic}
The symmetric monoidal bicategory $\XB^{\PD}$ is a computadic unbiased semistrict symmetric monoidal $2$-category with the presentation $\XP=(\mathcal{XG}_0,\mathcal{XG}_1,\mathcal{XG}_2,\mathcal{XR})$ given by the diagram versions\footnote{For generating $2$-morphisms, we mean equivalence classes of $G$-planar diagrams. We consider $G$-linear and $G$-planar diagrams whose chambering sets and graphs are trivial corresponding to covers with single elements.} of elements in Figures \ref{fig:generators} and \ref{fig:relations}, and pairs of $G$-planar diagrams corresponding to equalities in Figure \ref{fig:movies2}, where the labels $g_1,g_2,g_3,g_4,g,g',g''$ are indexed over $G$ such that $g_1 g_2 g_3 g_4=e$.
%\begin{figure}[h!]
%    \centering
%    \includesvg{relT4}
%   \caption{Generating relations ($\mathcal{X}\mathcal{R}$) among $2$-morphisms}
%    \label{fig:relations}
%    \vspace{-0.6cm}
%\end{figure}
\begin{figure}[h!]
    \centering
    \def\svgwidth{\columnwidth}
    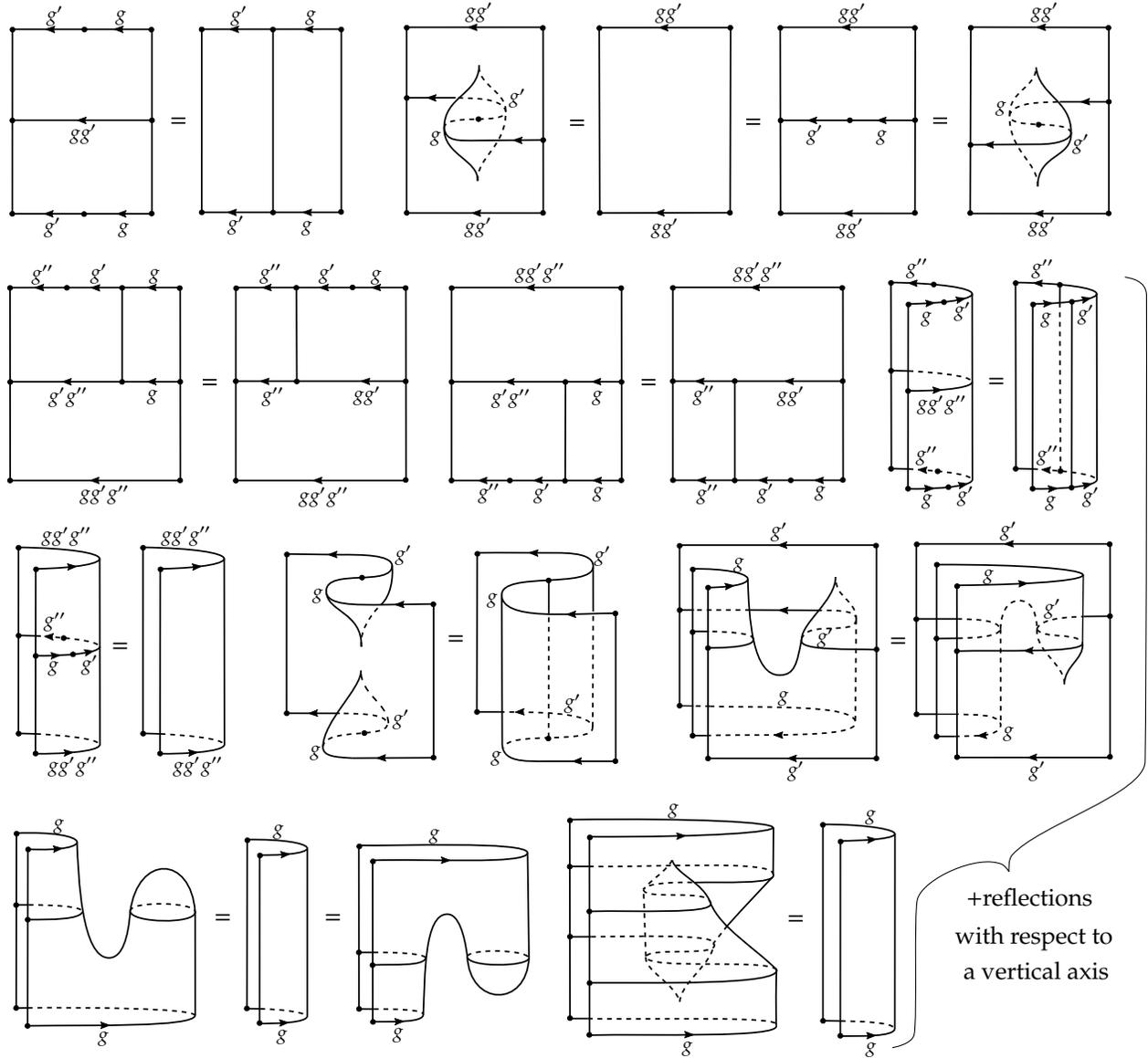
    \caption{Generating relations ($\mathcal{X}\mathcal{R}$) among $2$-morphisms}
    \label{fig:relations}
\end{figure}  
\end{theorem}

\subsection{The Cofibrancy theorem}\label{cofibrancy_section}
Knowing the equivalence $\XB^{\PD}\simeq \X\Bord_2$, the classification of $2$-dimensional extended X-HFTs mainly concerns the understanding of symmetric monoidal $2$-functors defined on $\XB^{\PD}$. The cofibrancy theorem (\cite{schommer}) is a coherence theorem for such $2$-functors. More precisely, this theorem allows replacing (weak) symmetric monoidal $2$-functors defined on computadic symmetric monoidal bicategories with their strict versions naturally. Furthermore, such strict $2$-functors are determined by the images of generating sets of a presentation subject to the relations.      

The cofibrancy theorem holds for any computadic monoidal bicategory (see \cite{schommer}, \cite{piotr}) and specifically for stricter versions of symmetric monoidal bicategories. In this section, we only focus on its version for computadic unbiased semistrict symmetric monoidal $2$-categories, which is the key step of the classification. In the following, we denote the collection of objects of a symmetric monoidal bicategory $\mathcal{C}$ by $\mathcal{C}_0$, $1$-morphisms by $\mathcal{C}_1$, and $2$-morphisms by $\mathcal{C}_2$.
\begin{definition}
Let $(\P=(\G_0,\G_1,\G_2, \mathcal{R}),s,t)$ be an unbiased semistrict symmetric presentation and $\mathcal{C}$ be a symmetric monoidal bicategory. \textit{The bicategory $\P(\mathcal{C})$ of $\P$-data in $\mathcal{C}$} is defined as follows:
\begin{itemize}
    \item The objects of $\P(\mathcal{C})$ are triples $\{(A_0,A_1,A_2) \ | \ A_i: \G_i \to \mathcal{C}_i \text{ for } i=0,1,2 \}$ of assignments satisfying following conditions. These assignments extend canonically to $BW^{\uss}(\G_0)$, $BS^{\uss}(\G_1)/\sim$, and $PG^{\uss}(\G_2)/\sim$ using the monoidal product and braiding of $\mathcal{C}$ as follows. For any $x= a_1\dots a_n \in BW^{\uss}(\G_0)$ we have $A_0(x)= A_0(a_1)\otimes \dots \otimes A_0(a_n)$, $\id_x= \id_{A_0(x)}$ and $\beta^{\sigma}_{a,\sigma(a)}$ for $\sigma \in S_n$ is given by writing $\sigma$ as a product of adjacent transpositions first and then applying the braiding of $\mathcal{C}$ to each of them. Composition and monoidal product of $1$-morphisms follow similarly extending $A_1$ to $BS^{\uss}(\G_1)/\sim$. In the extension to $PG^{\uss}(\G_2)$ interchanger $2$-morphisms $\phi$ are taken as identity and the extension to remaining $2$-morphisms uses the naturality of the braiding of $\mathcal{C}$ and the corresponding modifications. These extensions are required to be globular i.e. $A_0 \circ p = p \circ A_1$ and $A_1 \circ p= p \circ A_2$ 
    where $p$ is used for source $s$ and target $t$ maps. Among all extensions only those with $A_2(x)=A_2(y)$ for all $(x,y) \in \mathcal{R}$ are considered.  
    \item The $1$-morphisms from $A=(A_0,A_1,A_2)$ to $B=(B_0,B_1,B_2)$ are tuples $\{(\alpha_0,\alpha_1) \ | \  \alpha_i: \G_i \to \mathcal{C}_{i+1} \text{ for } i=0,1 \ s(\alpha_0(a))=A_0(a), \ t(\alpha_0(b))=B_0(b) \ \& \ \alpha_1(f):  B_1(f) \circ \alpha_0(a) \xrightarrow{\cong} \alpha_0(b) \circ A_1(f) \ \forall f :a \to b \in \G_1\}$ of assignments. These assignments extend canonically to $BW^{\uss}(\G_0)$ using monoidal product of $\mathcal{C}$ and to $BS^{\uss}(\G_1)/\sim$ using induction on the number of monoidal product and compositions as follows. For $x=a_1\ldots a_n \in BW^{\uss}(\G_0)$, $\alpha_1(\id_x)$ is the composition
    \begin{align*} 
    B_1(\id_x) \circ \alpha_0(x) = \id_{B_0(x)} \circ \alpha_0(x) \xrightarrow{\ell^{\mathcal{C}}} \alpha_0(x) \xrightarrow{(r^{\mathcal{C}})^{-1}} \alpha_0(x) \circ \id_{A_0(x)} = \alpha_0(x) \circ A_1(\id_x)
        \end{align*}
    where $\ell$ and $r$ are the left and right unitors of the underlying bicategory of $\mathcal{C}$. For $x= a_1 \dots a_n \in BW^{\uss}(\G_0)$ and $\sigma \in S_n$, the $2$-morphism $\alpha_1(\beta^{\sigma}_{x,\sigma(x)})$ is given by the components of the braiding (equivalence transformation) of $\mathcal{C}$ on $1$-morphism $\alpha_0(x)$. Next, for elements $f: a \to b$ and $f : a' \to b'$ in $BW^{\uss}(\G_1)$, $\alpha_1( f \otimes f')$ is defined as the composition
    \begin{align*}
    (B_1(f) \otimes B_1(f')) \circ (\alpha_0(a) &\otimes \alpha_0(a')) \to (B_1(f) \circ \alpha_0(a)) \otimes (B_1(f') \circ \alpha_0(a')) \xrightarrow{\alpha_1 (f) \otimes \alpha_1(f')} \\ 
    ( \alpha_0(b) \circ A_1(f)) &\otimes (\alpha_0(b') \circ A_1(f')) \to (\alpha_0(b) \otimes \alpha_0(b')) \circ (A_1(f) \otimes A_1(f')).
    \end{align*}
    For elements $f:a \to b$ and $g : b \to c$ in $BW^{\uss}(\G_1)$, $\alpha_1(g \circ f)$ is defined as the composition
    \begin{align*}
    (B_1(g) \circ  B_1(f)) \circ \alpha_0(a) \xrightarrow{\dot{a}^{\mathcal{C}}} B_1(g) \circ (B_1(f) &\circ \alpha_0(a)) \xrightarrow{\id \ast \alpha_1(f)} B_1(g) \circ (\alpha_0(b) \circ A_1(f)) \xrightarrow{(\dot{a}^{\mathcal{C}})^{-1}}\\
    (B_1(g)  \circ \alpha_0(b)) \circ A_1(f) \xrightarrow{\alpha_1(g) \ast \id}
    (\alpha_0(c) &\circ A_1(g)) \circ A_1(f) \xrightarrow{\dot{a}^{\mathcal{C}}} 
    \alpha_0(c) \circ (A_1(g) \circ A_1(f))   
    \end{align*}
    where $\dot{a}^{\mathcal{C}}$ is the associator of the underlying bicategory of $\mathcal{C}$. These assignments are also required to be natural with respect to equivalence classes of paragraphs. That is, for any $[\beta] \in PG^{\uss}(\G_2)/\sim$ with $\beta: f \to g$, we have $ ( \id_{\alpha_0(b)} \ast A_2(\beta)) \circ \alpha_1(f) = \alpha_1(g) \circ ( B_2(\beta) \ast \id_{\alpha_0(a)})$, equivalently;
    \begin{center}
    \begin{tikzcd}[column sep={4 cm,between origins}, row sep={3 cm,between origins}]
    A_0(a) \arrow[d, "\alpha_0(a)"'] \arrow[r, bend left=20, "A_1(g)"{name=C, above}] \arrow[r, bend right=20, "A_1(f)"{name=D, below}] \arrow[dr, bend left=10, ""{name=E, above}] \arrow[dr, bend right=35, ""{name=F, below}]  & A_0(b) \arrow[d,"\alpha_0(b)"] \\
    B_0(a) \arrow[r, "B_1(f)"']  & B_0(b) \arrow[Rightarrow, from=D, to=C, "A_2(\beta)"'{yshift=0ex}, start anchor={[xshift=-0.5ex, yshift=1.3ex]},end anchor={[xshift=-0.5ex,yshift=-1ex]}] \arrow[Rightarrow, from=F, to=E, "\alpha_1(f)"', start anchor={[xshift=.4ex, yshift=.4ex]},end anchor={[xshift=-0.4ex,yshift=-0.4ex]}]
\end{tikzcd} \hspace{.3cm} $=$ \hspace{.3cm}
        \begin{tikzcd}[column sep={4 cm,between origins}, row sep={3 cm,between origins}]
    A_0(a) \arrow[r, "A_1(g)"] \arrow[d, "\alpha_0(a)"'] \arrow[dr, bend left=35, ""{name=E, above}] \arrow[dr, bend right=10, ""{name=F, below}] \arrow[Rightarrow, from=F, to=E, "\alpha_1(g)"', start anchor={[xshift=.4ex, yshift=.4ex]},end anchor={[xshift=-0.4ex,yshift=-0.4ex]}] 
    & A_0(b) \arrow[d,"\alpha_0(b)"] \\
    B_0(a)  \arrow[r, bend left=20, "B_1(g)"{name=U, above}] \arrow[r, bend right=20,"B_1(f)"{name=V,below}]  & B_0(b). \arrow[Rightarrow, from=V, to=U, "B_2(\beta)"'{yshift=-0.1ex}, start anchor={[xshift=-.7ex, yshift=1.2ex]},end anchor={[xshift=-0.7ex,yshift=-.9ex]}]
\end{tikzcd}
    \end{center}
    \item The $2$-morphisms from $\alpha=(\alpha_0,\alpha_1)$ to $\beta=(\beta_0,\beta_1)$ are assignments $\{ \theta_0: \G_0 \to \mathcal{C}_2 \ | \ \theta_0(a): \alpha_0(a) \to \beta_0(a), \ \forall a \in \G_0\}$ where $\alpha,\beta: A \to B$ for $A=(A_0,A_1,A_2)$ and $B=(B_0,B_1,B_2)$. These assignments extend canonically to $BW^{\uss}(\G_0)$ using the monoidal product of $\mathcal{C}$ and they are required to be natural with respect to $BW^{\uss}(\G_1)$. That is, for any $f: a \to b \in BW^{\uss}(\G_1)$ we have $ \beta_1(f) \circ ( \theta_0(a) \ast \id_{B_1(f)})=  (\id_{A_1(f)} \ast  \theta_0(b) ) \circ \alpha_1(f)$, equivalently;
    \begin{center}
 \begin{tikzcd}[column sep={4 cm,between origins}, row sep={3 cm,between origins}]
 A_0(a) \arrow[r, "A_1(f)"] \arrow[d, bend right=45, "\alpha_0(a)\hspace{.7cm}"{name=A, below}] \arrow[d, bend left=35, "\hspace{.7cm}\beta_0(a)"{name=B, below}] \arrow[Rightarrow, from=A, to=B, "\theta_0(a)"{yshift=0.2ex}, start anchor={[xshift=-3ex, yshift=0ex]},end anchor={[xshift=3ex, yshift=0ex]}] \arrow[dr, bend left=35, ""{name=C, below}] \arrow[dr, bend right=10, ""{name=D, below}]
\arrow[Rightarrow, from=D, to=C, "\beta_1(f)"', start anchor={[xshift=.2ex, yshift=.2ex]},end anchor={[xshift=-0.4ex,yshift=-0.4ex]}] 
 & A_0(b) \arrow[d, "\beta_0(b)"] \\B_0(a) \arrow[r, "B_1(f)"'] & B_0(b)
 \end{tikzcd}\hspace{.3cm}$=$ \hspace{.3cm}
 \begin{tikzcd}[column sep={4 cm,between origins}, row sep={3 cm,between origins}]
 A_0(a) \arrow[r, "A_1(f)"] \arrow[d, "\alpha_0(a)"'] \arrow[dr, bend left=10, ""{name=C,below}] \arrow[dr, bend right=35, ""{name=D,above}] \arrow[Rightarrow, from=D, to=C, "\alpha_1(f)"'] & A_0(b) \arrow[d, bend right=35, "\alpha_0(b)\hspace{.7cm}"{name=A, above}] \arrow[d, bend left=45, "\hspace{0.7cm}\beta_0(b)"{name=B, above}] \arrow[Rightarrow, from=A, to=B, "\theta_0(b)"{yshift=0.2ex}, start anchor={[xshift=-3ex, yshift=0ex]},end anchor={[xshift=3ex, yshift=0ex]}]  \\ B_0(a) \arrow[r, "B_1(f)"'] & B_0(b).
 \end{tikzcd}
 \end{center}
\end{itemize}
\end{definition}
With the trivial coherence data $(\chi,\iota, \mathcal{W}, \mathcal{G}, \mathcal{R},\mathcal{U})$, any object of $\P(\mathcal{C})$ considered with the extensions gives rise to a strict symmetric monoidal $2$-functor $\mathsf{F}_{\uss}(\P) \to \mathcal{C}$ (see \cite{schommer}, \cite{tezim}). Similarly, any $1$-morphism and $2$-morphism of $\P(\mathcal{C})$ considered with their extensions yield a strict\footnote{Those monoidal transformations whose monoidal structure $1$- and $2$-morphisms are identities.} symmetric monoidal transformation and a symmetric monoidal modification, respectively (see Section 2.3 in \cite{schommer}).
\begin{definition}
Let $\mathcal{C}$ be a symmetric monoidal bicategory, $\mathsf{F}_{\uss}(\P)$ be a computadic unbiased semistrict symmetric monoidal $2$-category for a given presentation $\P=(\mathcal{G}_0,\mathcal{G}_1,\mathcal{G}_2,\mathcal{R})$. The bicategory $\text{SymMon}(\mathsf{F}_{\uss}(\P),\mathcal{C})$ has symmetric monoidal $2$-functors as objects, symmetric monoidal transformations as $1$-morphisms, and symmetric monoidal modifications as $2$-morphisms.
\end{definition}
By construction of $\P(\mathcal{C})$ we have a strict inclusion functor $\imath: \P(\mathcal{C})\hookrightarrow \text{SymMon}(\mathsf{F}_{\uss}(\P),\mathcal{C})$ given by the associated $2$-functors, transformations, and modifications. The cofibrancy theorem below states that $\imath$ is an equivalence of bicategories.
\begin{theorem}[Cofibrancy Theorem, Theorem 2.78, \cite{schommer}]\label{cofibrancy}
Let $\mathcal{C}$ be a symmetric monoidal bicategory and let $\mathsf{F}_{\uss}(\P)$ be a computadic unbiased semistrict symmetric monoidal $2$-category constructed from an unbiased semistrict presentation $\P=(\G_0,\G_1,\G_2,\mathcal{R})$. Then, the inclusion $\imath: \P(\mathcal{C})\hookrightarrow \text{SymMon}(\mathsf{F}_{\uss}(\P),\mathcal{C})$  is an equivalence of bicategories.
\end{theorem}

The following lemma taken from \cite{schommer} (Lemma 2.15) implies that for any symmetric monoidal bicategory $\mathcal{C}$, the bicategories $\text{SymMon}(\X\Bord_2 ,\mathcal{C})$ and $\text{SymMon}(\XB^{\PD},\mathcal{C})$ of symmetric monoidal $2$-functors, symmetric monoidal transformations, and modifications are equivalent.
\begin{lemma}[\cite{schommer}]\label{whiskering}
Let $\mathcal{M}$ and $\mathcal{M}'$ be symmetric monoidal bicategories, and $H: \mathcal{M} \to \mathcal{M}'$ be a symmetric monoidal $2$-functor, which is an equivalence. Then there is a canonical equivalence $H^{\ast}: \text{SymMon}(M', B) \to \text{SymMon}(M,B)$ of bicategories given by composition on the level of objects, and by symmetric monoidal whiskering on the level of $1$- and $2$-morphisms.
\end{lemma}
We denote the bicategory $\text{SymMon}(\X\Bord_2,\mathcal{C})$ by $\mathcal{E}$-$\mathcal{HFT}(X,\mathcal{C})$ and state the classification of $2$-dimensional extended HFTs with target $X\simeq K(G,1)$ as follows. 
\begin{theorem}\label{genelsinif}
\label{thm:genelsinif}
Let $\XP$ be the presentation of $\XB^{\PD}$ given in Theorem \ref{computadic}. Then for any symmetric monoidal bicategory $\mathcal{C}$, there is an equivalence of bicategories $ \mathcal{E}\text{-}\mathcal{HFT}(X,\mathcal{C}) \simeq \XP(\mathcal{C})$.
\end{theorem}
\begin{proof}
Theorem \ref{computadic} gives a presentation $\XP$ of $\XB^{\PD}$ as a computadic unbiased semistrict symmetric monoidal $2$-category. By the cofibrancy theorem we have $\text{SymMon}(\XB^{\PD},\mathcal{C})\simeq \XP(\mathcal{C})$. Using the symmetric monoidal equivalence between $\XB^{\PD}$ and $\X\Bord_2$ in Proposition $\ref{thm1}$ and Lemma \ref{whiskering}, we obtain the result.
\end{proof}
\subsection{$\Alg$-valued 2-dimensional extended X-HFTs}\label{alg_or_sinif}
Every $2$-dimensional extended X-HFT gives a nonextended one by restricting to oriented X-circles and X-cobordisms between them. A natural question is how the classification of $2$-dimensional extended HFTs is related to Turaev's classification of $2$-dimensional HFTs by crossed Frobenius $G$-algebras (\cite{ilkhqft}). To understand this relation we study extended HFTs taking values in $\Alg$ which has $\Bbbk$-algebras as objects, bimodules as $1$-morphisms, and bimodule maps as $2$-morphisms for a commutative ring $\Bbbk$ with unity. 

The symmetric monoidal structure of $\Alg$ is given by tensoring over $\Bbbk$. We denote $(E,C)$-bimodule $D$ by ${}_{E}D_C$ and omit the symbol $\Bbbk$ when either $C$ or $E$ is $\Bbbk$. We regard ${}_{E}D_C$ as a $1$-morphism from $C$ to $E$ which is in line with the composition in $\X\Bord_2$ (see Figure \ref{fig:Sdecomp}). Composition of $1$-morphisms ${}_{E}D_C$ and ${}_{C}B_A$ is the bimodule ${}_E (D\otimes_C B)_A$. 

Before studying $\Alg$-valued $2$-dimensional extended X-HFTs, we recall necessary algebraic notions and introduce quasi-biangular $G$-algebras. Recall that a $G$-algebra over a commutative ring $\Bbbk$ is an associative $\Bbbk$-algebra $K$ equipped with a decomposition $K=\oplus_{g \in G} K_g$ such that $K_g K_h \subseteq K_{gh}$ for any $g,h \in G$. In this case, $K_e$ is the \textit{principal component} and $K$ is called \textit{strongly graded} if  $K_g K_{h} = K_{gh}$ or equivalently the natural map $K_g \otimes_{K_e} K_h \to K_{gh}$ is an isomorphism for all $g,h \in G$. The opposite $G$-algebra of $K$ is defined as $K^{op}= \oplus_{g \in G}K_{g^{-1}}$ where the order of multiplication is reversed.

\begin{definition} (\cite{tur-hqft})\label{frobG}
Let $K= \oplus_{g \in G}K_g$ be a $G$-algebra over a commutative ring $\Bbbk$. Recall that an \textit{inner product} on $K$ is a symmetric bilinear form $\eta : K \otimes K \to \Bbbk$ satisfying $\eta(ab,c)= \eta(a,bc)$ for any $a,b,c \in K$ such that $\eta|_{K_g \otimes K_h}$ is nondegenerate when $gh=e$ and zero otherwise. A \textit{Frobenius $G$-algebra} is a $G$-algebra $K$ with an inner product $\eta$ and components of $K$ are finitely generated projective $\Bbbk$-modules.
\end{definition}
Let $(K=\oplus_{g \in G}K_g,\eta)$ be a Frobenius $G$-algebra over $\Bbbk$. Each nondegenerate form $\eta|_{K_g \otimes K_{g^{-1}}}$ yields an element $\eta_g^{-}= \sum_{i \in I_g} p_i^g \otimes q_i^{g} \in K_g \otimes K_{g^{-1}}$, called an \textit{inner product element}, where $I_g$ is finite and $\eta_g^-$ is characterized by $a = \sum_{i \in I_g} \eta \big(a,q_i^{g}\big)p_i^g$ for any $a \in K_g$. Since $\eta$ is symmetric we have $\sum_i p_i^{g^{-1}} \otimes q_i^{g^{-1}} = \sum_i q_i^g \otimes p_i^g$ for all $g \in G$.

Recall that an associative $\Bbbk$-algebra $A$ is \textit{separable} if there exists an element $a= \sum_{i=1}^n p_i \otimes q_i \in A \otimes_{\Bbbk} A^{op}$ called \textit{separability idempotent} such that $\sum_{i=1}^n p_i q_i =1$ and $ab=ba$ for all $b \in A$. A separable algebra $A$ is called \textit{strongly separable} if the separability idempotent is symmetric i.e. $a= \sum_{i=1}^n p_i \otimes q_i = \sum_{i=1}^n q_i \otimes p_i$. 
\begin{lemma}\label{label2}
Let $(K = \oplus_{g \in G} K_g, \eta)$ be a Frobenius $G$-algebra with inner product elements $\big\{\eta_g^- = \sum_i p_i^g \otimes q_i^{g}\big\}_{g \in G}$ and $z \in K_e$. Then, for any $g,h \in G$ and $b \in K_{g^{-1}}$ we have 
\begin{equation}\label{eq:1}
    \sum_i p_i^h \otimes zq_i^{h} b = \sum_j bp_j^{gh} \otimes zq_j^{gh}.
\end{equation}
In particular, for any $b \in K$ and $c \in K_{h^{-1}}$ we have $\sum_j p_j^g b zq_j^g c= \sum_k c p_k^{hg}bzq_k^{hg}$.
\end{lemma}
\begin{proof}
Since both sides belong to $K_h \otimes K_{h^{-1}g^{-1}}$ it is enough to check that they give the same functionals on the dual $\Bbbk$-module $K_{h^{-1}} \otimes K_{gh}$. For any $x \in K_{h^{-1}}$ and $y \in K_{gh}$ applying $x \otimes y$ to the left hand side of equation (\ref{eq:1}) and using cyclic symmetry property of $\eta$ we obtain
\begin{align*}
    \sum_i \eta \big( p_i^h ,x \big)\eta \big(zq_i^{h}b,y\big)= \sum_i \eta \big(x,p_i^h \big) \eta \big(q_i^{h},byz \big)= \eta \bigg(x, \sum_i \eta(byz,q_i^{h})p_i^h \bigg)= \eta(x,byz).
\end{align*}
Similarly applying $x \otimes y$ to the right hand side of the equation (\ref{eq:1}) we have
\begin{align*}
    \sum_j \eta \big(bp_j^{gh},x \big) \eta \big(zq_j^{gh},y \big)= \sum_j\eta \big (xb,p_j^{gh} \big)\eta \big(q_j^{gh},yz \big)=\eta\bigg(xb, \sum_j \eta (yz,q_j^{gh})p_j^{gh}\bigg)=\eta(xb,yz).%=\eta(x, by)
\end{align*}
\end{proof}
We generalize biangular $G$-algebras which were introduced by Turaev \cite{ilkhqft} as follows.
\begin{definition}\label{genbiangular}
A strongly graded Frobenius $G$-algebra $(K, \eta)$ is called \textit{quasi-biangular} if there exists a central element $z \in K_e$, i.e. $za=az$ for all $a \in K_e$, such that for the collection of inner product elements $\big\{\sum_i p_i^g \otimes q_i^g\big\}_{g \in G}$ equations $\sum_i p_i^g z q_i^g=1$ hold for all $g \in G$. 
\end{definition}
\begin{remark}
By the Lemma \ref{label2} the principal component of a quasi-biangular $G$-algebra is a separable algebra with separability idempotent $\sum_i p_i^e \otimes z q_i^e$. A biangular $G$-algebra is a quasi-biangular $G$-algebra with $z=1$. Similarly, the principal component of a biangular $G$-algebra is strongly separable. 
\end{remark} 
One way of studying an algebra is to study the category of modules over it. Recall that Morita equivalence of algebras is the equivalence of categories of modules. In the case of a graded algebra one studies the category of graded modules. An equivalence of such categories is called a graded Morita equivalence (see Theorem 3.2 in \cite{gmorita}) which was introduced by Boisen \cite{gmorita} as follows.
\begin{definition} (\cite{gmorita})\label{gradedMorita}
A $G$-\textit{graded Morita equivalence} $\zeta$ between $G$-algebras $K=\oplus_{g \in G}K_g$ and $L= \oplus_{g \in G}L_g$ is a quadruple $({}_{L}U_K, {}_{K}V_L, \tau, \mu)$ where ${}_{L}U_K= \oplus_{g \in G} U_g$ is a graded $(L,K)$-bimodule that is $L_g U_h K_{g'}\subset U_{ghg'}$, ${}_{K}V_L= \oplus_{g \in G} V_g$ is a graded $(K,L)$-bimodule, and $\tau :{}_{K}K_{K} \to {}_{K}V \otimes_L U_K$ and $\mu: {}_{L}U \otimes_K V_L \to {}_{L}L_L$ are graded $(K,K)$ and $(L,L)$ bimodule maps respectively such that the following compositions 
\begin{align*}
    {}_{L}U_K& \xrightarrow{} {}_{L}U \otimes_K K_K \xrightarrow{\text{id} \otimes \tau} {}_{L}U \otimes_K(V \otimes_L U_K) \xrightarrow{} ({}_{L}U \otimes_K V) \otimes_L U_K \xrightarrow{\mu \otimes \text{id}} {}_{L}L \otimes_L U_K \xrightarrow{} {}_{L}U_K \\
    {}_{K}V_L& \xrightarrow{} {}_{K}K \otimes_K V_L \xrightarrow{\tau \otimes \text{id}}({}_{K}V \otimes_L U) \otimes_K V_L \xrightarrow{} {}_{K}V \otimes_L (U \otimes_K V_L) \xrightarrow{\text{id} \otimes \mu} {}_{K}V \otimes_L L_L \xrightarrow{} {}_{K}V_L
\end{align*}
are $id_U$ and $id_V$ respectively. When $\tau$ and $\mu$ are invertible as $G$-graded bimodule maps, $\zeta$ is called a $G$-\textit{graded Morita context}.
\end{definition}
\begin{definition}\label{equivofGmorita}
Let $\zeta=({}_{L}U_K, {}_{K}V_L, \tau, \mu)$ and $\zeta'=({}_{L}U'_K, {}_{K}V'_L, \tau', \mu')$ be two $G$-graded Morita equivalences. An \textit{equivalence} of $G$-graded Morita equivalences $\zeta$ and $\zeta'$ is a $G$-graded bimodule maps $\xi : {}_{L}U_K \to {}_{L}U'_K$ and $\rho: {}_{K}V_L \to {}_{K}V'_L$ such that $\mu= \mu' \circ (\xi \otimes \rho)$ and $\tau'= (\rho \otimes \xi) \circ \tau$.
\end{definition} 
\begin{lemma}(\cite{gradedmorita1})
Assume that $G$-algebras $K=\oplus_{g \in G}K_g$ and $L=\oplus_{g \in G}L_g$ are $G$-graded Morita equivalent. Then, if $K$ is strongly graded then $L$ is also strongly graded.
\end{lemma}
Next, we transfer the inner product of one Frobenius $G$-algebra to another using a graded Morita context between them. As the first step we recall the following lemma. 
\begin{lemma}(\cite{schommer}, \cite{homotopy_fixed_pts_on_bicats})
Any Morita context $\zeta=({}_{L}U_K,{}_{K}V_L,\tau,\eta)$ between $\Bbbk$-algebras $K$ and $L$ induces a canonical isomorphism of $\Bbbk$-modules $\zeta_{\ast}: K/[K,K] \to L/[L,L]$.
\end{lemma}
An explicit formula for the isomorphism $\zeta_{\ast}$ in the lemma is provided in \cite{homotopy_fixed_pts_on_bicats}. The inner product $\eta$ of a Frobenius $G$-algebra $(K,\eta)$ is determined at its principal component by $\eta(a,b\cdot 1)=\eta(ab,1)$.  This allows us to denote $(K,\eta)$ by $(K,\Lambda)$ where $\Lambda: K_e \to \Bbbk$ is a nondegenerate trace. Since $\eta$ is symmetric $\Lambda$ factors through $K_e/[K_e,K_e]$. Lemma \ref{label2} implies that for a symmetric Frobenius algebra $(K_e,\Lambda_e)$, an inner product element $\sum_i p_i^e \otimes q_i^e$ can be considered as the image of $1 \otimes 1$ under a bimodule map $\xi:{}_{K_e^1}(K_e)_{K_e^2} \otimes {}_{K^3_e}(K_e)_{K_e^4} \to {}_{K_e^1}(K_e)_{K_e^4} \otimes {}_{K_e^3}(K_e)_{K_e^2}$ where numbers indicate module actions i.e. $K^i_e= K_e$ for $i=1,2,3,4$. In the case of a quasi-biangular $G$-algebra $(K=\oplus_{g \in G}K_g,\Lambda)$ inner product elements $\big\{\sum_i p_i^g \otimes q_i^g\big\}_{g \in G \backslash \{e\}}$ are the image of $1 \otimes 1$ under the following composition 
\begin{align}\label{eqninnpd}
    {}_{K_e^1}(K_e)_{K_e^2} \otimes {}_{K_e^3}(K_e)_{K_e^4} \to {}_{K^1_e}(K_e)_{K_e^2} \otimes {}_{K_e^3} (K_{g^{-1}} \otimes_{K_e} K_e \otimes_{K_e} K_{g})_{K_e^4} \to 
    \end{align}
    \begin{align*}
    \to {}_{K_e^1}(K_e \otimes_{K_e} K_g)_{K_e^4} \otimes {}_{K_e^3}(K_{g^{-1}}\otimes_{K_e} K_e)_{K_e^2} \to {}_{K_e^1}(K_g)_{K_e^4} \otimes {}_{K_e^3}(K_{g^{-1}})_{K_e^2}
\end{align*}
where the second homomorphism is identity on $K_{g^{-1}}$ and $K_g$, and $\xi$ on $K_e \otimes K_e$. In the following, we consider inner product elements as the images of $1\otimes 1$ under the above bimodule maps.
\begin{definition}\label{compmorita}
Let $(K,\Lambda_K)$ and $(L, \Lambda_L)$ be quasi-biangular $G$-algebras over $\Bbbk$ with collections of inner product elements $\{\eta_g^K\}_{g \in G}$ and $\{\eta_g^L\}_{g \in G}$ respectively. A $G$-graded Morita context $\zeta=({}_{L}U_K, {}_{K}V_L, \tau, \mu)$ between $K$ and $L$ is said to be \textit{compatible} if $\Lambda_L=(\zeta_{\{e\}})_{\ast}\Lambda_K$ and $\eta^L_g=(\zeta_{\{e\}})_{\ast}(\eta^K_g)$ for all $g \in G$ where $(\zeta_{\{e\} })_{\ast}(\eta^K_g)$ consists of inner product elements for $(L,(\zeta_{e})_{\ast}\Lambda_K)$ given by $\xi'(1 \otimes 1)$ under the commutative diagram
\begin{align*}
    \xymatrix{{}_{L_e}(U_e \otimes K_e \otimes V_e)_{L_e} \otimes_{\Bbbk} {}_{L_e} (U_e \otimes K_e \otimes V_e)_{L_e} \ar[r]^{\text{id} \otimes \xi \otimes \text{id}} \ar[d]_{\mu_{\{e\}}} &  {}_{L_e}(U_e \otimes K_e \otimes V_e)_{L_e} \otimes_{\Bbbk} {}_{L_e} (U_e \otimes K_e \otimes V_e)_{L_e} \ar[d]^{\mu_{\{e\}}} \\ {}_{L_e}(L_e)_{L_e} \otimes_{\Bbbk} {}_{L_e}(L_e)_{L_e} \ar[r]_{\xi'} & {}_{L_e}(L_e)_{L_e} \otimes_{\Bbbk} {}_{L_e}(L_e)_{L_e}.
    }
\end{align*}
Remaining inner product elements are obtained from $\xi'$ as described above.
\end{definition} 
\begin{theorem}\label{orclass}
\label{thm:orclass}
Any $\Alg$-valued $2$-dimensional extended X-HFT $Z: \X\Bord_2 \to \Alg$ whose precomposition $\XB^{\PD} \xrightarrow{\simeq} \X\Bord_2 \xrightarrow{Z} \Alg$ gives a strict symmetric monoidal $2$-functor determines a triple $(A,B,\zeta)$ where $A$ and $B$ are quasi-biangular $G$-algebras, and $\zeta$ is a compatible $G$-graded Morita context between $A$ and $B^{\op}$. Conversely, for any such triple $(A,B,\zeta)$ there exists an $\Alg$-valued $2$-dimensional extended X-HFT.
\end{theorem} 
\begin{proof}
Let $Z: \X\Bord_2 \to \Alg$ be such a $2$-dimensional extended HFT. The cofibrancy theorem implies that there exists an object $Z'$ in $\XP(\Alg)$ such that $\imath(Z')$ is the composition $\XB^{\PD} \xrightarrow{\simeq} \X\Bord_2 \xrightarrow{Z} \Alg$ where $\imath: \XP(\Alg)\to \text{SymMon}(\XB^{\PD}, \Alg)$ is the equivalence of bicategories.

We have $\Bbbk$-algebras $ Z'(\nokta) =A_e$ and $Z'(\noktab)=B_e$ in $Z'_0(\mathcal{XG}_0)$ corresponding to two generating objects of $\XP$. There are four types of generating $1$-morphisms and each is indexed by the elements of $G$. For every $g \in G$ they give the following bimodules in $Z'_1(\mathcal{XG}_1)$ 
\begin{align*}
    Z'\big(\arti \big) &= A_g \quad (A_e ,A_e)\text{-bimodule}\\
    Z'\big(\eksi \big) &= B_g \quad (B_e,B_e)\text{-bimodule}\\
    Z'\big(\leftelbow \big) &= M_g \quad (B_e \otimes_{\Bbbk} A_e ,\Bbbk)\text{-bimodule} \\
    Z'\big(\rightelbow \big) &= N_g \quad (\Bbbk, A_e \otimes_{\Bbbk} B_e)\text{-bimodule}.
    \end{align*}
The first $2$-morphism in Figure \ref{fig:Gmod} defines a $G$-graded product on the bimodule $\oplus_{g \in G}A_g$. Associativity of this product is the obvious relation in Figure \ref{fig:relations}. Denote the corresponding $G$-algebra by $A= \oplus_{g \in G} A_g$. The first relation in Figure \ref{fig:relations} shows that the bimodule map
\begin{align}\label{eqrank1}
    {}_{A_e}(A_{g'}) \otimes_{A_e} (A_{g})_{A_e} &\xrightarrow{\cong} {}_{A_e}(A_{gg'})_{A_e}
    \end{align}
is invertible for all $g,g' \in G$. Since multiplication of $G$-algebra $A$ is defined using (\ref{eqrank1}) we have $A_g A_{g'} = A_{gg'}$ for all $g,g' \in G$ i.e. $A$ is strongly graded. Similar arguments for $(B_e,B_e)$-bimodules $\{ B_g\}_{g \in G}$ yield another strongly graded $G$-algebra $B= \oplus_{g \in G}B_g$. 

Using the opposite algebra we can turn algebra actions on bimodules around. More precisely, a left $B_e$ action on ${}_{A_e \otimes B_e}M_g$ can be turned into a right $B_e^{op}$ action and the right $B_e$ module action on $(N_g)_{A_e \otimes B_e}$ can be turned into a left $B^{op}_e$ action.
%\begin{figure}
%    \centering
%    \includesvg{Gmod2}
%    \caption{Part of generators and relations giving $G$-algebra and $G$-graded module}
%    \label{fig:Gmod}
%\end{figure}
\begin{figure}
    \centering
    \def\svgwidth{\columnwidth}
    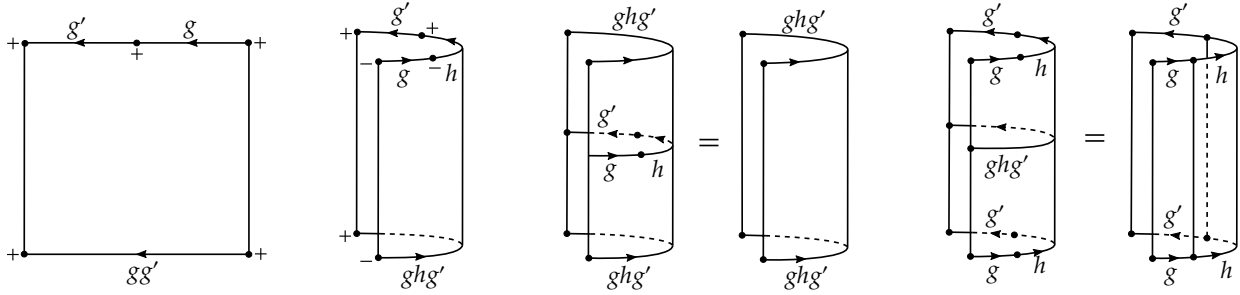
    \caption{Part of generators and relations giving $G$-algebra and $G$-graded module}
    \label{fig:Gmod}
\end{figure} 
The second $2$-morphism in Figure \ref{fig:Gmod} gives  
\begin{align}
      {}_{B_e \otimes A_e}( B_{g^{-1}}\otimes_{\Bbbk}A_{g'}) \otimes_{B_e \otimes A_e}M_h \xrightarrow{\cong} {}_{B_e \otimes A_e}M_{g^{-1}hg'}. 
\end{align}
Turning $B_e$ actions on $B_g$ around gives ${}_{A_e}(M_{ghg'})_{B_e^{op}}$ and the collection of all such bimodule maps turns $\{ {}_{A_e}(M_g)_{B_e^{\op}} \}_{g \in G}$ into a $G$-graded $(A,B^{\op})$-bimodule $M = \oplus_{g \in G} M_g$. Similarly, reflections of this $2$-morphism and corresponding relations with respect to a vertical axis show that $N = \oplus_{g \in G} N_g$ is a $G$-graded $(B^{op},A)$-bimodule. 

There are four types of cusp generators and each is indexed by two elements of $G$. For every $g,g' \in G$ they give the following bimodule maps in $Z'_2(\mathcal{XG}_2)$ 
\begin{align*}
    f_1^{gg'} &: {}_{A_e}(A_{gg'})_{A_e} \to {}_{A_e}M_g \otimes_{B_e^{\text{op}}} (N_{g'})_{A_e} &
    f_2^{gg'} &: {}_{B_e^{\text{op}}}N_{g} \otimes_{A_e} (M_{g'})_{B_e^{\text{op}}} \to {}_{B_e^{\text{op}}}(B^{op}_{gg'})_{B_e^{\text{op}}} \\
    f_3^{gg'} &: {}_{B_e^{\text{op}}}(B^{op}_{gg'})_{B_e^{\text{op}}} \to {}_{B_e^{\text{op}}}N_g \otimes_{A_e} (M_{g'})_{B_e^{\text{op}}} &    f_4^{gg'} &: {}_{A_e}M_g \otimes_{B_e^{\text{op}}}  (N_{g'})_{A_e} \to {}_{A_e}(A_{gg'})_{A_e}
\end{align*}
given in the order of cusp generators in Figure \ref{fig:generators}. These bimodule maps are required to satisfy relations in $\mathcal{XR}$. Relations containing cusp generators indicate that these bimodule maps are both sided inverses i.e. $f_1^{gg'}=\big(f_4^{gg'}\big)^{-1}$ and $f_2^{gg'}=\big(f_3^{gg'}\big)^{-1}$. It is not hard to see that for each $i$ the collection $\big\{f_i^{gg'}\big\}_{g,g' \in G}$ of bimodule maps forms a $G$-graded bimodule map $f_i$. The collection of swallowtail morphisms corresponds to following compositions of graded bimodule maps 
\begin{align*}
    {}_{B^{op}}N_A \xrightarrow {}_{B^{op}}N \otimes_A A_A \xrightarrow{id \otimes f_1} {}_{B^{op}}N \otimes_A M \otimes_{B^{op}} N_A \xrightarrow{f_2 \otimes id} {}_{B^{op}} B^{op} \otimes_{B^{op}} N_A \xrightarrow {}_{B^{op}}N_A \\
    {}_{A}M_{B^{op}} \xrightarrow {}_{A}A \otimes_A M_{B^{op}} \xrightarrow{f_1 \otimes id} {}_{A}M \otimes_{B^{op}} N \otimes_A M_{B^{op}} \xrightarrow{id \otimes f_2} {}_{A}M \otimes_{B^{op}} B^{op}_{B^{op}} \xrightarrow {}_{A}M_{B^{op}}.
\end{align*}
Swallowtail relations imply that both compositions equal to identity bimodule maps of $N$ and $M$ respectively. In other words, $\zeta= ({}_{B^{op}}N_A ,{}_{A}M_{B^{op}}, f_1, f_2)$ is a $G$-graded Morita context. Using $\zeta$ we can replace $B^{\op}$-module actions with $A$-module actions as follows. Right (left) $B^{\op}$-module can be turned into a right (left) $A$-module by tensoring with ${}_{B^{\op}}(N)_A$ (${}_{A}M_{B^{\op}}$) such as tensoring ${}_{A}M_{B^{\op}}$ with ${}_{B^{\op}}(N)_A$ yields ${}_AM \otimes_{B^{\op}} N_A$ which is isomorphic to ${}_{A}A_A$ via $f_4$.

Remaining generators are Morse generators consisting of saddles, cup, and cap $2$-morphisms. The collection of bimodule maps in $Z'_2(\mathcal{XG}_2)$ for the first saddle morphism in Figure \ref{fig:generators} yields a graded bimodule map of the form
\begin{align*}
      {}_{A \otimes B}M \otimes_{\Bbbk} N_{A \otimes B} \to {}_{A \otimes B}(A \otimes_{\Bbbk} B)_{A\otimes B}
\end{align*}
by turning the $B$-module actions around we obtain $ {}_{A \otimes B^{\op}}(M \otimes_{\Bbbk} N)_{A \otimes B^{\op}} \to {}_{A \otimes B^{\op}}( A \otimes_{\Bbbk} B)_{A \otimes B^{\op}}$ where the left $B^{\op}$-action is on $N$ and the right $B^{\op}$-action is on $M$. 
%\begin{figure}[ht]
%    \centering
%    \includesvg[]{innpd_1}
%    \caption{Saddle morphisms and cusp flip relation}
%    \label{fig:innpd}
%\end{figure}
\begin{figure}[ht]
    \centering
    \def\svgwidth{\columnwidth}
    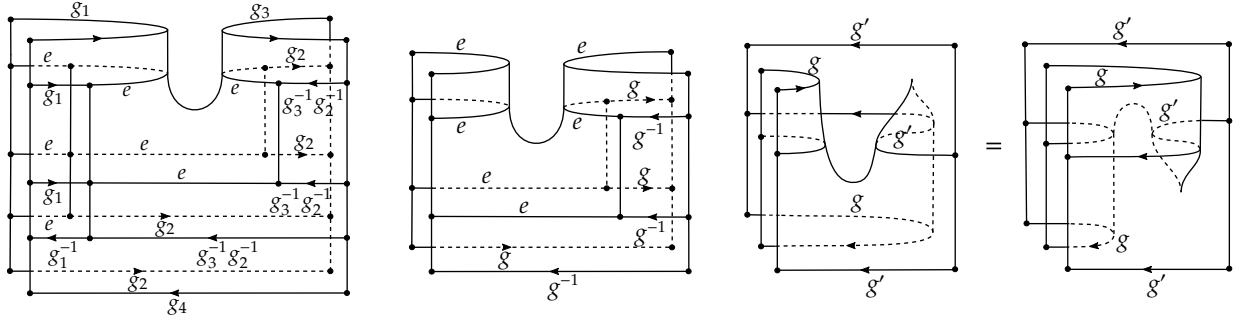
    \caption{Saddle morphisms and cusp flip relation}
    \label{fig:innpd}
\end{figure}
As pointed out above $B^{\op}$-module actions can be replaced by $A$-module actions and we get a graded $(A_1 \otimes A_3,A_2 \otimes A_4)$ bimodule map of the form
\begin{align*}
 \xi:  {}_{A_1} A_{A_2} \otimes_{\Bbbk} {}_{A_3}A_{A_4} &\to {}_{A_1} A_{A_4} \otimes_{\Bbbk} {}_{A_3}A_{A_2}
\end{align*}
where numbers indicate module actions i.e. $A_i=A$ for $i=1,2,3,4$. The first morphism in Figure \ref{fig:innpd} shows that an arbitrary saddle morphism can be obtained from $e$-labeled saddle morphism and other generating $2$-morphisms. This implies that the graded bimodule map $\xi$ is determined at $1 \otimes 1 \in A_e \otimes A_e$ which we denote by a finite sum $\sum_i p_i^e \otimes q_i^{e}$ and it satisfies $\sum_i a p_i^e \otimes q_i^e = \sum_i p_i^e \otimes  q_i^e a$ for all $a \in A$. Similarly, we denote the image of $1 \otimes 1$ under the second $2$-morphism in Figure \ref{fig:innpd} by $\eta_g^A =\sum_i p_i^g \otimes q_i^g$ for all $g \in G$ (compare with equation (\ref{eqninnpd})). 

In the same way, the collection of bimodule maps in $Z'_2(\mathcal{XG}_2)$ for the second saddle morphism gives a graded $(A_1 \otimes A_3, A_2 \otimes A_4)$ bimodule map of the form $\eta : {}_{A_1} A_{A_2} \otimes_{\Bbbk} {}_{A_3} A_{A_4} \to {}_{A_1} A_{A_4} \otimes_{\Bbbk} {}_{A_3}A_{A_2}$. The cusp flip relation shown in Figure \ref{fig:innpd} implies that $\xi = \eta$. 
%\begin{figure}[ht]
%    \centering
%    \includesvg{Gcup2}
%    \caption{Cup and cap morphisms on nonprincipal components}
%    \label{fig:Gcup}
%    \end{figure}
\begin{figure}[ht]
    \centering
    \def\svgwidth{\columnwidth}
    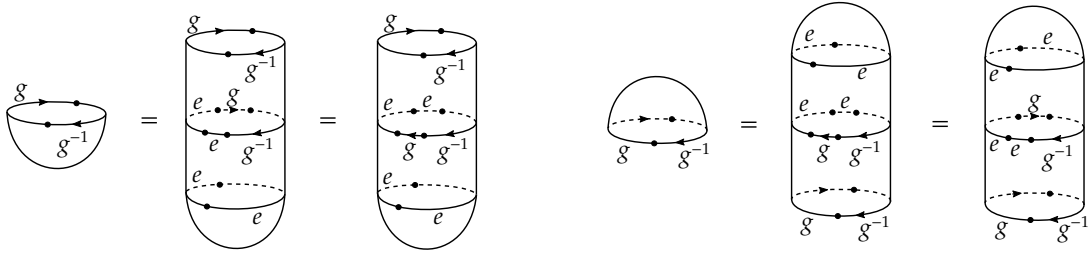
    \caption{Cup and cap morphisms on nonprincipal components}
    \label{fig:Gcup}
\end{figure}
Before considering cup and cap generators note that using $\zeta$ we can assign the collection of all $g,g^{-1}$ labeled circles to $\oplus_{g \in G} A_g \otimes_{ (A_e \otimes A_{e}^{\op})} A_{g^{-1}}$. The collections of $2$-morphisms in Figure \ref{fig:Gcup} give the following bimodule maps 
\begin{align*}
     \Lambda &: \oplus_{g \in G} A_g \otimes_{A_e \otimes A_e^{\op}} A_{g^{-1}} \to \Bbbk\\
     u&: \Bbbk \to  \oplus_{g \in G}A_g \otimes_{A_e \otimes A_e^{\op}} A_{g^{-1}}
\end{align*}
respectively. Figure \ref{fig:Gcup} implies that cup and cap morphisms are determined on the principal component. Since $A_e \otimes_{A_e \otimes A_e^{op}}A_e= A_e/[A_e,A_e]$, cup morphism on the principal component can be considered as a symmetric linear map $\Lambda :A_e \to \Bbbk$. Additionally, Figure \ref{fig:Gcup} shows that on nonprincipal components cup morphism is given by multiplication followed by $\Lambda$ leading to a symmetric $\Bbbk$-bilinear map $\eta_g: A_g \otimes A_{g^{-1}} \to \Bbbk$. Morse relations involving cup morphism indicates the nondegeneracy of $\eta_g$ as follows. Assuming $\sum_j \beta_g^j \otimes 1 \otimes \beta_{g^{-1}}^j$ as the image of $1$ under $A_e \xrightarrow{\sim} A_{g} \otimes_{A_e} A_e \otimes_{A_e} A_{g^{-1}}$ the first (left) $2$-morphism in Figure \ref{fig:nsg} corresponds to following compositions 
\begin{align*}
    a \to 1 \otimes 1 \otimes a \to \sum_j \beta_g^j \otimes \bigg(\sum_i p_i^e \otimes q_i^e \bigg) \otimes \beta_{g^{-1}}^j \otimes a \to \sum_i p_i^g \otimes q_i^{g} \otimes a \to \sum_i p_i^g \eta \big(q_i^{g},a \big)
\end{align*}
and Morse relation implies that it is equivalent to $\text{id}_{A_g}$. Similarly, reflection of this morphism with $g^{-1}$ label gives $b= \sum_i \eta_g(b, p_i^{g})q_i^{g}$ for any $b \in A_{g^{-1}}$, which shows that $\eta_g$ is nondegenerate. Thus, $(A,\eta_A)$ is a Frobenius $G$-algebra where $(\eta_A)|_{A_g \otimes A_{h}}$ is $\eta_g$ when $h=g^{-1}$ and zero otherwise. 

Remaining Morse relations contain cap morphisms which are determined on the principal component. For any $c \in A_e$, assuming $u(1)|_{A_e \otimes A_e}= \sum_j a_j \otimes b_j$ the second $2$-morphism in Figure \ref{fig:nsg} corresponds to the following compositions
\begin{align*}
    c \otimes \sum_j a_j \otimes b_j \to \sum_{i,j} c p_i^e \otimes a_j q_i^e \otimes b_j \to \sum_{i,j} c p_i^e b_j a_j q_i^e \to c \sum_i p_i^e z q_i^e  
\end{align*}
where $z= \sum_j b_j a_j \in A_e$. Morse relation implies that $\sum_i p_i^e z q_i^e=1$ and consequently $\sum_i p_i^e \otimes z q_i^e$ is a separability idempotent of the algebra $A_e$. Thus, $(A_e,\eta_e)$ is a separable symmetric Frobenius algebra as shown in \cite{schommer}. Similarly, we have $\sum_i p_i^g z q_i^g=1$ using the saddle whose image gives $\eta_g^A$. Until now we used $\zeta$ to replace $B^{\op}$ actions by $A$ actions. By changing the roles of $A$ and $B$ we obtain a quasi-biangular $G$-algebra $B$ and $\zeta$ is a compatible graded Morita context between $B$ and $A^{\op}$. 

Thus, any object in $\XP^{\PD}$ determines a triple $(A,B,\zeta)$. Conversely, for any such triple, one constructs an object of $\XP(\Alg)$ by assigning values to generating objects, $1$-morphisms, and $2$-morphisms of $\XP$ satisfying generating relations using the above arguments. Then, by the cofibrancy theorem this object gives a strict symmetric monoidal $2$-functor $\XB^{\PD} \to \Alg$ whose composition with the equivalence $\X\Bord_2 \xrightarrow{\sim} \XB^{\PD}$ produces the desired extended X-HFT.
\end{proof}
%\begin{figure}[t]
%    \centering
%    \includesvg{nsg}
%    \caption{Compositions of generating $2$-morphisms forming Morse relations}
%    \label{fig:nsg}
%\end{figure}
\begin{figure}[t]
    \centering
    \def\svgwidth{\columnwidth}
    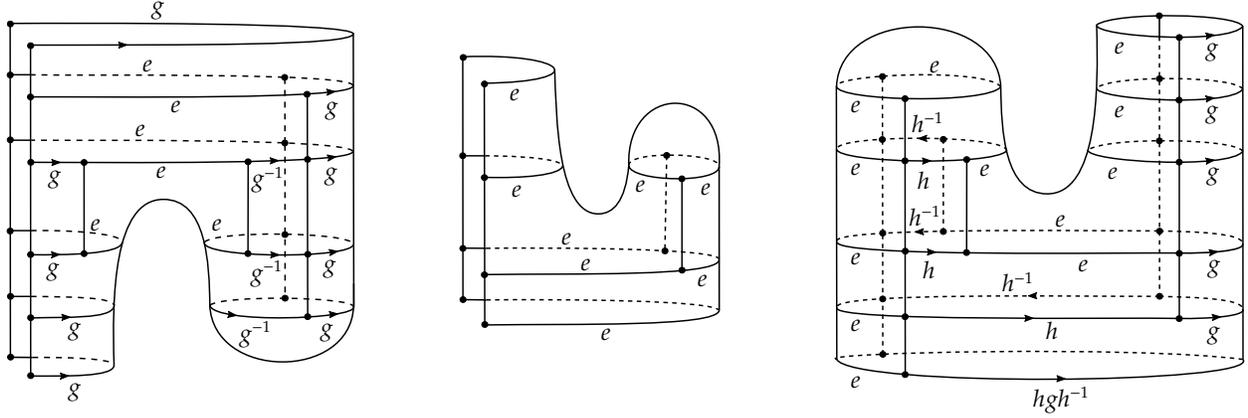
    \caption{Compositions of generating $2$-morphisms forming Morse relations}
    \label{fig:nsg}
\end{figure}
\begin{remark}
The cofibrancy theorem implies that any symmetric monoidal $2$-functor $Z: \XB^{\PD} \to \Alg$ can be strictified. That is, $Z$ is equivalent to a strict symmetric monoidal $2$-functor. From now on, by an $\Alg$-valued extended X-HFT giving triple $(A,B,\zeta)$ we mean the triple coming from the corresponding strict symmetric monoidal $2$-functor.  
\end{remark}
Turaev \cite{tur-hqft} defined the $G$-center of a biangular $G$-algebra. We extend this notion to a $G$-center of a quasi-biangular $G$-algebra $(K,\eta)$ as $Z_G(K)=\oplus_{g \in G} \Psi(K_g)$ where $\Psi(a)=\sum_i p_i^e a q_i^{e}$ for inner product elements $\big\{\sum_i p_i^g \otimes q_i^g\big\}_{g \in G}$. In general, $G$-center is not commutative and it differs from the usual center of the algebra. However, it has a crossed Frobenius $G$-algebra structure which is defined as follows.
\begin{definition}[\cite{tur-hqft}]
A Frobenius $G$-algebra $(L= \oplus_{g \in G}L_g,\eta)$ is \textit{crossed} if $L$ is endowed with a group homomorphism $\varphi:G \to \text{Aut}(L)$ satisfying the following conditions;
\begin{enumerate}[(i)]
    \item $\varphi$ is conjugation type i.e. $\varphi_h(L_g)= L_{hgh^{-1}}$ and $\varphi_h|_{L_h}=\text{id}_{L_g}$ for every $g,h \in G$,
    \item $ba= \varphi_h(a)b$ for any $a \in L$ and $b \in L_h$,
    \item $\text{Tr}(\mu \varphi_h: L_g \to L_g) = \text{Tr}(\varphi_{g^{-1}}\mu_c : L_h \to L_h)$ for all $g,h \in G$ and $c \in L_{ghg^{-1}h^{-1}}$ where $\mu_c : L \to L$ is a left multiplication by $c$ and $\text{Tr}$ is the trace of a map,
    \item $\eta$ is invariant under $\varphi$.
\end{enumerate}
\end{definition}
\begin{lemma}\label{Gcenteri}
Let $(K, \eta)$ be a quasi-biangular $G$-algebra with a central element $z \in K_e$ and a collection of inner product elements $\big\{\sum_i p_i^g \otimes q_i^g\big\}_{g \in G}$. Then, $Z_G(K)$ is a unital $G$-algebra with a multiplication $\sum_i p_i^e a q_i^e \cdot \sum_i p_i^e b q_i^e= \sum_{i,j} p_i^e a q_i^e p_j^e b q_j^e z^{-1}$ for all $a,b \in K$ and the triple $(Z_G(K), \eta|_{Z_G(K)}, \{\varphi_{g}|_{Z_G(K)} \}_{g \in G})$ is a crossed Frobenius $G$-algebra where $\varphi_g(a)=\sum_i p_i^g a zq_i^{g}$ for all $a \in K$ and all $g \in G$. 
\end{lemma}
\begin{proof}
The unit of $Z_G(K)$ is $\Psi(z^2)=z$. By Lemma \ref{label2} we have the equality
\begin{align} \label{eq4}
    \sum_{i} p_i^h bz' q_i^{h}c = \sum_i c p_i^{gh} b z' q_i^{gh}
\end{align}
for all $c \in K_{g^{-1}}$, $z' \in K_e$, $b \in K$ and $g,h \in G$. Taking $z'=1$ and $g=h=e$ gives
\begin{align}\label{eq5}
    \Psi(a\Psi(b)z^{-1})= \sum_{i,j} p_i^e a p_j^e b  q_j^{e} z^{-1} q_i^{e}= \sum_{i,j} p_i^e a p_j^e b  q_j^{e} q_i^{e} z^{-1}= \sum_{i,j} p_i^e a  q_i^{e} p_j^{e}b  q_j^{e} z^{-1}= \Psi(a)\cdot \Psi(b)
\end{align}
which implies that $Z_G(K)$ is closed under multiplication. Restriction of $\eta$ to $Z_G(K)$ is an inner product and hence $(Z_G(K),\eta|_{Z_G(K)})$ is a Frobenius $G$-algebra. For any $b \in K$ and for all $h \in G$ we have
\begin{align*}
    \Psi(z\varphi_h(b))= \sum_j p_j^e z \bigg(\sum_i p_i^h bz q_i^h\bigg)q_j^e= \sum_{i,j}p_j^ezq_j^e p_i^h bz q_i^h= \sum_i p_i^h bz q_i^h= \varphi_h(b)
\end{align*}
which shows that $\varphi_h(K)\subset Z_G(K)$. Similarly for any $\sum_i p_i^e a q_i^e \in Z_G(K)$ we have
\begin{align*}
    \varphi_e\bigg(\sum_i p_i^e a q_i^e \bigg)= \sum_j p_j^e \bigg(\sum_i p_i^e a q_i^e \bigg) z q_j^e=\sum_{i,j}p_j^ez q_j^e p_i^e a q_i^e = \sum_i p_i^e a q_i^e
\end{align*}
showing $\varphi_e|_{Z_G(K)}= id_{Z_G(K)}$. Note that for any $g \in G$ and $a \in K$ we have 
\begin{align*}
    \varphi_g(\Psi(a))= \sum_j p_j^g \bigg(\sum_i p_i^e a q_i^e \bigg)z q_j^g= \sum_j p_j^g z q_j^g \sum_i p_i^g a q_i^g=\sum_i p_i^g a q_i^g
\end{align*}
and using this we have the following equality for all $\bar{a}=\Psi(a),\bar{b}=\Psi(b) \in Z_G(K)$ and $g \in G$
\begin{align*}
    \varphi_g(\bar{a} \cdot \bar{b})= \sum_k p_k^g \bigg( \sum_{i,j} p_i^e a q_i^e p_j^e b q_j^e z^{-1}\bigg) z q_k^g = \bigg( \sum_i p_i^g a q_i^g \bigg) \bigg( \sum_j p_j^g b q_j^g  \bigg) \bigg( \sum_k p^g_k q_k^g \bigg)= \varphi_g( \bar{a})\cdot \varphi_g(\bar{b})
    \end{align*}
showing $\varphi_g$ is an algebra homomorphism. For the last equality, we have $\sum_k p_k^g q_k^g= z^{-1}$ since $\sum_k p_k^g q_k^g = \sum_{i,j} p_i^e \beta_g^j \beta_{g^{-1}}^j q_i^e= \sum_i p_i^e q_i^e=z^{-1}$ where $\sum_j \beta_g^j \otimes \beta_{g^{-1}}^j \in K_g \otimes_{K_e} K_{g^{-1}}$ is the inverse of $1 \in K_e$ under the product map  $K_h \otimes_{K_e} K_{h^{-1}}\to K_e$ (compare the first equality with the bimodule map corresponding to the second $2$-morphism in Figure \ref{fig:innpd}). We also have
\begin{align*}
    \varphi_g(\varphi_h(\Psi(b)))= \sum_j p_j^g \bigg( \sum_i p_i^h b q_i^h \bigg) z q_j^g = \sum_j p_j^g z q_j^g \sum_i p_i^{gh} b q_i^{gh}= \varphi_{gh}(\Psi(b))
\end{align*}
for all $g,h \in G$ and $b \in K$, which also implies that $\varphi_{g^{-1}}$ is the inverse of $\varphi_g$ for all $g \in G$. For all $\bar{a}=\Psi(a),\bar{b}=\Psi(b) \in Z_G(K)$ and $g \in G$ using the cyclic symmetry of $\eta$ we have
\begin{align*}
\eta(\varphi_g(\bar{a}),\bar{b})= \eta \bigg(\sum_{i} p_i^g \bar{a} z q_i^g, \sum_j p_j^e b q_j^e \bigg)= \eta \bigg( \bar{a}, \sum_{i,j} zq_i^g p_j^e b q_j^e p_i^g  \bigg)= \eta \bigg( \bar{a}, \sum_{i,k} p_k^{g^{-1}} b q_{k}^{g^{-1}}z q_i^g p_i^g  \bigg)= \eta ( \bar{a}, \varphi_{g^{-1}}(\bar{b}))
\end{align*}
showing the inner product $\eta$ is invariant under $\varphi : G \to \text{Aut}(Z_G(K))$.
For any $\bar{c}= \Psi(c) \in Z_G(K)_{h}$ we have
\begin{align*}
\varphi_h(\bar{c})= \sum_{i} p_i^h c q_i^h = \sum_{i,j} p_i^e \beta_h^j c \beta_{h^{-1}}^j q_i^e = \sum_i p_i^e c q_i^e= \bar{c}
\end{align*}
where $\sum_j \beta_h^j \otimes \beta_{h^{-1}}^j \in K_{h} \otimes K_{h^{-1}}$ is the inverse of $1 \in K_e$ under the product bimodule map. This shows that $\varphi_{h}$ acts by identity on $Z_G(K)_h$ for all $h \in G$. Equation (\ref{eq4}) gives $\varphi_g(a)b= b \varphi_{h^{-1}g}(a)$ for $a \in K, b \in K_h$ and $g,h\in G$. In this case by taking $g=h$ we have $\varphi_h(a)b=ba$. Let $\mu_c :K \to K$ be a multiplication by $c \in K$, then for any $g,h \in G$ and $c \in K_{ghg^{-1}h^{-1}}$ we have\footnote{See \cite{tur-hqft} for the first equality.}
\begin{align*}
    \Tr (\mu_c \varphi_{h} : K_g \to K_g) &= \sum_{i} \eta \big(c \varphi_{h} \big(p_i^g \big),q_i^{g}\big)= \sum_{i,j} \eta \big(c p_j^{h} p_i^g z q_j^{h}, q_i^{g} \big) = \sum_{i,j}\eta \big(q_i^{g} c p_j^{h}z p_i^g ,q_j^{h} \big) \\
    &= \sum_j \eta \big(\varphi_{g^{-1}}\big(c p_j^{h}\big), q_j^{h}\big)=\Tr(\varphi_{g^{-1}}\mu_c : K_h \to K_h).
 \end{align*}
\end{proof}
Any $2$-dimensional extended HFT produces a nonextended one by restricting it to a symmetric monoidal full subcategory $\X\mathbbmss{Cob}_2$ of $\X\Bord_2$ defined as follows. The objects of $\X\mathbbmss{Cob}_2$ are $\{ \gcirc \}_{g \in G}$, the empty $1$-morphism in $\X\Bord_2$, and disjoint union of these $1$-morphisms. The morphisms of $\X\mathbbmss{Cob}_2$ are the $2$-morphisms of $\X\Bord_2$ among these $1$-morphisms. We define a symmetric monoidal functor $D : \X\mathbbmss{Cob}_2 \to \X\text{Cob}_2$ by $\gcirc \mapsto \teknokta$ for any $g \in G$. On morphisms $D$ forgets a point on each boundary component and takes the corresponding relative homotopy class. Using definitions it is not hard to see that $D$ is an equivalence of categories. Then, by restriction of $Z : \X\Bord_2 \to \Alg$ to $\X\mathbbmss{Cob}_2$ above we mean precomposing $Z$ with $D^{-1}: \X\text{Cob}_2 \to \X\mathbbmss{Cob}_2$.
\begin{cor}
Let $Z:\X\Bord_2 \to \Alg$ be an extended HFT giving $(A,B,\zeta)$. Then, the nonextended HFT obtained from $Z$ by restricting to $\X\mathbbmss{Cob}_2$ is the nonextended HFT associated to the $G$-center of the quasi-biangular $G$-algebra $(A,\eta_A)$.
\end{cor}
\begin{proof}
Proceeding with the notation used in the proof of Theorem \ref{orclass} the image of a $g$-labeled circle under $Z$ is given by 
\begin{align*}
    A_e \otimes_{A_e \otimes A_e^{\op}} A_g=\{ b \in A_g \ | \ a \cdot b = b \cdot a  \text{ for all } a \in A_e\}.
\end{align*}
The $G$-center of $(A,\eta_A)$ is given by $Z_G(A)= \oplus_{g \in G} \Psi (A_g)$. For any $a  \in A_e \otimes_{A_e \otimes A^{\op}_e} A_g$ we have
\begin{align*}
    a= 1.a = \bigg(\sum_i p_i^e z q_i^e \bigg)a = \sum_i p_i^e  a z q_i^e =\Psi(az) \in \Psi(A_g) 
\end{align*}
and for any $\sum_i p_i^e a  q_i^e \in \Psi(A_g)$ and $b\in A_e$ we have
\begin{align*}
    \bigg(\sum_i p_i^e a  q_i^e \bigg)b = \sum_i p_i^e a q_i^e b = \sum_i b p_i^{e} a q_i^{e} = b \bigg(\sum_i p_i^e a q_i^e \bigg) 
\end{align*}
where the middle equality is the result of Lemma \ref{label2}. Thus, we have $A_e \otimes_{A_e \otimes A^{\op}_e} A_g =  \Psi(A_g)$ for all $g \in G$. The third $2$-morphism in Figure \ref{fig:nsg} gives the crossed structure on the restricted HFT and it corresponds to the following sequence of compositions
\begin{align*}
    1 \otimes a &\to 1 \otimes \sum_j a_jb_j \otimes a \to  \sum_j   1 \otimes \beta_h^j z  \beta_{h^{-1}}^j \otimes a \to \\
    \to 1 \otimes \sum_j \beta_h^j \bigg( \sum_i p_i^e \otimes  & z q_i^e \bigg) \beta_{h^{-1}}^j \otimes a \to 1 \otimes \sum_i p_i^h \otimes z q_i^{h} \otimes a \to 1 \otimes \sum_i p_i^h a z q_i^{h} 
\end{align*}
which coincides with the crossed structure of $Z_G(A)$.
\end{proof}
\begin{example}\label{ornek}
Let $\Bbbk$ be an algebraically closed field. Then separable $\Bbbk$-algebras are the same as semisimple $\Bbbk$-algebras. By Artin-Wedderburn structure theorem any separable algebra is isomorphic to a product of finitely many matrix algebras over $\Bbbk$. Consider the $G$-algebra $A=\oplus_{g \in G}A_g$ whose principal component is a product $A_e = \prod_{i=1}^n M_{k_i}(\Bbbk)$ of $(k_i\times k_i)$-matrix algebras over $\Bbbk$ such that each $k_i$ is invertible in $\Bbbk$ and each component is given by $A_g = \ell_g A_e$ where $\ell_g$ is a basis, i.e. for any $a \in A_g$ there exists $b \in A_e$ such that $a= \ell_g b$. Define an inner product $\eta$ on $A$ as 
\begin{align*}
\eta(a,b)=\begin{cases}
r \text{Tr}(L_{ab}: A_e \to A_e) & \text{ when } ab \in A_e \\
0 & \text{ otherwise}
\end{cases}    
\end{align*}
where $r \in \Bbbk$ is invertible and $\text{Tr}(L_{ab})$ is the trace of left multiplication by $ab$ map. We can express the inner product concretely as $\eta(\ell_g \prod_{i=1}^n A_i,\ell_{g^{-1}} \prod_{i=1}^n B_i)= r \sum_{i=1}^n k_i\text{Tr}(A_iB_i)$ where $\text{Tr}(A_iB_i)$ is the trace of the matrix $A_iB_i$. For each $g \in G$ an inner product element can be chosen as $\eta_g^-= r^{-1}\prod_{i=1}^n k_i^{-1}\sum_{\alpha,\beta=1}^{k_i} \ell_{g} E_{\alpha,\beta} \otimes \ell_{g^{-1}} E_{\beta,\alpha} \in A_g \otimes A_{g^{-1}}$ where $E_{\alpha,\beta}$ is the $(\alpha,\beta)$-elementary matrix. In this case, the central element $z \in A_e$ is given by $(r I_{k_1},\dots,r I_{k_n})$ where $I_{k_i}$ denote $(k_i \times k_i)$ identity matrix. Note that $\prod_{i=1}^n k_i^{-1}\sum_{\alpha,\beta=1}^{k_i}  E_{\alpha,\beta} \otimes  E_{\beta,\alpha}$ is a separability idempotent of $A_e$. Thus, the map $\Psi: A_g \to A_g$ is given by $$\Psi\bigg(\ell_g \prod_{i=1}^n A_i \bigg)= r^{-1}\prod_{i=1}^n k_i^{-1} \sum_{\alpha,\beta=1}^{k_i} E_{\alpha,\beta} (\ell_g A_i) E_{\beta,\alpha}= r^{-1}\prod_{i=1}^n k_i^{-1} \ell_g\text{Tr}(A_i)I_{k_i}$$ which is a projection onto its center $\ell_g \Bbbk^n$.
\end{example}
\subsection{The bicategory of 2-dimensional extended X-HFTs}\label{equivbicat_kanit}
Until now we have studied the objects of $\XP(\Alg)$. Theorem \ref{genelsinif} implies that studying $1$- and $2$-morphisms of $\XP(\Alg)$ leads us to a bicategory equivalent to $\mathcal{E}\text{-}\mathcal{HFT}(X,\Alg)$. Let $Z_0$ and $Z_1$ be extended HFTs with target $X$ giving triples $(A,B,\zeta)$ and $(A',B',\zeta')$ respectively. A $1$-morphism $\alpha : Z_0 \to Z_1$ in $\XP(\Alg)$ gives $1$-morphisms $ \alpha_0(\nokta)=  {}_{A_e'}R_{A_e}$ and $\alpha_0(\noktab) = {}_{B_e'}S_{B_e}$, and $2$-morphisms 
\begin{align*}
    \alpha_1(\arti)&: {}_{A'_e}A_g' \otimes_{A_e'} R_{A_e} \to {}_{A_e'}R \otimes_{A_e} (A_g)_{A_e}\\
   \alpha_1(\eksi)&: {}_{B_e'}B_g' \otimes_{B_e'} S_{B_e} \to {}_{B_e'}S \otimes_{B_e} (B_g)_{B_e}\\
    \alpha_1(\leftelbow)&:  {}_{A_e' \otimes B_e'}(M_g')_{\Bbbk} \to {}_{A_e' \otimes B_e'}(R \otimes S)\otimes_{A_e \otimes B_e} (M_g)_{\Bbbk}\\
    \alpha_1(\rightelbow)&: {}_{\Bbbk}N_g' \otimes_{B_e' \otimes A_e'} (S \otimes R)_{B_e \otimes A_e} \to {}_{\Bbbk}(N_g)_{B_e \otimes A_e} 
\end{align*}
which are isomorphisms for all $g \in G$ and $G$-graded bimodules $M,M',N$, and $N'$ are the components of $\zeta$ and $\zeta'$. These morphisms are natural with respect to generating $2$-morphisms. Naturality with respect to graded multiplication, $\commuc$, leads to the commutativity of the diagram
\begin{align*}
    \xymatrixcolsep{5pc}\xymatrix{ {}_{A_e'}(A_{g'}') \otimes_{A_e'} A_g' \otimes _{A_e'} R_{A_e} \ar[r]^{\alpha_1(\commua)} \ar[d]_{Z_1\big(\commuc \big)} & {}_{A_e'}R\otimes_{A_e}A_{g'} \otimes_{A_e}(A_g)_{A_e} \ar[d]^{Z_0 \big(\commuc \big)} \\ {}_{A_e'}A_{gg'}' \otimes_{A_e'} R_{A_e} \ar[r]_{\alpha_1(\commub)} & {}_{A_e'}R \otimes_{A_e} (A_{gg'})_{A_e} 
    } 
\end{align*}
for all $g,g' \in G$. We denote bimodules ${}_{A_e'}A_g' \otimes_{A_e'} R_{A_e}$ and ${}_{A_e'}R \otimes_{A_e}(A_g)_{A_e}$ by ${}_{A_e'}(R_g')_{A_e}$ and ${}_{A_e'}(R_g'')_{A_e}$ respectively. Commutativity of the above diagram implies that they are naturally isomorphic. Thus, we can use one of them and denote it by $R_g$. Similarly, $S_g$ denotes $(B_e',B_e)$-bimodule. These assignments and naturality with respect to $\big\{ \commuc \big\}_{g,g' \in G}$ turn these bimodules into $G$-graded $(A',A)$ and $(B',B)$-bimodules $R=\oplus_{g \in G} R_g$ and $S=\oplus_{g \in G}S_g$ respectively. Similarly, naturality with respect to $G$-module generators turns collections $\{\alpha_1(\leftelbow)\}_{g \in G}$ and $\{\alpha_1(\rightelbow)\}_{g \in G}$ into $G$-graded $(A' \otimes B',\Bbbk)$ and $(\Bbbk,B \otimes A)$-bimodule maps respectively. 

Using $\alpha_0(\noktab)$ we define a $1$-morphism $\alpha_0'(\nokta) = {}_{A_e}R'_{A_e'}$ as follows 
\begin{align*}
    \alpha_0'(\nokta) &= [Z_1(\rightelbowe) \otimes id_{Z_0(\nokta)}] \circ [\alpha(\noktab) \otimes \sigma_{Z_0(\nokta),Z_1(\nokta)}] \circ [Z_0(\leftelbowe) \otimes id_{Z_1(\nokta)}]\\
    \alpha_0'(\nokta) = [(N_e')_{B_e' \otimes A_e'} &\otimes_{\Bbbk} {}_{A_e}(A_e)_{A_e}] \otimes_{B_e' \otimes A_e' \otimes A_e} [{}_{B_e'}S_{B_e} \otimes_{\Bbbk} \sigma_{A_e,A_e'}] \otimes_{B_e \otimes A_e \otimes A_e'} [ {}_{B_e \otimes A_e}M_e \otimes_{\Bbbk} {}_{A_e'}(A_e')_{A_e'}]
\end{align*}
where $\sigma$ is the symmetric braiding of $\Alg$. Using $\alpha_0'(\nokta)$ we define a $2$-morphism
\begin{align*}
\alpha_1'(\arti) &= Z_0(\arti) \circ \alpha_0'(\nokta) \to \alpha_0'(\nokta) \circ Z_1(\arti)\\
 \alpha_1'(\arti) &= {}_{A_e}A_e \otimes_{A_e} R'_{A_e'} \to {}_{A_e}R' \otimes_{A_e'} (A_g')_{A_e'}.
\end{align*}
Using the naturality $R'$ is turned into a $G$-graded $(A,A')$-bimodule $R'= \oplus_{g \in G}R_g'$. The $1$-morphism $\alpha_0(\noktab)$ can be obtained from $\alpha_0'(\nokta)$ by applying $Z_1\big(\cuspa \big)\circ id_{\alpha_0(\noktab)} \circ Z_0\big(\cuspb \big)$ to the $1$-morphism 
\begin{align*}
    [Z_1(\rightelbowe) \otimes Z_0(\rightelbote) \otimes id_{Z_1(\noktab)}]\circ[\alpha_0(\noktab) \otimes \sigma_{Z_0(\nokta), Z_1(\nokta)}\otimes \sigma_{Z_1(\noktab),Z_0(\noktab)}]\circ[Z_0(\leftelbowe) \otimes Z_1(\leftelbote) \otimes id_{Z_0(\noktab)} ]
\end{align*}
and similarly $\alpha_1(\eksi)$ can be obtained from $\alpha_1'(\arti)$. Likewise, using $\alpha_0'(\nokta)$ in the images of cusps generators under $Z_0$, the $2$-morphisms $\alpha_1'(\leftelbow)$ and $\alpha_1'(\rightelbow)$ are defined and both $\alpha_1(\leftelbow)$ and $\alpha_1(\rightelbow)$ can be obtained from these $2$-morphisms.

As in the proof of Theorem \ref{orclass} using $G$-graded Morita contexts $\zeta$ and $\zeta'$ graded bimodules $M$ and $M'$ can be replaced by ${}_{A \otimes A^{op}}A$ and ${}_{A' \otimes (A')^{op}}A'$. We can also replace the graded bimodule $S$ by $R'$ using $\alpha_0'(\nokta)$. Thus, naturality with respect to $G$-module generators turn the collection $\{\alpha_1'(\leftelbow)\}_{g \in G}$ into a bimodule map ${}_{A'}A'_{A'} \to {}_{A'}R \otimes_{A}R'_{A'}$. Similarly, the collection $\{\alpha_1'(\rightelbow)\}_{g \in G}$ is turned into a bimodule map ${}_{A}R' \otimes_{A'} R_{A} \to {}_{A}A_{A}$. Naturality with respect to cusp generators indicate that compositions 
\begin{align*}
    {}_{A'}R_A \xrightarrow{} {}_{A'}A'_{A'} \otimes {}_{A'}R \xrightarrow{ \alpha_1'(\leftelbor)\otimes \text{id}} {}_{A'}R \otimes_A R' \otimes_{A'} R_A \xrightarrow{\text{id} \otimes \alpha_1'(\rightelbos)} {}_{A'}R \otimes_{A} A_A \xrightarrow{} {}_{A'}R_A \\
    {}_{A}R'_{A'} \xrightarrow{} {}_{A}R'\otimes_{A'}A'_{A'} \xrightarrow{\text{id} \otimes \alpha_1'(\leftelbor)} {}_{A}R' \otimes_{A'} R \otimes_A R'_{A'} \xrightarrow{\alpha_1'(\rightelbos)\otimes \text{id}} {}_{A}A \otimes_{A} {R'}_{A'} \xrightarrow{} {}_{A}R_{A'}'
\end{align*}
are $\text{id}_{R}$ and $\text{id}_{R'}$ respectively. In other words, $\alpha$ gives a $G$-graded Morita context between $A$ and $A'$. Similarly, one can define $\alpha_0'(\eksi)$ and obtain a $G$-graded Morita context between $B$ and $B'$. Naturality with respect to Morse generators indicates that $G$-graded Morita contexts are compatible. Hence, $\alpha$ leads to two compatible $G$-graded Morita contexts. In the theory of bicategories this means that both $\alpha_0(\nokta)$ and $\alpha_0(\noktab)$ are parts of two adjoint equivalences. Since an adjoint equivalence is the same as an equivalence (see Proposition A.27 in \cite{schommer}) $Z_0$ and $Z_1$ are equivalent extended HFTs.

Let $\alpha^1,\alpha^2:Z_0 \to Z_1$ be $1$-morphisms in $\XP(\Alg)$ and $\theta: \alpha^1 \to \alpha^2$ be a $2$-morphism in $\XP(\Alg)$. Assume that $Z_0$ and $Z_1$ give triples $(A,B,\zeta)$ and $(A',B',\zeta')$ as before and $1$-morphisms give $\alpha_0^1(\nokta)= {}_{A_e'}R_{A_e}$ and $\alpha_0^2(\nokta) ={}_{A_e'}P_{A_e}$. Then, $\theta_0(\nokta) = {}_{A_e'}R_{A_e} \to {}_{A_e'}P_{A_e}$ and the naturality of $\theta_0(\nokta)$ with respect to $\arti$ is the commutativity of the following diagram
\begin{align*}
  \xymatrixcolsep{5pc}\xymatrix{ {}_{A_e'}A_g' \otimes_{A_e'} R_{A_e} \ar[r]^{\alpha_1^1 (\arti)} \ar[d]_{\theta_0 (\nokta)} & {}_{A_e'}R \otimes_{A_e} (A_g)_{A_e} \ar[d]^{\theta_0(\nokta)} \\ {}_{A_e'}A_g' \otimes_{A_e'} P_{A_e} \ar[r]_{\alpha_1^2(\arti)} & {}_{A_e'}P \otimes_{A_e} (A_g)_{A_e}
    }
\end{align*}
which shows that $\theta_0(\nokta)$ is a $G$-graded bimodule map. Assuming $(\alpha_0')^1(\nokta)={}_{A_e}R'_{A_e'}$ and $(\alpha_{0}')^2(\nokta)= {}_{A_e'}P'_{A_e}$ we similarly have a graded bimodule map $\theta_0'(\nokta):{}_{A_e}R'_{A_e'}\to {}_{A_e}P'_{A_e'}$ using $\theta_0(\noktab)$ and $(\alpha_0')^i(\nokta)$ for $i=1,2$. Naturality with respect to $\leftelbow$ and $\rightelbow$ corresponds to the commutativity of these bimodule maps with the unit and counit of the adjunctions. In other words, $\theta$ leads to an equivalence of graded Morita contexts. In the same way, using $B$ and $B'$ one gets another equivalence of graded Morita contexts.

Motivated by these observations we define a bicategory $\text{Frob}^G$ and a forgetful $2$-functor $\mathcal{F}': \XP(\Alg) \to \text{Frob}^G$ as follows. The bicategory $\text{Frob}^G$ has quasi-biangular $G$-algebras as objects, compatible $G$-graded Morita contexts as $1$-morphisms, and equivalences of $G$-graded Morita contexts as $2$-morphisms. The forgetting $2$-functor $\mathcal{F}'$ maps an object of $\XP(\Alg)$ giving $(A,B, \zeta)$ to $A$. On $1$-morphisms $\mathcal{F}'$ maps $\alpha: Z_0 \to Z_1$ to a compatible $G$-graded Morita context between quasi-biangular $G$-algebras whose principal components are $Z_0(\nokta)$ and $Z_1(\nokta)$. On $2$-morphisms $\mathcal{F}'$ maps $\theta: \alpha^1 \to \alpha^2$ to an equivalence of the compatible $G$-graded Morita contexts. Composing $\mathcal{F}'$ with the equivalence $\mathcal{E}$-$\mathcal{HFT}(X,\Alg)\simeq \XP(\Alg)$ we define $\mathcal{F}$.
\begin{theorem}\label{equivbicats}
\label{thm:equivbicats}
The $2$-functor $\mathcal{F}$ is an equivalence of bicategories $ \mathcal{E}\text{-}\mathcal{HFT}(X,\Alg) \simeq \text{Frob}^G$.  
\end{theorem}
\begin{proof}
It is enough to show that $\mathcal{F}'$ is an equivalence and we use Whitehead theorem (Theorem \ref{whiteadforsym}). For a given quasi-biangular $G$-algebra $A$, the triple $(A,A^{op},\text{id})$ gives an object $Z$ of $\XP(\Alg)$ such that $\mathcal{F}'(Z)=A$. Let $\alpha$ be a compatible $G$-graded Morita context between quasi-biangular $G$-algebras $A$ and $A'$. Then triples $(A,(A')^{op},\alpha)$ and $(A',A^{op},\alpha)$ give objects $Z_0$ and $Z_1$ in $\XP(\Alg)$ such that $\mathcal{F}'(\alpha')=\alpha$ where $\alpha': Z_0 \to Z_1$.

For any two $1$-morphisms $\alpha^1,\alpha^2: Z_0 \to Z_1$, we claim that $$\mathcal{F}'(\alpha^1,\alpha^2): \Hom(\alpha^1,\alpha^2) \to \Hom(\mathcal{F}'(\alpha^1),\mathcal{F}'(\alpha^2))$$ is an injection. Assume that different $2$-morphisms $\theta^1,\theta^2 : \alpha^1 \to \alpha^2$ in $\XP(\Alg)$ give the same equivalence of $G$-graded Morita contexts. This means that pairs $(\theta_0^1(\noktab),(\theta_0')^1(\noktab))$ and $((\theta_0^2(\noktab),(\theta_0')^2(\noktab))$ give different graded bimodule maps while pairs, images of $\theta^1$ and $\theta^2$ under $\mathcal{F}'$,  $((\theta_0^1(\nokta),(\theta_0')^1(\nokta))$ and $((\theta_0^2(\nokta),(\theta_0')^2(\nokta))$ give the same graded bimodules maps. This is a contradiction because each $(\theta_0')^i(\noktab)$ is obtained from $(\theta_0)^i(\nokta)$ and each $(\theta_0')^i(\nokta)$ is obtained from $(\theta_0)^i(\noktab)$ for $i=1,2$. For the surjectivity let $\theta: \mathcal{F}'(\alpha^1)\to \mathcal{F}'(\alpha^2)$ be an equivalence of graded Morita contexts. Then the equivalence of graded Morita contexts $(\theta_0(\noktab),\theta_0'(\noktab))$ can be obtained from $\theta_0(\nokta)$, $\theta_0'(\nokta)$, $(\alpha_0')^1(\noktab)$, and $(\alpha_0')^2(\noktab)$.
\end{proof}
\begin{cor}\label{impliedequiv}
\label{thm:impliedequiv}
Two triples $(A_1,B_1,\zeta_1)$ and $(A_2,B_2,\zeta_2)$ produce equivalent extended X-HFTs if and only if there exists a compatible $G$-graded Morita context between $A_1$ and $A_2$.
\end{cor}
Lastly, we comment on the relation between extended HFTs whose targets are related by covering maps. Let $\text{Y}\simeq K(H,1)$ be a pointed CW-complex for a nontrivial subgroup $H \leq G$ and $p : (Y,y) \to (X,x)$ be a covering. Then, any Y-manifold can be turned into an X-manifold by postcomposing a representative of the characteristic map with $p$. This gives a symmetric monoidal $2$-functor $\iota_H : \text{Y}\Bord_2 \to \X\Bord_2$ and precomposing any extended X-HFT with $\iota_H$ yields an extended Y-HFT. Moreover, for any symmetric monoidal bicategory $\mathcal{C}$, precomposition of $\mathcal{C}$-valued extended X-HFT with $\iota$ lifts to a $2$-functor $\text{SymMon}(\X\Bord_2,\mathcal{C}) \to \text{SymMon}(\text{Y}\Bord_2,\mathcal{C})$
by forgetting the naturality of transformations with respect to $G \backslash H$ labeled $1$-morphisms. Correspondingly, there is a $2$-functor $\XP(\mathcal{C}) \to \mathbb{YP}(\mathcal{C})$ where $\XP$ and $\mathbb{YP}$ are the presentations of $\X\Bord_2$ and $\text{Y}\Bord_2$ respectively. When $\mathcal{C}$ is $\text{Alg}_{\Bbbk}^2$, the functor $\text{Frob}^G \to \text{Frob}^H$ is given by forgetting the $G \backslash H$ components of quasi-biangular $G$-algebras, compatible $G$-graded Morita contexts, and equivalences of $G$-graded Morita contexts. In other words, a $G$-graded Morita context can be considered as a collection of Morita contexts indexed by the subgroups of $G$ (see \cite{gmorita}).

\subsection{The $(G \times SO(2))$-structured cobordism hypothesis}\label{gso2sch}
A different approach to categorical classification of (fully-)extended oriented HFTs is given by the structured cobordism hypothesis due to J.~Lurie \cite{lurie}. Cobordism hypothesis (\cite{ayalafrancis}, \cite{lurie}, \cite{baezdolan}) is conjectured by J. Baez and J. Dolan in their seminal paper \cite{baezdolan}. Lurie \cite{lurie} reformulated the cobordism hypothesis using $(\infty,n)$-categories and generalized it to a structured cobordism hypothesis using homotopy fixed points. 

The bordism category involved in the structured cobordism hypothesis consists of manifold with corresponding structures. For a topological group $\Gamma$ and a fixed continuous homomorphism $\chi: \Gamma \to O(n)$, let $\zeta_{\chi}: \R^n \times_{\Gamma} E\Gamma$ denote the associated rank $n$ vector bundle over the classified space $B\Gamma$. Then, a $\Gamma$-structure on a manifold $M$ of dimension $k \leq n$ consists of a continuous map $f : M \to B\Gamma$ and an isomorphism $TM \oplus \underline{\R}^{n-k} \to f^{\ast} \zeta_{\chi}$ of vector bundles where $\underline{\R}^{n-k}$ is a trivial rank $n-k$ vector bundle. In the following theorem the category $\Bord_n^{\Gamma}$ consists of manifolds equipped with $\Gamma$-structure for a fixed homomorphism $\chi$. \\[.2cm]
\textbf{$(\Gamma,\chi)$-Structured Cobordism Hypothesis.} \textit{(Lurie, \cite{lurie})
Let $\mathcal{C}$ be a symmetric monoidal $(\infty,n)$-category (see \cite{bordn}) and $\Bord_n^{\Gamma}$ be the symmetric monoidal $\Gamma$-structured cobordism $(\infty,n)$-category for a group $\Gamma$. Then, there is a canonical equivalence of $(\infty,n)$-categories 
\begin{align*}
    \text{Fun}^{\otimes}(\Bord_n^{\Gamma}, \mathcal{C}) \xrightarrow{\sim} ((\mathcal{C}^{fd})^{\sim})^{h\Gamma}
\end{align*}
where $\text{Fun}^{\otimes}$ is the $(\infty,n)$-category of symmetric monoidal functors between symmetric monoidal $(\infty,n)$-categories, $\mathcal{C}^{fd}$ is the sub-$(\infty,n)$-category of fully dualizable objects with duality data, $(\mathcal{C}^{fd})^{\sim}$ is the underlying $\infty$-groupoid and $((\mathcal{C}^{fd})^{\sim})^{h\Gamma}$ is the $\infty$-groupoid of homotopy $\Gamma$-fixed points given by
\begin{align*}
    ((\mathcal{C}^{fd})^{\sim})^{h\Gamma} = \Hom_{\Gamma}(E\Gamma, (\mathcal{C}^{fd})^{\sim})
\end{align*}
where $E\Gamma$ is a weakly contractible $\infty$-groupoid equipped with a free $\Gamma$-action.} 
\begin{remark}
A $2$-dimensional E-HFT with target $X \simeq K(G,1)$ i.e. a classifying space $BG$, is a $(G \times SO(2))$-structured $2$-dimensional E-TFT where $\chi: G \times SO(2) \to O(2)$ is given by $(g,A) \mapsto A$.
\end{remark}
When $\Bbbk$ is an algebraically closed field of characteristic zero Davidovich \cite{davidovich} showed that for a finite group $G$, homotopy $(G\times SO(2))$-fixed points in $\big(\Alge^{fd}\big)^{\sim}$ are given by $G$-equivariant algebras. A $G$-equivariant algebra is a strongly graded Frobenius $G$-algebra with semisimple principal component. Her methods do not particularly require group $G$ to be finite and can be extended to discrete groups directly. Since the notions of separability and semisimplicity for a $\Bbbk$-algebra are equivalent when $\Bbbk$ is an algebraically closed field of characteristic zero, the objects of $\text{Frob}^G$ and the objects of $\Big(\big(\Alge^{fd}\big)^{\sim}\Big)^{h(G \times SO(2))}$ coincide.

Assume that $\Bbbk$ is an algebraically closed field of characteristic zero. Artin-Wedderburn theorem implies that any separable $\Bbbk$-algebra is isomorphic to a product of matrix algebras over $\Bbbk$. Let $A_e= \text{End}(V_1)\times \text{End}(V_2) \times \dots \times \text{End}(V_n)$ be a such algebra where $V_1,V_2,\dots,V_n$ are finite dimensional $\Bbbk$-vector spaces. Recall that $A=\oplus_{g \in G}A_g$ is strongly graded by the generators leading to bimodule isomorphisms $\{\tau_{g,g'} : A_{g'} \otimes_{A_e}A_g \xrightarrow{\cong} A_{gg'}\}_{g,g\in G}$, that is each $A_g$ is an invertible $(A_e,A_e)$-bimodule. Under the above assumption on $A_e$, these isomorphisms form a function $\tau: G \times G \to (\Bbbk^{\ast})^n$. Moreover, the relations involving these generators give the following commutative diagram for all $g,g',g'' \in G$
\begin{align*}
   \xymatrixcolsep{4pc}  \xymatrix{ (A_{g''} \otimes_{A_e} A_{g'}) \otimes_{A_e} A_{g} \ar[d]_-{\cong} \ar[r]^-{\tau(g',g'')\otimes \text{id}} & A_{g'g''} \otimes_{A_e} A_g \ar[r]^-{\tau(g,g'g'')} & A_{gg'g''} \ar[d]^{\text{id}}  \\
    A_{g''} \otimes_{A_e} (A_{g'} \otimes_{A_e} A_g) \ar[r]_-{\text{id} \otimes \tau(g,g')} & A_{g''} \otimes_{A_e} A_{gg'} \ar[r]_-{\tau(gg',g'')} & A_{gg'g''}
    }
\end{align*}
and isotopy classes of $G$-linear diagrams generate the relations which can be expressed as the following commutative diagram for all $g \in G$ 
\begin{align*}
    \xymatrixcolsep{4pc} \xymatrix{ A_g \otimes_{A_e} A_e \ar[r]^-{\tau(e,g)} \ar[dr]_-{\cong} & A_g \ar[d]_-{\text{id}} & A_e \otimes_{A_e} A_g \ar[l]_-{\tau(g,e)} \ar[dl]^-{\cong} \\ & A_g & 
    }
\end{align*}
which imply that $\tau$ is a normalized $2$-cocycle. Davidovich \cite{davidovich} showed that any invertible $(A_e,A_e)$-bimodule is isomorphic to one of the form $$\Hom_{\Bbbk}(V_{\sigma(1)},V_1)\times \Hom_{\Bbbk}(V_{\sigma(2)},V_2)\times \dots \times \Hom_{\Bbbk}(V_{\sigma(n)},V_n)$$ for some permutation $\sigma \in S_n$ and denote this bimodule by $A^{\sigma}$. Since the direct sum $A= \oplus_{g\in G} A_g$ forms a $G$-algebra permutations indeed form a homomorphism $\sigma :G \to S_n$. 

It is known that all traces on a matrix algebra are given as some (nonzero) constant multiple of the matrix trace. Thus, in the case of $A_e= \text{End}(V_1)\times \text{End}(V_2) \times \dots \times \text{End}(V_n)$ there are constants $r_i \in \Bbbk^{\ast}$ for $i=1,\dots,n$ and the inner product of quasi-biangular $G$-algebra $A= \oplus_{g \in G} A^{\sigma(g)}$ is given by $\eta(f,g)= \text{Tr}(r \circ (g \circ^{\sigma}f))$ for any $f \in A_g$ and $g \in A_{g^{-1}}$ where $r=(r_1 \text{id}_{V_1},\dots,r_n \text{id}_{V_n})$ and $\circ^{\sigma}$ is the composition of morphisms under $\sigma$ such as $f_i \circ g_{\sigma(i)}$ for $f_i \in \Hom_{\Bbbk}(V_{\sigma(i)},V_i), g_{\sigma(i)} \in \Hom_{\Bbbk}(V_{\sigma(\sigma(i))},V_{\sigma(i)})$. Since the inner product is invariant under cyclic order i.e. $\eta(f,g \cdot h)= \eta(h \cdot f,g) =\eta(h, f \cdot g)$, the vector $r \in (\Bbbk^{\ast})^n$ must satisfy $\text{Im}(\sigma)\subseteq \text{Stab}_{S_n}(r)$ where $S_n$ acts on $r$ by permuting the entries. More explicitly, as an example consider the products $(h \circ^{\sigma} g) \circ^{\sigma} f$ and $g \circ^{\sigma} (f \circ^{\sigma}h)$ for $f \in A_g, g \in A_{g'}$, and $h \in A_{(gg')^{-1}}$. Then, the corresponding traces of these morphisms in $A_e =\text{End}(V_1)\times \text{End}(V_2) \times \dots \times \text{End}(V_n)$ are related by the permutation $\sigma(gg') \in S_n$. 

Using the above arguments when $\Bbbk$ is algebraically closed field of characteristic zero we can conclude that up to an isomorphism a quasi-biangular $G$-algebra $(A=\oplus_{g \in G} A_g,\eta)$ is determined by a Morita class of the principal component $(n \geq 1)$, a normalized $2$-cocycle $\tau: G \times G \to (\Bbbk^{\ast})^n$, a homomorphism $\sigma : G \to S_n$, and an element $r \in (\Bbbk^{\ast})^n$ with $\text{Im}(\sigma)\subseteq \text{Stab}_{S_n}(r)$. 

Let $({}_{A_2}M_{A_1}, {}_{A_1}N_{A_2}, \kappa, \mu)$ be a graded Morita context between two quasi-biangular $G$-algebras $(A_1,\eta_1)$ and $(A_2,\eta_2)$ which are determined by the normalized $2$-cocycles $\tau_i$, homomorphisms $\sigma_i : G \to S_n$, and elements $r_i \in (\Bbbk^{\ast})^n$ with $\text{Im}(\sigma_i)\subseteq \text{Stab}_{S_n}(r_i)$ for $i=1,2$. Then $M$ and $N$ are invertible $(A_2,A_1)$ and $(A_1,A_2)$ bimodules respectively, which means there exists $\sigma \in S_n$ such that $M_e$ is isomorphic to 
\begin{align*}
M^{\sigma}= \Hom(V_{\sigma(1)},W_{1})\times \Hom(V_{\sigma(2)},W_{2})\times \dots \times \Hom(V_{\sigma(n)},W_n)    
\end{align*}
where $(A_i)_g= \Hom(V_{\sigma_i^g(1)},V_1)\times \dots \times \Hom(V_{\sigma_i^g(n)},V_n)$ for all $g \in G$ and $\sigma_i^g = \sigma_i(g) \in S_n$ for $i=1,2$. Being a graded $(A_2,A_1)$-bimodule forces $\sigma$ to satisfy $\sigma_2^g = \sigma \sigma_1^g \sigma^{-1}$ for all $g \in G$. In this case, nonprincipal components are given as $M_g= \Hom(V_{\sigma_g'(1)},W_1)\times \dots \times \Hom(V_{\sigma_g'(n)},W_n)$ where $\sigma_g'= \sigma \sigma_1^g= \sigma_2^g \sigma$ for all $g\in G$. Using the similar arguments for the invertible $(A_1,A_2)$-bimodule $N$, we obtain $N_g = \Hom(W_{\sigma_g''(1)},V_1) \times \dots \times \Hom(W_{\sigma_g''(n)},V_n)$ for $\sigma_g'' = \sigma_1^{g^{-1}} \sigma^{-1}= \sigma^{-1}\sigma_2^{g^{-1}}$ for all $g \in G$. 

Transferring the Frobenius form via the graded Morita context amounts to finding a central element in $(A_2)_e$ corresponding to $r_1 \in (\Bbbk^{\ast})^n$. Using the identity component $M_e$ this element is given by $\sigma(r_1)(\text{id}_{W_1},\text{id}_{W_2},\dots,\text{id}_{W_n})$. Thus, we have the equality $\sigma(r_1)=r_2 \in (\Bbbk^{\ast})^n$. The bimodule isomorphisms $\kappa: {}_{A_1}(A_1)_{A_1} \to {}_{A_1}N \otimes_{A_2} M_{A_1}$ and $\mu: {}_{A_2}M \otimes_{A_1} N_{A_2} \to {}_{A_2}(A_2)_{A_2}$ lead to a map $\phi : G \to (\Bbbk^{\ast})^n$ and graded Morita context equations produce a map $\phi : (A_1)_g \to (A_2)_g$ for all $g \in G$ so that the diagram 
\begin{align*}
   \xymatrixcolsep{4pc}  \xymatrix{ (A_1)_g \otimes (A_1)_h \ar[r]^-{\phi_1(g,h)} \ar[d]_{\phi(g)\phi(h)} & \ \ (A_1)_{gh} \ar[d]^{\phi(gh)} \\ (A_2)_g \otimes (A_2)_h \ar[r]_-{\phi_2(g,h)} & \ \ (A_2)_{gh}
    }
\end{align*}
commutes for all $g,h \in G$. This means that normalized $2$-cocycles $\phi_1,\phi_2: G \times G \to (\Bbbk^{\ast})^{n}$ differ by a coboundary $\d \phi$. Thus, we conclude that quasi-biangular $G$-algebras up to compatible $G$-graded Morita contexts are in bijection with $\amalg_{r=1}^{\infty} \amalg_{[r]\in (\Bbbk^{\ast})^n/S_n} H^2(G; (\Bbbk^{\ast})^n) \times \Hom(G, \text{Stab}_{S_n}(r))/\sim$ where the equivalence $\sim$ is given by conjugation. Using Theorem \ref{equivbicats} we derive the following proposition which was previously proven by Davidovich \cite{davidovich}.
\begin{prop}\label{spaceofequivtfts}(\cite{davidovich})
The set of equivalence classes of fully extended oriented $2$-dimensional $G$-equivariant TFTs, i.e. E-HFTs with $K(G,1)$-target, with values in $\Alg$ is in bijection with
\begin{align*}
    \pi_0 \text{Fun}^{\otimes}(\Bord_2^{G\times SO(2)},\Alg) \cong \coprod_{r=1}^{\infty} \coprod_{[r]\in (\Bbbk^{\ast})^n/S_n} H^2(G; (\Bbbk^{\ast})^n) \times \Hom(G, \text{Stab}_{S_n}(r))/\sim
\end{align*}
where the equivalence $\sim$ is given by conjugation.
\end{prop}
The fact that $\Alg$-valued fully-extended oriented $2$-dimensional $G$-equivariant TFTs are classified in two different ways, namely using the structured cobordism hypothesis and without using it, is an important step towards the verification of the $(G \times SO(2))$-structured cobordism hypothesis for $\Alg$-valued such TFTs. In this case, the $((G \times SO(2))$-structured cobordism hypothesis gives an equivalence of bigroupoids $\mathcal{E}\text{-}\mathcal{HFT}(X,\Alg) \simeq \Big(\big(\Alge^{fd}\big)^{\sim}\Big)^{h(G \times SO(2))}$. We have $ \mathcal{E}\text{-}\mathcal{HFT}(X,\Alg) \simeq \text{Frob}^G$ by Theorem \ref{equivbicats} and Davidovich \cite{davidovich} showed that the fundamental bigroupoid of $\Big(\big(\Alge^{fd}\big)^{\sim}\Big)^{h(G \times SO(2))}$ is equivalent to the bigroupoid $\text{Grp}_2(BG, \mathcal{G}_{\text{ori}})$ of $2$-functors, transformations, and modifications. Here the bigroupoid $BG$ has one object, $G$ worth of $1$-morphisms, and only identity $2$-morphisms and the bigroupoid $\mathcal{G}_{\text{ori}}$ has semisimple Frobenius algebras as objects, invertible bimodules compatible with Frobenius forms as $1$-morphisms, and invertible bimodule maps as $2$-morphisms (see Proposition 3.3.2 in \cite{davidovich}).  

Davidovich showed that every $2$-functor $F: BG \to \mathcal{G}_{\text{ori}}$ gives rise to a quasi-biangular $G$-algebra, and vice versa (Proposition 3.4.5, \cite{davidovich}). We define a $2$-functor $\mathcal{F} : \text{Frob}^G \to \text{Grp}_2(BG, \mathcal{G}_{\text{ori}})$ as follows. The image of a quasi-biangular $G$-algebra is the corresponding functor described above. Next, let $({}_{A_2}M_{A_1}, {}_{A_1}N_{A_2}, \kappa, \mu)$ be a $G$-graded compatible Morita context between two quasi-biangular $G$-algebras $(A_1,\eta_1)$ and $(A_2,\eta_2)$ which are determined by two triples $(\tau_i, \sigma_i,r_i)$ for $i=1,2$ such that $M_e \cong M^{\sigma}$ as in the proof of Proposition \ref{spaceofequivtfts} above. Then we define $\mathcal{F}(({}_{A_2}M_{A_1}, {}_{A_1}N_{A_2}, \kappa, \mu))$ to be natural transformation between the corresponding $2$-functors producing the bimodule $M^{\sigma}$ (see the proof of Proposition \ref{spaceofequivtfts} in \cite{davidovich}). Lastly, the image of an equivalence of $G$-graded Morita contexts under $\mathcal{F}$ is the modification producing the invertible bimodule map $M^{\sigma} \to M^{\sigma'}$ between the corresponding bimodules described above. 

Proposition \ref{spaceofequivtfts} implies that $\mathcal{F}$ is essentially surjective on objects. The fact that any natural transformation $\eta: F_1 \to F_2$ between functors $F_1, F_2 : BG \to \mathcal{G}_{\text{ori}}$ is isomorphic to a one producing bimodule of the form $M^{\sigma}$ (see proof of Proposition \ref{spaceofequivtfts} in \cite{davidovich}) implies that $\mathcal{F}$ is essentially full on $1$-morphisms. Using the similar arguments used in the proof of Theorem \ref{equivbicats}, one can show that $\mathcal{F}$ is fully-faithfull on $2$-morphisms. Consequently, the $2$-functor $\mathcal{F} :\text{Frob}^G \to \text{Grp}_2(BG, \mathcal{G}_{\text{ori}})$ is an equivalence by the Whitehead theorem for bicategories (Theorem \ref{whiteadforsym}). 
\begin{cor}\label{Gcobhyp}
\label{thm:Gcobhyp}
For any algebraically closed field $\Bbbk$ of characteristic zero, the $(G \times SO(2))$-structured cobordism hypothesis for $\Alg$-valued oriented E-HFTs with target $X \simeq K(G,1)$ holds true. 
\end{cor}

\section{Extended Unoriented X-HFTs and Their Classifications}\label{chapter_unoriented}
In this section by allowing X-manifolds and X-cobordisms to be nonorientable, we define $2$-dimensional extended unoriented X-HFTs and classify them. Both definition and classification of these theories are parallel with the oriented case and we only describe the changes. All of the manifolds and cobordisms in this chapter does not have any orientation data and they are not necessarily orientable.
\subsection{The unoriented X-cobordism bicategory and its presentation} In this section, we define $2$-dimensional extended unoriented HFTs with target $X\simeq K(G,1)$ where every element of $G$ has order two. To avoid repetition, from now on, we assume that $G$ is such a group and X is a pointed $K(G,1)$-space. The restriction to such groups is not essential but for convenience\footnote{Order two elements appear as a result of new generators and we avoid keeping track of which elements are required to have order two and which ones are not. See \cite{unoriented_hqft_kapustin} for the case of nonextended HFT with arbitrary aspherical targets.}. As in the oriented case, the $2$-dimensional extended unoriented X-cobordism bicategory plays the essential role and it is defined using X-halations (see Section \ref{halation_subsectionu}) as follows. 
\begin{definition}
\textit{The $2$-dimensional extended unoriented X-cobordism bicategory $\X\Bord_2^{\un}$} has 
\begin{itemize}
    \item quadruples $\{(M,\hat{M}_1,\hat{M}_2,\hat{\mathtt{g}}_2)\}$ consisting of compact $0$-manifolds equipped with two co-oriented X-halations as objects,
    \item quintuples $\{(A,\hat{A}_0,\hat{A}_1,T,\hat{\mathtt{p}}_1)\}$ consisting of $1$-dimensional marked X-manifolds equipped with two co-oriented X-halations as $1$-morphisms,
    \item isomorphism classes $\{[(S,\hat{S},R,\hat{\mathtt{F}})]\}$ of quadruples consisting of cobordism type $\<2\>$-X-surfaces equipped with a codimension zero X-halation as $2$-morphisms. 
\end{itemize}
\end{definition}
Similar to $\X\Bord_2$, the disjoint union operation is the symmetric monoidal product for $\X\Bord_2^{\un}$. 
\begin{definition}
For a symmetric monoidal bicategory $\mathcal{C}$, a \textit{$\mathcal{C}$-valued $2$-dimensional extended unoriented homotopy field theory with target X} is a symmetric monoidal $2$-functor from $\X\Bord_2^{\un}$ to $\mathcal{C}$.
\end{definition}

It is not hard to modify $G$-linear, $G$-planar, and $G$-spatial diagrams for the unoriented setting. We only need to consider sheets with no additional orientation data. In this case, each diagram produces an unoriented version of the corresponding X-manifold. There are new sheet data for the fold singularity coming from $\<2\>$-X-surface representatives of a M\"obius band. These $\<2\>$-X-surfaces are shown in Figure \ref{fig:unorientedgen}. This extra sheet data produces new relations coming from different possible gluings of these generators with themselves and with the earlier (orientable) generators. Figure \ref{fig: unorientedrel} shows these relations of $2$-morphisms of $\X\Bord_2^{\un}$ instead of the corresponding unoriented $G$-planar diagrams. By generalizing the equivalence relations on the set of unoriented $G$-planar diagrams given by unoriented $G$-spatial diagrams, we obtain the following unoriented $G$-planar decomposition theorem. 
%\begin{figure}[t]
%    \centering
%    \includesvg[]{unorientedgen1}
%    \caption{Additional generating $2$-morphisms of $\mathcal{X}\mathcal{G}_2^{\un}$}
%    \label{fig:unorientedgen}
%\end{figure}
\begin{figure}[t]
    \centering
    \def\svgwidth{\columnwidth}
    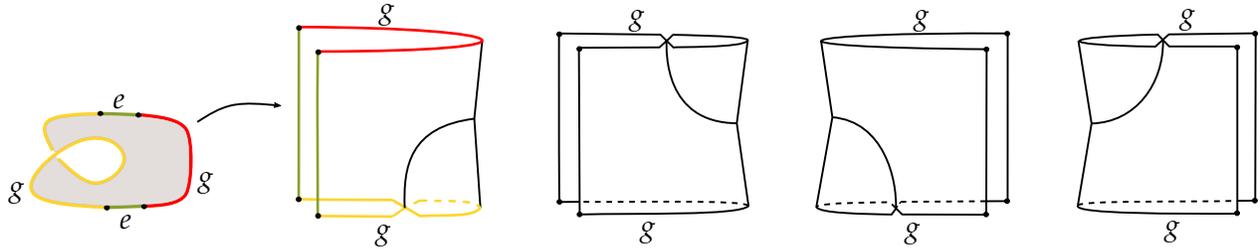
    \caption{Additional generating $2$-morphisms of $\mathcal{X}\mathcal{G}_2^{\un}$}
    \label{fig:unorientedgen}
\end{figure}

\begin{theorem}
The relative X-homeomorphism classes of unoriented cobordisms type $\<2\>$-X-surfaces are in bijection with the equivalence classes of unoriented $G$-planar diagrams.
\end{theorem}
%\begin{figure}[ht]
%    \centering
%    \includesvg{unorientedrel2}
%    \caption{Additional generating relations of $\mathcal{X}\mathcal{R}^{\un}$}
%    \label{fig: unorientedrel}
%\end{figure}
\begin{figure}[ht]
    \centering
    \def\svgwidth{\columnwidth}
    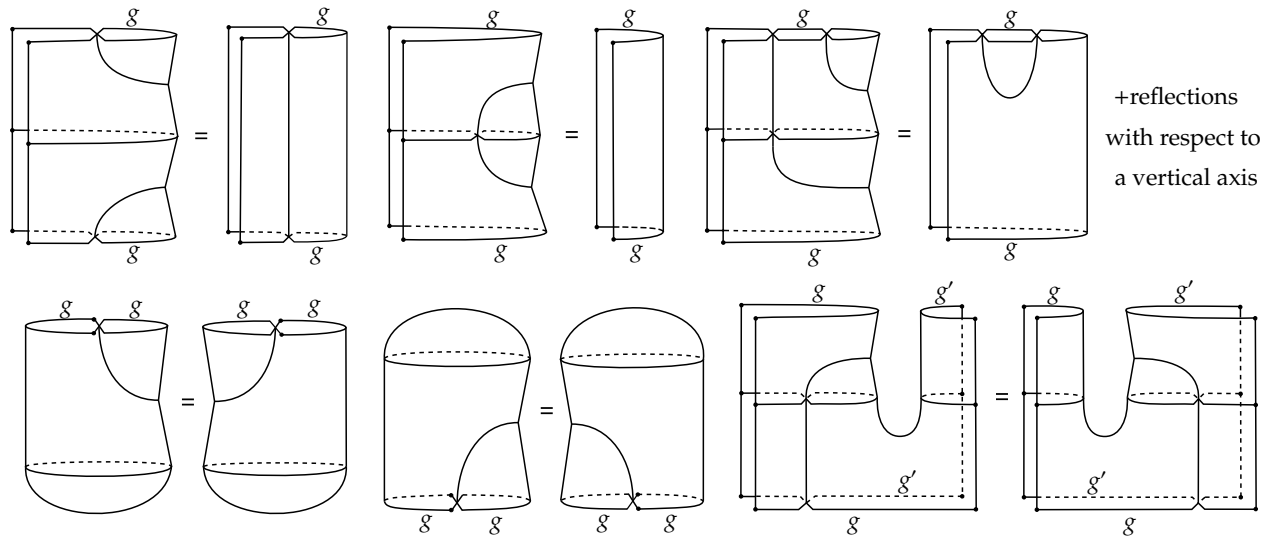
    \caption{Additional generating relations of $\mathcal{X}\mathcal{R}^{\un}$}
    \label{fig: unorientedrel}
\end{figure}

Parallel to oriented case, we define a new symmetric monoidal bicategory $\XB^{\PD,\text{un}}$ which has unoriented points as objects, unoriented $G$-linear diagrams as $1$-morphisms, and equivalence classes of unoriented $G$-planar diagrams as $2$-morphisms. The unoriented $G$-planar decomposition theorem implies that $\XB^{\PD,\un}$ is symmetric monoidally equivalent to $\X\Bord_2^{\un}$ via Whitehead theorem for symmetric monoidal bicategories. The correspondence between G-planar diagrams and the string diagrams for unbiased semistrict symmetric monoidal bicategories gives the following result.
\begin{theorem}
The symmetric monoidal bicategory $\XB^{\PD,\un}$ is computadic unbiased semistrict symmetric monoidal $2$-category with the presentation $\XP^{\un}=(\mathcal{XG}^{\un}_0,\mathcal{XG}^{\un}_1,\mathcal{XG}^{\un}_2,\mathcal{XR}^{\un})$ which has one generating object $\{ \noktac \}$, $G$-linear diagrams of $\big\{ \bos, \leftelbov,\rightelbov\big\}_{g \in G}$ as the generating $1$-morphisms, $G$-planar diagram versions of elements in Figures \ref{fig:generators} and \ref{fig:unorientedgen} as the generating $2$-morphisms, and $G$-planar diagram versions of pairs in Figures \ref{fig:relations} and \ref{fig: unorientedrel} as the generating relations. 
\end{theorem}

The bicategory $\mathcal{E}$-$\mathcal{HFT}^{\un}(X,\mathcal{C})$ of $2$-dimensional extended unoriented X-HFTs has $\mathcal{C}$-valued unoriented extended X-HFTs as objects, symmetric monoidal transformations as $1$-morphisms, and symmetric monoidal modifications as $2$-morphisms. For any given symmetric monoidal bicategory $\mathcal{C}$, using the cofibrancy theorem, we state the classification of $\mathcal{C}$-valued $2$-dimensional extended unoriented X-HFTs as the equivalence of bicategories $\mathcal{E}\text{-}\mathcal{HFT}^{\un}(\X,\mathcal{C})\simeq \XP^{\un}(\mathcal{C})$.
\begin{remark}
There is a symmetric monoidal $2$-functor $\text{Forget}^{\text{or}}: \X\Bord_2 \to \X\Bord_2^{\text{un}}$ given by forgetting the orientation. In the same way, any oriented or unoriented $2$-dimensional extended TFT leads to an oriented or unoriented extended HFT respectively by forgetting the X-manifold data. The following diagram indicates the universality of unoriented $2$-dimensional extended TFTs in this context
\begin{align*}
 \xymatrixcolsep{5pc}   \xymatrix{ \mathcal{E}\text{-}\mathcal{TFT}^{\un}(\mathcal{C}) \ar[r]^{\text{Forget}^{\ori}} \ar[d]_{\text{Forget}^{X}} & \mathcal{E}\text{-}\mathcal{TFT}(\mathcal{C}) \ar[d]^{\text{Forget}^{X}}\\
    \mathcal{E}\text{-}\mathcal{HFT}^{\un}(X,\mathcal{C}) \ar[r]_{\text{Forget}^{\ori}} & \mathcal{E}\text{-}\mathcal{HFT}(X,\mathcal{C})
    }
\end{align*}
where $\mathcal{E}\text{-}\mathcal{TFT}^{\un}(\mathcal{C})$ and $\mathcal{E}\text{-}\mathcal{TFT}(\mathcal{C})$ are defined similarly using $\Bord_2^{\un}$ and $\Bord_2$ respectively.
\end{remark}

\subsection{$\Alg$-valued 2-dimensional extended unoriented X-HFTs}
Tagami~\cite{unorientedhq} classified $2$-dimensional nonextended unoriented HFTs by extended crossed Frobenius $G$-algebras (see also \cite{unoriented_hqft_kapustin}). Similar to oriented case our goal is to understand the relation between his classification and the restriction of $\Alg$-valued $2$-dimensional extended unoriented HFTs to circles and cobordisms between them. 

Firstly, we introduce necessary algebraic notions. Let $K$ be a $G$-algebra and $V$ be a $(K,K^{op})$ bimodule. \textit{Conjugate} of $V$ is the $(K,K^{op})$ bimodule $\underline{V}$ obtained by turning actions around. Similarly, the \textit{conjugate} of a graded Morita context $\zeta=({}_{K^{op}}U_{K},{}_{K}V_{K^{op}},\tau,\mu)$ is given by $\underline{\zeta}=({}_{K^{op}}\underline{U}_{K},{}_{K}\underline{V}_{K^{op}},\underline{\tau},\underline{\mu})$. We generalize stellar algebras introduced in \cite{schommer} to stellar $G$-algebras as follows.
\begin{definition}
A \textit{stellar $G$-algebra} is a $G$-algebra $K=\oplus_{g \in G}K_g$, equipped with a $G$-graded Morita context $\zeta=({}_{K^{op}}U_{K},{}_{K}V_{K^{op}},\tau,\mu)$ together with an isomorphism of $G$-graded Morita contexts $\sigma : \zeta \cong \underline{\zeta}$ such that $\sigma \circ \underline{\sigma}$ is the identity isomorphism where $\underline{\sigma}$ is the induced isomorphism between $\underline{\underline{\zeta}}$ and $\underline{\zeta}$.
\end{definition}
Stellar structure on a $G$-algebra can be transferred along a graded Morita context as follows. Let $\rho=({}_{K}U'_L,{}_{L}V'_K,\kappa,\nu)$ be a $G$-graded Morita context between $G$-algebras $K$ and $L$ and let $(K,\zeta,\sigma)$ be a stellar structure on $K$ with $\zeta=({}_{K^{op}}U_{K},{}_{K}V_{K^{op}},\tau,\mu)$. Then, $(L,\rho_{\ast}\zeta,\rho_{\ast}\sigma)$ is a stellar algebra where $\rho_{\ast}\zeta= 
\big({}_{L^{op}}\underline{U'} \otimes_{K^{op}} U \otimes_K U'_{L}, {}_{L}V'\otimes_K V \otimes_{K^{op}} \underline{V'}_{L^{op}}, \underline{\kappa}\otimes \tau \otimes \kappa, \underline{\nu} \otimes \mu \otimes \nu \big)$ and $\rho_{\ast}\sigma:\rho_{\ast}\zeta \cong \underline{\rho_{\ast}\zeta}$ is given by $\sigma$.
\begin{definition}
Let $(K,\zeta,\sigma)$ be a stellar $G$-algebra with $\zeta=({}_{K^{op}}U_K,{}_{K}V_{K^{op}},\tau,\mu)$ and let $(K,\eta)$ be a quasi-biangular $G$-algebra. Stellar structure is said to be \textit{compatible} with quasi-biangular $G$-algebra if there exists an element $\sum_j a_j \otimes b_j \in K_e \otimes K_e$ giving the central element $z = \sum_j b_j a_j$ such that the following diagrams\footnote{Tensors in diagrams are taken over $K$,$K^{op}$ or $K\otimes_{\Bbbk}K^{op}$.} commute 
\begin{align*}
    \xymatrixcolsep{4pc}\xymatrix{K \otimes K \ar[r]^-{\tau \otimes \text{id}} \ar[d]_{\text{id} \otimes \tau} & (V \otimes U) \otimes K \ar[r]^-{\sigma \otimes \text{id}} & (\underline{V} \otimes \underline{U}) \otimes K \ar[r]^-{\underline{\tau}^{-1}\otimes \text{id}} &K \otimes K \ar[d]^-{\eta} \\
    K \otimes (V \otimes U) \ar[r]_-{\text{id}\otimes \sigma} & K \otimes (\underline{V} \otimes \underline{U}) \ar[r]_-{\text{id}\otimes \underline{\tau}^{-1}} & K \otimes K \ar[r]_-{\eta} & \Bbbk
    }\\
    \xymatrixcolsep{4pc}\xymatrix{  \Bbbk \ar[r]^-{\iota } \ar[d]_-{ \iota} & K \otimes K \ar[r]^-{\tau \otimes \text{id}} & (V \otimes U) \otimes K \ar[r]^-{\sigma \otimes \text{id}} & (\underline{V} \otimes \underline{U}) \otimes K \ar[d]^-{\underline{\tau}^{-1}\otimes \text{id}}  \\ 
    K \otimes K \ar[r]_-{\text{id}\otimes \tau} & K \otimes (V \otimes U) \ar[r]_-{\text{id} \otimes \sigma}  & K \otimes (\underline{V} \otimes \underline{U}) \ar[r]_-{\text{id}\otimes \underline{\tau}^{-1}} & K \otimes K
    }\\
    \xymatrixcolsep{4pc}\xymatrix{ K \otimes K \ar[r]^-{\tau \otimes \text{id}} \ar[d]_-{\text{id} \otimes \tau} & (V \otimes U) \otimes K \ar[r]^-{\sigma \otimes \text{id}}  & (\underline{V} \otimes \underline{U}) \otimes K \ar[r]^-{\underline{\tau}^{-1} \otimes \text{id}} & K \otimes K \ar[d]^-{\xi}  
  \\ K \otimes (V \otimes U) \ar[r]  \ar[r]_-{\text{id} \otimes \sigma} & K \otimes (\underline{V} \otimes \underline{U}) \ar[r]_-{\text{id} \otimes \underline{\tau}^{-1}} & K \otimes K \ar[r]_-{\xi} & K \otimes K }
\end{align*}
where $\iota(1)|_{K_e \otimes K_e}= \sum_j a_j \otimes b_j$ and $\xi: {}_{K_1}K_{K_2} \otimes {}_{K_3}K_{K_4} \to {}_{K_1}K_{K_4} \otimes {}_{K_3}K_{K_2}$ is a graded bimodule map with $K_i=K$ for $i=1,2,3,4$ and $\xi(1)=\sum_i p_i^e \otimes q_i^e$ is an inner product element of the principal component. We call such a compatible quadruple $(K,\eta,\zeta,\sigma)$ a \textit{quasi-biangular stellar $G$-algebra}.
\end{definition}
\begin{definition} \label{morphGstellar}
A morphism of quasi-biangular stellar $G$-algebras $(K,\eta^k, \zeta^K,\sigma^K)$ and $(L,\eta^L, \zeta^L, \allowbreak \sigma^L)$ is a compatible $G$-graded Morita context $\rho=({}_{K}U_L,{}_{L}V_K,\tau,\mu)$ together with an equivalence of $G$-graded Morita contexts $\phi: \zeta^L \to \rho_{\ast}\zeta^K $ such that $ \rho_{\ast}\sigma^K\circ \phi= \underline{\phi} \circ \sigma^L$ where $\underline{\phi}: \underline{\zeta}^L \to \underline{\rho_{\ast}\zeta}^K$. Two such morphisms $(\rho,\phi)$ and $(\rho',\phi')$ are isomorphic if there exists an equivalence of $G$-graded Morita contexts $\alpha : \rho \to \rho'$ such that $\phi'=\alpha \circ \phi$ and $\underline{\phi'}=\underline{\alpha}\circ \underline{\phi}$ for $\underline{\alpha}: \underline{\rho} \to \underline{\rho'}$.
\end{definition}
\begin{theorem} \label{unorclass}\label{thm:unorclass}
Let $G$ be a group with each nonidentity element having order $2$ and X be a $K(G,1)$-space. Any $\Alg$-valued $2$-dimensional extended unoriented X-HFT $Z: \X\Bord_2^{\un} \to \Alg$ whose precomposition $\XB^{\PD,\un} \xrightarrow{\simeq} \X\Bord_2^{\un} \xrightarrow{Z} \Alg$ gives a strict symmetric monoidal $2$-functor determines a quasi-biangular stellar $G$-algebra $(A,\eta,\zeta,\sigma)$. Conversely, for any quasi-biangular stellar $G$-algebra $(A,\eta, \zeta,\sigma)$ there exists an $\Alg$-valued $2$-dimensional extended unoriented X-HFT.
\end{theorem}
\begin{proof}
Let $Z: \X\Bord_2^{\un} \to \Alg$ be such a $2$-dimensional extended unoriented HFT. The cofibrancy theorem implies that there exists an object $Z'$ in $\XP^{\un}(\Alg)$ such that $\imath(Z')$ is the composition $\XB^{\PD,\un} \xrightarrow{\simeq} \X\Bord_2^{\un} \xrightarrow{Z} \Alg$ where $\imath: \XP^{\un}(\Alg)\to \text{SymMon}(\XB^{\PD,\un}, \Alg)$ is the equivalence of bicategories.

Following the proof of Theorem \ref{orclass}, we have a strongly graded $G$-algebra $A= \oplus_{g \in G} A_g$ where $Z'(\noktac)=A_e$. We also have $G$-graded $(A\otimes A,\Bbbk)$ and $(\Bbbk,A \otimes A)$-bimodules $M= \oplus_{g \in G}M_g$, $N=\oplus_{g \in G} N_g$ respectively. By turning actions around we obtain $(A,A^{op})$-bimodule $M$ and $(A^{op},A)$-bimodule $N$. 

Bimodule maps in $Z'_2(\mathcal{X}^{\un}\mathcal{G}_2)$ corresponding to cusp generators (subject to relations) yield a $G$-graded Morita context $\zeta=({}_{A^{op}}N_A,{}_{A}M_{A^{op}},f_1,f_2)$ between $A$ and $A^{op}$ where $f_1: {}_{A}A_A \to {}_{A}M \otimes_{A^{op}} N_{A}$ and $f_2:{}_{A^{op}}N \otimes_{A}M_{A^{op}} \to {}_{A^{op}}A^{op}_{A^{op}}$ are invertible $G$-graded bimodule maps. Bimodule maps in $Z_2'(\mathcal{X}^{\un}\mathcal{G}_2)$ for the Morse generators satisfying relations imply that $(A,\eta)$ is a quasi-biangular $G$-algebra. The generators in Figure \ref{fig:unorientedgen} give the following graded bimodule maps in $Z'_2(\mathcal{X}^{\un}\mathcal{G}_2)$
\begin{align*}
    \sigma_1 &: {}_{A}M_{A^{op}} \to {}_{A}\underline{M}_{A^{op}} \hspace{2cm}\sigma_2: {}_{A^{op}}N_A \to {}_{A^{op}}\underline{N}_A\\
    \sigma_1' &: {}_{A}\underline{M}_{A^{op}} \to {}_{A}M_{A^{op}} \hspace{2cm}\sigma_2': {}_{A^{op}}\underline{N}_A \to {}_{A^{op}}N_A.
\end{align*}
These graded bimodule maps are subject to the relations in Figure \ref{fig: unorientedrel}. Thereby, we have $\sigma_1' \circ \sigma_1= \text{id}_M$, $\sigma_1 \circ \sigma_1'= \text{id}_{\underline{M}}$, $\sigma_2' \circ \sigma_2 =\text{id}_N$, and $\sigma_2 \circ \sigma_2'= \text{id}_{\underline{N}}$. These isomorphisms of bimodules lead to an isomorphism $\sigma: \zeta \cong \underline{\zeta}$. Applying $\sigma$ to $\underline{\underline{\zeta}}$ gives another isomorphism $\underline{\sigma}: \underline{\underline{\zeta}}\to \underline{\zeta}$ whose composition with $\sigma$ gives $\sigma \circ \underline{\sigma}: \underline{\underline{\zeta}}\cong \zeta$. Third relation on the first row of Figure \ref{fig: unorientedrel} and its reflection indicate that compositions of bimodule maps $\underline{\underline{M}} \to \underline{M} \to M$ and $\underline{\underline{N}} \to \underline{N} \to N$ are identity maps. 

Thus, additional generators and relations among them lead to a stellar structure $(\zeta,\sigma)$ on the quasi-biangular $G$-algebra $A$. Remaining relations imply the compatibility giving the quasi-biangular stellar $G$-algebra $(A,\eta, \zeta,\sigma)$. For any quasi-biangular stellar $G$-algebra, one constructs an object of $\XP^{\un}(\Alg)$ by assigning values to generating objects, $1$-morphisms, and $2$-morphisms of $\XP$ satisfying generating relations using the above arguments. Then, this object gives a strict symmetric monoidal $2$-functor $\XB^{\PD,\un} \to \Alg$ whose composition with the equivalence $\X\Bord_2^{\un} \xrightarrow{\sim} \XB^{\PD,\un}$ produces the desired unoriented extended X-HFT.
\end{proof}
Similar to oriented case every $2$-dimensional extended unoriented HFT with target X produces a nonextended one by precomposition $\X\text{Cob}^{\text{un}}_2 \to \X \mathbbmss{Cob}^{\mathbbmss{un}}_2 \to \Alg$ where $\X\text{Cob}_2^{\text{un}}$ and $\X\mathbbmss{Cob}_2^{\mathbbmss{un}}$ are defined just as $\X\text{Cob}_2$ and $\X\mathbbmss{Cob}_2$ using unoriented X-manifolds. In the unoriented case extended crossed Frobenius $G$-algebras plays an important role in the study of $2$-dimensional nonextended unoriented X-HFTs and they are defined as follows.
\begin{definition}\label{extendedalg} (\cite{unorientedhq})
Let $(K,\eta,\varphi)$ be a crossed Frobenius $G$-algebra over $\Bbbk$. An \textit{extended} structure on $K$ consists of a $\Bbbk$-module homomorphism $\Phi: K \to K$ and a family of elements $\{ \theta_g \in K_e\}_{g \in G}$ satisfying the following conditions
\begin{enumerate}
    \item  $\Phi(K_g) \subset K_g$ and $\Phi(\theta_g)=\theta_g$ for all $g\in G$,
    \item $\Phi \circ \varphi_g = \varphi_g \circ \Phi$ for all $g \in G$, 
    \item $\Phi(vw)= \Phi(w)\Phi(v)$ for any $v,w \in K$ and $\Phi(1_K)=1_K$,
    \item $\Phi^2 = id$,
    \item $\eta \circ (\Phi \otimes \Phi)= \eta$,
    \item for any $g,h,l \in G$ and $v \in K_{gh}$, we have
    \begin{align*}
        m \circ( \Phi \circ \varphi_l) \circ \Delta_{g,h}(v)= \varphi_l(\theta_{gl}\theta_l v),\\
        m \circ (\varphi_l \otimes \Phi) \circ \Delta_{g,h}(v) =\varphi_l(\theta_{hl}\theta_l v),
    \end{align*}
where $\Delta_{g,h}: K_{gh} \to K_g \otimes K_h$ is defined by the equation $(id_g \otimes \eta) \circ (\Delta_{g,h} \otimes id_h) =m$. Such a map $\Delta_{g,h}$ is uniquely determined since $\eta$ is nondegenerate and each $K_g$ is finitely generated,
    \item $\Phi(\theta_h v)=\varphi_{hg}(\theta_{hg}v)$ for any $g,h \in G$ and $v \in K_g$,
    \item $\varphi_h(\theta_g)= \theta_g$ for any $g,h \in G$,
    \item for any $g,h,l \in G$, we have $\theta_g \theta_h \theta_l= q(1)\theta_{ghl}$ where $q: \Bbbk \to K_e$ is defined as follows; let $\{a_i \in K_{gh}\}_{i=1}^n$ and $\{b_i \in K_{gh}\}_{i=1}^n$ be families of elements of $K_{gh}$ satisfying the equation $\sum_i \eta(b_i \otimes v) a_i = \varphi_{hl}(v)$ for any $v \in K_{gh}$. As in $(3)$, such $a_i$ and $b_i$ are uniquely determined and $q(1)= \sum_i a_i b_i$.
\end{enumerate}
\end{definition}
\begin{theorem}(Tagami, \cite{unorientedhq})
Let $G$ be a group with each nonidentity element having order $2$. There is a bijection between the isomorphism classes of $2$-dimensional unoriented HFTs with target $X \simeq K(G,1)$ and the isomorphism classes of extended crossed Frobenius $G$-algebras.
\end{theorem}
\begin{cor}
Assume that $Z: \X\Bord_2^{\un} \to \Alg$ determines a quasi-biangular stellar $G$-algebra $(A,\eta,\zeta,\sigma)$. The stellar structure $(\zeta,\sigma)$ gives an extended structure on the crossed Frobenius $G$-algebra $Z_G(A)$. Moreover, the corresponding $2$-dimensional X-HFT is the unoriented X-HFT obtained by restricting $Z$ to $\X\mathbbmss{Cob}^{\mathbbmss{un}}_2$.
\end{cor}
\begin{proof} 
We have a crossed Frobenius $G$-algebra $(Z_G(A), \eta|_{Z_G(A)}, \{\varphi|_{Z_G(A)}\}_{g \in G})$. By Tagami's classification, the $2$-dimensional unoriented HFT given by the restriction of $Z$ to circles and cobordism between them induces an extended structure on $Z_G(A)$. We claim that homomorphism $\Phi$ and elements $\{ \theta_g \in Z_G(A)_e\}_{g \in G}$ come from the stellar structure $(\zeta,\sigma)$ on $A$. 

In \cite{unorientedhq}, for each $g \in G$ the restriction $\Phi|_{Z_G(A)_g} : Z_G(A)_g \to Z_G(A)_g$ is the involution induced by an orientation reversing homeomorphism of a $g$-labeled circle. In the extended case this morphism is given by additional $2$-morphisms (Figure \ref{fig:unorientedgen}). More precisely, $\Phi|_{Z_G(A)_g}: A_e \otimes_{A_e \otimes A_e^{op}} A_g \to A_e \otimes_{A_e \otimes A_e^{op}}A_g$ is defined by $\Phi(a\otimes b)= a \otimes \Phi_g(b)$ where $\Phi_g$ is defined so that the following diagram 
\begin{align*}
    \xymatrix{{}_{A_e}M_g \otimes_{A_e^{op}} (N_e)_{A_e} \ar[r]^-{\cong} \ar[d]_{\sigma_1 \otimes \text{id}} & {}_{A_e}(A_g)_{A_e} \ar[d]^{\Phi_g= Z\big(\syma \big)} \\ {}_{A_e}\underline{M_g} \otimes_{A_e^{op}}(N_e)_{A_e} \ar[r] & {}_{A_e}(A_g)_{A_e}
    }
\end{align*}
commutes. It is not hard to see that $\Phi$ reverses the orientation of the oriented (input) circle. In \cite{unorientedhq}, for every $g \in G$ the element $\theta_g$ is the image of HFT under the M\"obius strip whose boundary is labeled by $g^2=e$ where the M\"obius strip is considered as the cobordism from the empty $1$-manifold to the boundary circle. In the extended case, $\theta_g \in A_e\otimes_{A_e \otimes A_e^{op}}A_e$ is the image of $1 \in \Bbbk$ under the following composition; $\{g,g\}$-labeled cap morphism followed by new generators (see Figure \ref{fig: unorientedrel}) which is composed with module actions turning boundary labels into $\{e,e\}$ (see Figure \ref{fig:Gcup}). 

We see that the involution $\Phi$ and elements $\{\theta_g\}_{g \in G}$ are defined according to their topological description given in \cite{unorientedhq}. Hence, $(Z_G(A), \eta|_{Z_G(A)}, \{\varphi_g|_{Z_G(A)}\}_{g\in G},\Phi, \{\theta_g\}_{g \in G})$ is an extended crossed Frobenius $G$-algebra which by definition corresponds to the restriction of $Z: \X\Bord_2^{\un} \to \text{Alg}^2_{\Bbbk}$ to X-circles and unoriented X-cobordisms between them. 
\end{proof}

\subsection{The bicategory of 2-dimensional extended unoriented X-HFTs}
In order to upgrade Theorem \ref{unorclass} to an equivalence of bicategories, we study morphisms in the bicategory $\XP^{\un}(\Alg)$. Let $\alpha$ be a $1$-morphism from $Z_0$ to $Z_1$ giving quasi-biangular stellar $G$-algebras $(A,\eta,\zeta, \sigma)$ and $(A',\eta',\zeta',\sigma')$ respectively. We know from the oriented case that $\alpha$ gives a compatible $G$-graded Morita context $\xi$ between $G$-algebras $A$ and $A'$. Assuming $\alpha_0(\noktac)= {}_{A_e}R_{A_e'}$ and $\xi=({}_{A}R_{A'}, {}_{A'}R'_{A},\tau,\mu)$ naturality with respect to the first generator in Figure \ref{fig:unorientedgen} is the commutativity of the following diagram
\begin{align*}
   \xymatrixcolsep{5pc}\xymatrix{{}_{A_e'}(M'_g)_{(A_e')^{op}} \ar[d]_{\sigma'=Z_1 \big(\syma \big)} \ar[r]^-{\alpha_1(\leftelbovg)} & {}_{A_e'}R' \otimes_{A_e} M_g \otimes_{A_e^{op}} \underline{R'}_{(A_e')^{op}} \ar[d]^{
   \xi_{\ast}\sigma=Z_0\big(\syma \big)}  \\
    {}_{A_e'}(\underline{M_g'})_{(A_e')^{op}} \ar[r]_-{\alpha_1(\leftelbovt)} & {}_{A_e'}R' \otimes_{A_e} \underline{M_g} \otimes_{A_e^{op}} \underline{R'}_{(A_e')^{op}}
    }
\end{align*}
where $M$ and $M'$ are components of the graded Morita contexts $\zeta$ and $\zeta'$ respectively. There are similar commutative diagrams for the remaining three generators. These diagrams indicate that the $G$-graded Morita context $\xi$ gives an equivalence of $G$-graded Morita contexts $\zeta'$ and $\xi_{\ast}\zeta$ with $\underline{\alpha} \circ \sigma'= \xi_{\ast}\sigma \circ \alpha$. In other words, $\alpha$ leads to a morphism of stellar $G$-algebras (see Definition \ref{morphGstellar}). 

Let $\theta : \alpha^1 \to \alpha^2$ be a $2$-morphism in $\XP(\Alg)$ with $\theta_0(\noktac)= {}_{A_e}R_{A_e'}\to {}_{A_e}P_{A_e'}$. In the oriented case we observed that $\theta$ induces an equivalence of $G$-graded Morita contexts $\xi=({}_{A}R_{A'},{}_{A'}R'_{A},\tau,\mu)$ and $\rho=({}_{A}P_{A'},{}_{A'}P'_{A},\kappa,\nu)$. Naturality of $\theta_0(\noktac)$ with respect to $\leftelbovt$ is the commutativity of the following diagram
\begin{align*}
     \xymatrixcolsep{5pc}\xymatrix{ {}_{A_e'}(\underline{M_g'})_{(A_e')^{op}} \ar[d]_{\text{id}} \ar[r]^-{\alpha^1_1(\leftelbovt)}& {}_{A_e'}R' \otimes_{A_e} \underline{M_g} \otimes_{A_e^{op}} \underline{R'}_{(A_e')^{op}} \ar[d]^{\theta_0(\noktac \ \noktac)} \\ 
    {}_{A_e'}(\underline{M_g'})_{(A_e')^{op}} \ar[r]_-{\alpha_1^2(\leftelbovt)} & {}_{A_e'}P' \otimes_{A_e} \underline{M_g} \otimes_{A_e^{op}} \underline{P'}_{(A_e')^{op}}
    }
\end{align*}
and there is a similar diagram for the naturality with respect to $\rightelbovt$. Naturality for $\{\rightelbov\}_{g \in G}$ and $\{\leftelbov\}_{g \in G}$ gives $\alpha^2 = \theta \circ \alpha^1$ and naturality for $\{\leftelbovt\}_{g \in G}$ and $\{\rightelbovt\}_{g \in G}$ gives $\underline{\alpha_2}= \underline{\theta}\circ \underline{\alpha^1}$. In other words, $\theta$ gives an isomorphism of stellar $G$-algebra morphisms (see Definition \ref{morphGstellar}).

These observations lead us to define a bicategory $\text{Frob}^G_{\ast}$ which has quasi-biangular stellar $G$-algebras as objects, their morphisms as $1$-morphisms, and isomorphisms of quasi-biangular stellar $G$-algebra morphisms as $2$-morphisms. Above arguments imply that there exists a $2$-functor $\mathcal{F}' : \XP^{\un}(\Alg) \to \text{Frob}^G_{\ast}$. Composing $\mathcal{F}'$ with the equivalence $\mathcal{E}\text{-}\mathcal{HFT}^{\un}(X,\Alg) \simeq \XP^{\un}(\Alg)$ we define the $2$-functor $\mathcal{F}$. 
\begin{theorem}\label{unorbicatsinif}
The $2$-functor $\mathcal{F}:\mathcal{E}\text{-}\mathcal{HFT}^{\un}(X,\Alg) \to \text{Frob}^G_{\ast}$ is an equivalence of bicategories.
\end{theorem}
\begin{proof}
Proof follows from above arguments and Whitehead theorem for bicategories.
\end{proof}

\subsection{The $(G\times O(2))$-structured cobordism hypothesis}
Parallel to oriented case, we want to compare Theorem \ref{unorbicatsinif} with the classification given by the $(G\times O(2))$-structured cobordism hypothesis. To do this, we need to understand homotopy $(G \times O(2))$-fixed points in $(\Alge^{fd})^{\sim}$ which are given by
\begin{align*}
    \Big(\big(\Alge^{\text{fd}}\big)^{\sim}\Big)^{h(G \times O(2))} &=    \text{Map}_G \big(EG,\text{Map}_{O(2)}(EO(2),\mathsf{Alg})\big)
\end{align*}
where $G$ acts on invariant maps trivially and $\mathsf{Alg}$ is the $2$-type corresponding to the $\infty$-groupoid $\big(\Alge^{\text{fd}}\big)^{\sim}$. Recall that unoriented Grassmannian $\text{Gr}(2,\R^{\infty})$ is a model for $BO(2)$ and Stiefel manifold $\text{V}(2,\R^{\infty})$ is for $EO(2)$. The universal principal $O(2)$-bundle $p: \text{V}(2,\R^{\infty})\to \text{Gr}(2,\R^{\infty})$ is given by $p((e_1,e_2))=\<e_1,e_2\>$ i.e. the plane generated by the orthonormal $2$-frame $(e_1,e_2)$. 
\begin{lemma} \label{reflection}
Reflection invariant maps in $\text{Map}(\text{V}(2,\R^{\infty}),\mathsf{Alg})$ determine stellar structures on $\Bbbk$-algebras.
\end{lemma}
\begin{proof}
A reflection $\omega$ in $O(2)$ acts on $\mathsf{Alg}$ by sending a $\Bbbk$-algebra $A$ to its opposite algebra $A^{op}$. Let $f$ be a reflection invariant map with $f((e_1,e_2))=A$. Let $\gamma$ be a representative of the nontrivial element of $\pi_1(\text{Gr}(2,\R^{\infty}),\<e_1,e_2\>)\cong \mathbb{Z}/2\Z$. Lift $\gamma$ to $\tilde{\gamma}$ starting at $(e_1,e_2)$ and ending at $\omega((e_1,e_2))$ (see Figure \ref{fig:reflection}). Then, $f(\tilde{\gamma})$ is a $(A^{op},A)$-bimodule $M$ and invariance under $\omega$ means $f(\omega(\tilde{\gamma}))={}_{A^{op}}\underline{M}_{A}=\omega (M)$. Lifting $\gamma$ to $\tilde{\gamma}'$ starting at $\omega((e_1,e_2))$ gives a path ending at $(e_1,e_2)$. Similarly, $f(\widetilde{\gamma}')$ is a $(A,A^{op})$-bimodule $N$ and we have $f(\omega(\widetilde{\gamma}'))={}_{A}\underline{N}_{A^{op}}=\omega(N)$. 
%\begin{figure}[ht]
%    \centering
%    \includesvg{reflection2}
%    \caption{A reflection invariant map}
%    \label{fig:reflection}
%\end{figure}
\begin{figure}[ht]
    \centering
    \def\svgwidth{\columnwidth}
    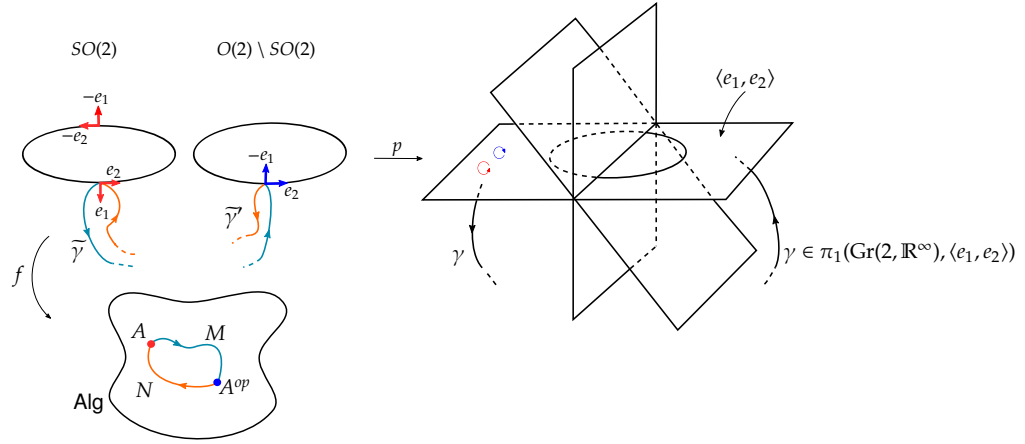
    \caption{A reflection invariant map}
    \label{fig:reflection}
\end{figure}
Loops $\widetilde{\gamma}' \ast \widetilde{\gamma}$ and $\widetilde{\gamma} \ast \widetilde{\gamma}'$ bound disks since $\text{V}(2,\R^{\infty})$ is contractible. In other words, there exist basepoint fixing homotopies which take these loops to constant loops at $(e_1,e_2)$ and $\omega((e_1,e_2))$ respectively. Images of the second homotopy and the time reversed version of the first homotopy under $f$ yield invertible bimodule maps $\mu: {}_{A^{op}}M \otimes_A N_{A^{op}} \to {}_{A^{op}}A^{op}_{A^{op}}$ and $\tau: {}_{A}A_{A} \to {}_{A}N \otimes_{A^{op}}M_A$ respectively. The compositions of homotopies corresponding to the conditions on $\mu$ and $\tau$ to form a Morita context are the constant homotopies of paths $\widetilde{\gamma}$ and $\widetilde{\gamma}'$. Thus, $\zeta= ({}_{A^{op}}M_{A}, {}_{A}N_{A^{op}},\tau,\mu)$ is a Morita context. Similarly, loops $\widetilde{\gamma} \ast \omega(\widetilde{\gamma})$ and $\widetilde{\gamma}' \ast \omega(\widetilde{\gamma}')$ bound, which implies that there is an equivalence of Morita contexts $\sigma: \zeta \cong \underline{\zeta}$. Since the order of reflection is two, we have $\sigma \circ \underline{\sigma}=\text{id}$. Thus, any reflection invariant map leads to stellar algebra structures on algebras. 
\end{proof}
\begin{lemma} \label{go2fix}
For an algebraically closed field $\Bbbk$ of characteristic zero, homotopy $(G \times O(2))$-fixed points of $(\Alge^{fd})^{\sim}$ are quasi-biangular stellar $G$-algebras.
\end{lemma}
\begin{proof}
Serre automorphism trivializes homotopy $SO(2)$-action (see \cite{davidovich}), which turns the space of homotopy $(G\times SO(2))$-fixed points into $\text{Map}_G(EG,\text{Map}(\widetilde{\text{Gr}}(2,\R^{\infty}), \mathsf{Alg}))$. Davidovich \cite{davidovich} showed that homotopy $SO(2)$-fixed points are semisimple symmetric Frobenius $\Bbbk$-algebras. Then, understanding homotopy $O(2)$-fixed points means understanding invariance under reflections. Using Lemma \ref{reflection} we conclude that homotopy $O(2)$-fixed points are finite dimensional semisimple symmetric Frobenius $\Bbbk$-algebras with a stellar structure. 

The stellar structure is compatible with the Frobenius form as follows. A Frobenius form on a $\Bbbk$-algebra $A$ is determined by a central element which is the image of $1$ under a bimodule map $z:{}_{A}A_A \to {}_{A}A_A$. Geometrically, $z(1)$ is an element of $\pi_2(\text{Map}(BSO(2),\mathsf{Alg}_r),f)\cong(\Bbbk^{\times})^r$ where the algebra $A \in \mathsf{Alg}_r \subset \mathsf{Alg} = \amalg_{r=1}^{\infty} \mathsf{Alg}_r$ is isomorphic to $\text{End}(V_1)\times \text{End}(V_2) \times \dots \times \text{End}(V_r)$ under Artin-Wedderburn isomorphism for finite dimensional $\Bbbk$-vector spaces $V_1,\dots,V_r$.

Compatibility means that (horizontal) composition of $z$ with $\underline{\zeta}$ yields $z$ again. Geometrically, this corresponds to conjugating the representing sphere based at $f$ with loops in $\mathsf{Alg}_r$ given by bimodules of $\underline{\zeta}$. Since this loop is contractible conjugation does not change $z(1)$ in the second homotopy group. Thus, we have a compatible stellar structure, and following Davidovich's methods \cite{davidovich} we obtain that for a discrete group $G$ homotopy $(G \times O(2))$-fixed points are quasi-biangular stellar $G$-algebras.
\end{proof}
Above lemma is an important step toward the verification of $(G\times O(2))$-structured cobordism hypothesis for $\Alg$-valued $2$-dimensional extended unoriented HFTs with $K(G,1)$-target. This version of the cobordism hypothesis states the equivalence of bigroupoids $\text{Frob}^G_{\ast}$ and $\big( \big( \Alge^{fd}\big)^{\sim}\big)^{h(G \times O(2))}_{\leq 2}$ where $\big( \big( \Alge^{fd}\big)^{\sim}\big)^{h(G \times O(2))}_{\leq 2}$ is the fundamental bigroupoid of $\big( \big( \Alge^{fd}\big)^{\sim}\big)^{h(G \times O(2))}$. Lemma \ref{go2fix} implies that the objects of these bigroupoids coincide. The next step is to give an explicit (algebraic) description of the bigroupoid $\big( \big( \Alge^{fd}\big)^{\sim}\big)^{h(G \times O(2))}_{\leq 2}$ and to write down the $2$-functor $\mathcal{F}: \text{Frob}^G_{\ast} \xrightarrow[]{\simeq} \big( \big( \Alge^{fd}\big)^{\sim}\big)^{h(G \times O(2))}_{\leq 2}$ similar to the oriented case. Here we skip these steps, which may later appear elsewhere.

\appendix
\section{Unbiased Semistrict Symmetric Monoidal 2-categories}\label{freely_gen_bicat_section}
In this section we recall unbiased semistrict monoidal $2$-categories, their computadic versions, and prove Theorem \ref{computadic}. Our main reference is \cite{schommer} (Section 2.10). A similar exposition is given in the Appendix of \cite{piotr}.

\subsection{String diagrams for bicategories}\label{defn_unbiased_section}
The symmetric monoidal bicategory $\XB^{\PD}$ is not a fully weak symmetric monoidal bicategory but a certain stricter version. The strict bicategories we are interested in are unbiased semistrict symmetric monoidal $2$-categories introduced by Schommer-Pries \cite{schommer}. To recall their definition, we first review string diagrams for bicategories. 
%\begin{figure}[h]
%    \centering
%    \includesvg{strdiag_tez}
%    \caption{Pasting diagram with the corresponding string diagram and a string diagram for an unbiased semistrict symmetric monoidal $2$-category}
%    \label{fig:strdiag}
%\end{figure} 
\begin{figure}[h]
    \centering
    \def\svgwidth{\columnwidth}
    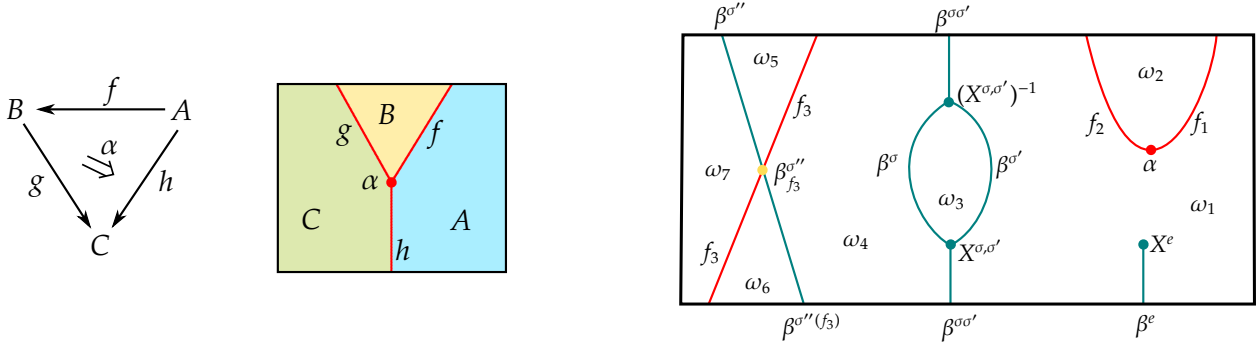
    \caption{Pasting diagram with the corresponding string diagram and a string diagram for an unbiased semistrict symmetric monoidal $2$-category}
    \label{fig:strdiag}
\end{figure}

Alternative to pasting diagrams, string diagrams are tools describing morphisms in a bicategory. Instead of arrows between objects and $1$-morphisms, a string diagram consists of regions, arcs, and vertices. Each region represents an object and each arc represents a $1$-morphism between objects whose corresponding regions share this arc as a common boundary. Each vertex represents a $2$-morphism between $1$-morphisms whose corresponding arcs are connected with each other via this vertex. On the left hand side of Figure \ref{fig:strdiag}, a pasting diagram and the corresponding string diagram is shown. Note that we read string diagrams from right to left and from top to bottom.

Unbiased semistrict symmetric monoidal $2$-categories are strict enough to admit a version of string diagram. An example of such a string diagram is shown in Figure \ref{fig:strdiag} in which regions are labeled with objects $\{\omega_i\}_{i=1}^7$, red arcs are labeled with $1$-morphisms $\{f_j\}_{j=1}^3$, and a red vertex is labeled with a $2$-morphism $\alpha$. However, there are additional strings and vertices of different colors coming from the coherence morphisms of an unbiased semistrict symmetric monoidal $2$-category. 
\begin{definition}[Definition 2.32, \cite{schommer}]\label{defnunbiased}
An \textit{unbiased semistrict symmetric monoidal $2$-category} is a triple $(\mathcal{C}, \beta,X)$ where $\mathcal{C}=(\mathcal{C}, \otimes, \one, \alpha, \lambda, \rho,\mathcal{P},\mathcal{M},\mathcal{L},\mathcal{R})$ is a monoidal bicategory (see Appendix of \cite{tezim}) such that 
\begin{enumerate}[(i)]
    \item the underlying bicategory is a strict $2$-category,
    \item transformations $\alpha,\lambda, \rho$ and modifications $\mathcal{P}, \mathcal{M},\mathcal{L},\mathcal{R}$ are identities,
    \item monoidal product $\otimes= \big(\otimes, \phi^{\otimes}_{(f,f'),(g,g')},\phi^{\otimes}_{(a,a')}\big): \mathcal{C} \times \mathcal{C} \to \mathcal{C}$ is cubical. That is, the interchanger 
    \begin{align*}
        \phi^{\otimes}_{(f,f'),(g,g')}: (f \otimes f') \circ (g \otimes g') \to (f \circ g) \otimes (f' \circ g')
    \end{align*}
    is identity if either $f$ or $g'$ is identity $1$-morphism and $\phi^{\otimes}_{(a,a')}: \id_{a \otimes a'} \to \id_{a} \otimes \id_{a'}$ is identity for all objects $a$ and $a'$.
\end{enumerate}
Secondly, $\beta$ denotes a collection of transformations (turquoise edges and yellow point in Figure \ref{fig:strdiag}) $$\{\beta^{\sigma}: (\mathcal{C}_1 \otimes \mathcal{C}_2 \otimes \dots \otimes \mathcal{C}_n \to \mathcal{C}) \to (\mathcal{C}_{\sigma(1)} \otimes \mathcal{C}_{\sigma(2)} \otimes \dots \otimes \mathcal{C}_{\sigma(n)} \to \mathcal{C}) \}_{\sigma \in S_n, n \geq 0}$$ where $\mathcal{C}_i=\mathcal{C}$ for all $i=1,\dots,n$ and $S_0 := \{ \one \}$ with $\beta^{\one}: (\one \hookrightarrow \mathcal{C}) \to (\one \hookrightarrow \mathcal{C})$ being the identity transformation between inclusion of functors. Lastly, $X$ denotes a collection of invertible modifications (turquoise points in Figure \ref{fig:strdiag})
\begin{align*}
    X^{\sigma,\sigma'} : (\beta^{\sigma}\ast 1) \circ  \beta^{\sigma'} \to \beta^{\sigma \sigma'} \quad \text{ and } \quad 
    X^e : \text{id} \to \beta^{e}
\end{align*}
for every $\sigma,\sigma' \in S_n$ and identity element $e \in S_n$ such that
\begin{enumerate}[(i)]
    \item transformations $\{\beta^{\sigma}\}_{\sigma \in S_n, n \geq 0}$ and modifications $\{ X^{\sigma,\sigma'}, X^e\}_{\sigma,\sigma',e \in S_n, n \geq 0}$ satisfy the following conditions
    \begin{align*}
    \beta^{\text{id} \sqcup \sigma} &= \text{id} \otimes \beta^{\sigma}, &
     \beta^{\sigma \sqcup \text{id}} &= \beta^{\sigma} \otimes \text{id},\\
    X^{(\text{id} \sqcup \sigma),(\text{id} \sqcup \sigma')} &= \text{id} \ast X^{\sigma,\sigma'}, &
    X^{(\sigma \sqcup \text{id}),(\sigma' \sqcup \text{id})} &= X^{\sigma,\sigma'}\ast \text{id}
\end{align*}
and the first three conditions on Figure \ref{fig:axiomsun} for all $\sigma, \sigma',\sigma'',e \in S_n$ and $n> 0$,
    \item for a fixed $n> 0$ and a collection of $n$ natural numbers $\{k_i\}_{i=1}^n$ let $\tilde{\sigma}\in S_{N}$ be given by the operadic product\footnote{By an operadic product we mean the composition $\sigma (\tau_1,\dots,\tau_n)$.} $\sigma \circ (\tau_i)$ where $N = \sum_i k_i$, $\sigma \in S_n$, and $\tau_i \in S_{k_i}$ for all $i=1,2,\dots,n$. Then $2$-morphism $\beta^{\sigma}_{(\beta^{\sqcup \tau_i})}=\beta^{\sqcup \tau_{\sigma(i)}} \circ \beta^{\sigma} \to \beta^{\sigma} \circ \beta^{\sqcup \tau_i}$ satisfies the equality given by the last condition on Figure \ref{fig:axiomsun} for all $n > 0$, $\sigma \in S_n$, and $\tau_i \in S_{k_i}$. In particular, when $\tau_i=e$ for all $i=1,\dots,n$, we have $\beta^{\tilde{\sigma}}=\beta^{\sigma}$, $X^{\tilde{e}}=X^e$, and $X^{\tilde{\sigma},\tilde{\sigma}'}=X^{\sigma,\sigma}$ for all $\sigma,\sigma',e \in S_n$, 
    %\begin{figure}[ht]
    %\centering
    %\includesvg{unnbiasedgen1_tez}
    %\caption{Some of the axioms of an unbiased semistrict symmetric monoidal $2$-category}
    %\label{fig:axiomsun}
%\end{figure}
\begin{figure}[ht]
    \centering
    \def\svgwidth{\columnwidth}
    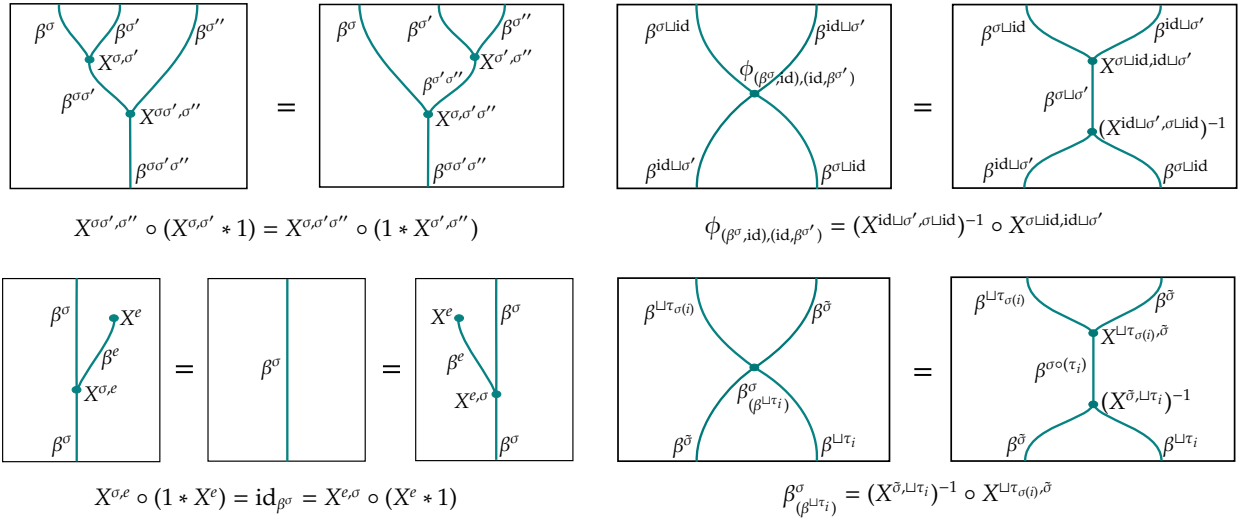
    \caption{Some of the axioms of an unbiased semistrict symmetric monoidal $2$-category}
    \label{fig:axiomsun}
\end{figure}
    \item transformations $\{\beta^{\sigma}\}_{\sigma \in S_n, n \geq 0}$ and modifications $\{ X^{\sigma,\sigma'}, X^e\}_{\sigma,\sigma',e \in S_n, n \geq 0}$ satisfy the conditions given by the reflections of diagrams in Figure \ref{fig:axiomsun} with respect to a horizontal axis.
\end{enumerate}
\end{definition}
In order to prove Theorem \ref{computadic}, we first need to show that the symmetric monoidal bicategory $\XB^{\PD}$ is an unbiased semistrict symmetric monoidal $2$-category. Recall that objects of $\XB^{\PD}$ are finite set of ordered oriented points, $1$-morphisms are isotopy classes of $G$-linear diagrams, and $2$-morphisms are equivalence classes of $G$-planar diagrams. 

\begin{lemma}
Chambering sets, graphs, and foams equip $\XB^{\PD}$ with the structure of an unbiased semistrict symmetric monoidal $2$-category. 
\end{lemma}
\begin{proof}
Recall that compositions of morphisms in $\XB^{\PD}$ are given by the concatenation of diagrams. Since $1$-morphisms are isotopy classes of $G$-linear diagrams and $2$-morphisms are equivalence classes of $G$-planar diagrams the underlying bicategory is a strict $2$-category. The symmetric monoidal structure of $\XB^{\PD}$ is cubical by definition. The transformations $\alpha,\lambda,\rho$ and modifications $\mathcal{P},\mathcal{L}$, $\mathcal{M}$ and $\mathcal{R}$ are identity since $2$-morphisms are equivalence classes of $G$-planar diagrams. The local models CP and $\text{CK}_4$ of chambering foam shown in Figure \ref{fig:cfoam} give two conditions on the left hand side of Figure \ref{fig:axiomsun}. For the remaining two conditions recall that a chambering graph can only have univalent and trivalent vertices. 
\end{proof}
Let $(\mathcal{C}, \beta,X)$ be an unbiased semistrict symmetric monoidal $2$-category. The invertibility of modifications $\{X^{\sigma,\sigma'}, X^e\}_{\sigma,\sigma',e \in S_n, n \geq 0}$ and axioms of unbiased semistrict symmetric monoidal $2$-category generate relations between structure morphisms. These relations are given on the left hand side of Figure\footnote{Note that different labelings of string diagrams are possible and each possible labeling is a relation.} \ref{fig:foamrels} in terms of string diagrams. Since chambering foams are responsible for the relations between boundary chambering graphs, on the right hand side of Figure \ref{fig:foamrels} chambering foams corresponding to these relations are shown.  
%\begin{figure}[h!]
%   \centering
%    \includesvg{foamrelns_tez2}
%    \caption{Relations between string diagrams for unbiased semistrict symmetric monoidal $2$-categories and the corresponding spatial foams}
%   \label{fig:foamrels}
%\end{figure}
\begin{figure}[h!]
    \centering
    \def\svgwidth{\columnwidth}
    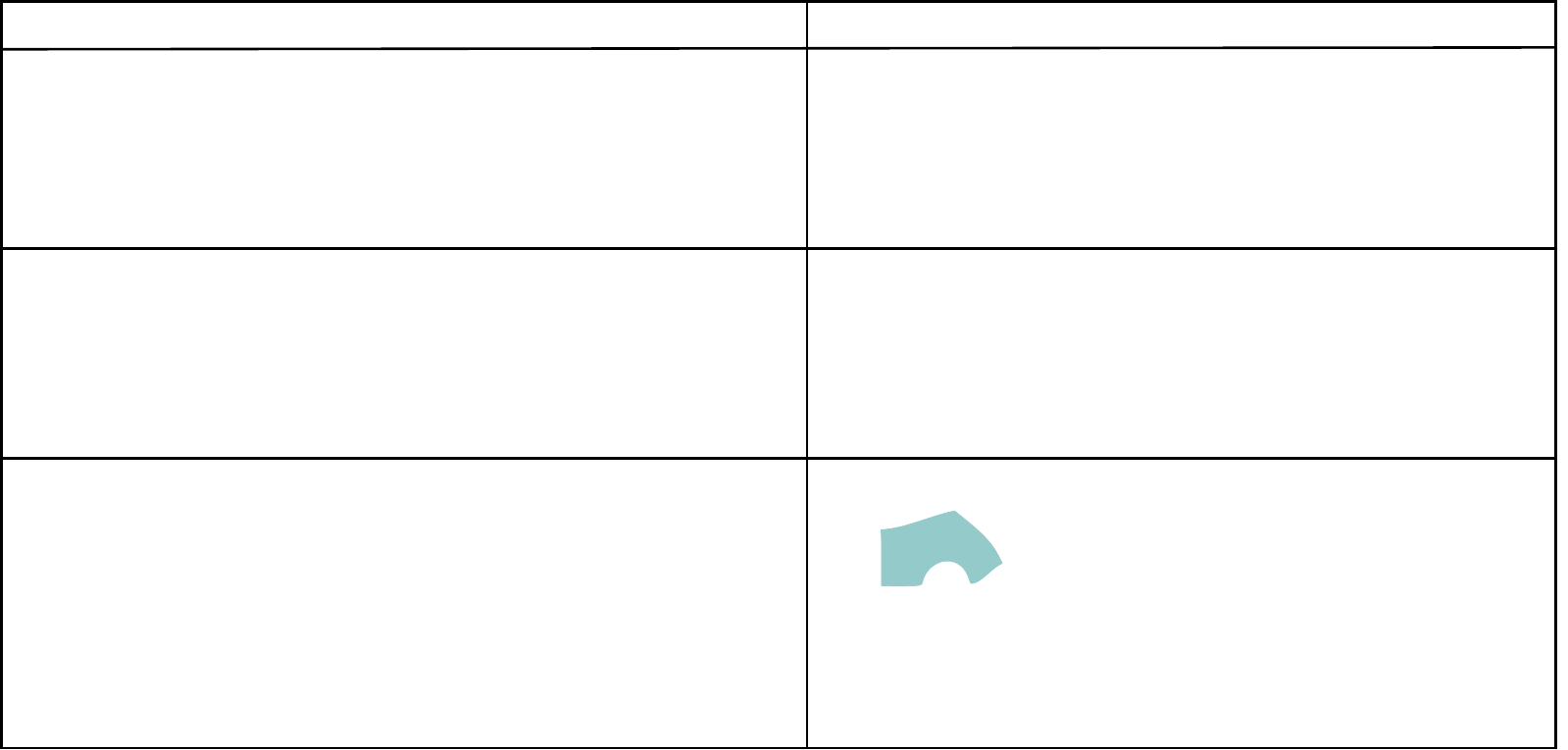
    \caption{Relations between string diagrams for unbiased semistrict symmetric monoidal $2$-categories and the corresponding spatial foams}
   \label{fig:foamrels}
\end{figure}
\subsection{Computadic unbiased semistrict symmetric monoidal $2$-categories}\label{free_gen_G}
To finish the proof of Theorem \ref{computadic}, we need to show that the unbiased semistrict symmetric monoidal $2$-category $\XB^{\PD}$ is computadic. As the next step, we review computadic unbiased semistrict symmetric monoidal $2$-categories and show that $\XB^{\PD}$ is an example. Such a $2$-category is constructed from a certain presentation (unbiased semistrict symmetric monoidal $3$-computad) which we call \textit{unbiased semistrict presentation}. This type of presentation $\P$ consists of four sets $(\G_0,\G_1,\G_2,\mathcal{R})$ together with source and target maps $s,t :\G_1 \to BW^{\uss}(\G_0)$, and $s,t: \G_2 \to BS^{\uss}(\G_1)$, which we describe below. For a given such $\P=(\G_0,\G_1,\G_2,\mathcal{R},s,t)$, the four sets are respectively called \textit{generating objects}, \textit{generating $1$-morphisms}, \textit{generating $2$-morphisms}, and \textit{generating relations among $2$-morphisms}. The following series of definitions start with the ingredients of $\mathsf{F}_{\uss}(\P)$ and continues with definition of each ingredient in the given order.
\begin{definition}
For a given unbiased semistrict presentation $\P=(\G_0,\G_1,\G_2,\mathcal{R},s,t)$, the objects of $\F_{\uss}(\P)$ are \textit{binary words} in $\G_0$, the $1$-morphisms are \textit{binary sentences} in $\G_1$, and the $2$-morphisms are \textit{equivalence classes of paragraphs} in $\G_2$. 
\end{definition}
\begin{definition}\label{binarywords_uss}
Let $\G_0$ be a set. The set $BW^{\uss}(\G_0)$ of \textit{binary words in $\G_0$} contains the symbol $\one$, the elements of $\G_0$, and $\otimes$-products i.e. $a \otimes b \in BW^{\uss}(\G_0)$ for all $a,b \in BW^{uss}(\G_0)$ such that for any $a \in \G_0$ the elements $\one \otimes a$, $a$, and $a \otimes \one$ are identified.
\end{definition}
Since binary words in $\G_0$ form the objects of $\mathsf{F}_{\uss}(\P)$, the set $\G_1$ of generating $1$-morphisms is equipped with source and target maps $s,t: \G_1 \to BW^{\uss}(\G_0)$.    
\begin{definition}
Let $\G_1$ be a set equipped with maps $s,t: \G_1 \to BW^{\uss}(\G_0)$. The set $BW^{\uss}(\G_1)$ of \textit{binary words in $\G_1$} contains elements of $\G_1$, $\text{id}_a$, $\beta_{a,\sigma(a)}^{\sigma}$ for any $a \in BW^{\uss}(\G_0)$ and $\sigma S_n$ where $a$ is a word of length $n$. The extension of the source and target maps to these elements are as follows; $s(\text{id}_a)=a$, $t(\text{id}_a)=a$, $s(\beta_{a,\sigma(a)}^{\sigma})=a$, $t(\beta_{a,\sigma(a)}^{\sigma})=\sigma(a)$.
\end{definition}
\begin{definition} \label{binarysentences_uss}
Let $BW^{\uss}(\G_1)$ be a set of binary words in $\G_1$ with $s,t: BW^{\uss}(\G_1)\to BW^{\uss}(\G_0)$. The set $\underline{BS^{\uss}}(\G_1)$ contains binary words in $\G_1$, compositions $g \circ f$ for any $f,g \in BW^{\uss}(\G_1)$ with $s(g)=t(f)$, and monoidal products $f \otimes g$ for any $f,g\in BW^{\uss}(\G_1)$. Source and target maps extend naturally to $\underline{BS^{\uss}}(\G_1)$ by
\begin{itemize}
    \item $s(g \circ f)=s(f)$ and $t(g \circ f)=t(g)$ for any $g \circ f \in BS^{\uss}(\G_1)$,
    \item $s(f \otimes g)=s(f) \otimes s(g)$ and $t(f \otimes g)=t(f)\otimes t(g)$ for any $f,g \in BW^{\uss}(\G_1)$.
\end{itemize}
The set $BS^{\uss}(\G_1)$ of \textit{binary sentences in $\G_1$} is the quotient $\underline{BS^{\uss}}(\G_1)/\sim$ where $\sim$ is the smallest equivalence relation generated by the identifications; $f \otimes \id_{\one}\sim f, f \sim \id_{\one} \otimes f$, $f \circ \id_a \sim f \sim \id_b \circ f$, and $f \otimes f' \sim (\id_b \otimes f') \circ (f \otimes \id_{a'})$ for any $f,f' \in BS^{\uss}(\G_1)$ with $s(f)=a$, $t(f)= b$, $s(f')= a'$, and $t(f')= b'$ (see Lemma 2.81 in \cite{schommer}).
\end{definition}
Since binary sentences in $\G_1$ form the $1$-morphisms of $\mathsf{F}_{\uss}(\P)$, the set $\G_2$ of generating $2$-morphisms is equipped with source and target maps $s,t: \G_2 \to BS^{\uss}(\G_1)$ satisfying $s\circ s=s \circ t$ and $t \circ s= t \circ t$. Then, an \textit{unbiased semistrict symmetric monoidal $2$-computad} $\P^{\G}$ consists of generating sets $\G_0$, $\G_1$, $\G_2$ together with maps $s,t: \G_1 \to BW^{\uss}(\G_0)$, and $s,t: \G_2 \to BS^{\uss}(\G_1)$ satisfying $s\circ s=s \circ t$ and $t\circ s=t \circ t$.

\begin{definition}
Let $\G_2$ be a set equipped with $s,t:\G_2 \to BS^{\uss}(\G_1)$ satisfying $s\circ s=s \circ t$ and $t \circ s= t \circ t$. The set $BW^{\uss}(\G_2)$ of \textit{binary words in $\G_2$} contains every element of $\G_2$ and the symbols in Table \ref{table:bw1G} for every $f \in BW^{\uss}(\G_1)$ and for every $f_1,f_2,f_3,f_4 \in BS^{\uss}(\G_1)$ with $s(f_1)=t(f_3)$ and $s(f_2)=t(f_4)$, and for all $\sigma, e \in S_n, \bar{\sigma}\in S_m$ for $n,m \geq 0$.  
\begin{table}[ht]
    \centering
    \begin{tabular}{|c||c|c|}
     \multicolumn{1}{c||}{\footnotesize$\mathsf{Symbol}$} & \multicolumn{1}{c}{\footnotesize$\mathsf{Source}$} & \multicolumn{1}{c}{\footnotesize$\mathsf{Target}$}\\
     \hline
     $\text{id}_{f}$  & $f$  & $f$ \\
	 \hline
     $\phi^{\otimes}_{(f_1,f_2),(f_3,f_4)}$ & $(f_1 \otimes f_2) \circ (f_3 \otimes f_4)$ & $(f_1 \circ f_3) \otimes (f_2 \circ f_4)$     \\
     $\phi^{\otimes}_{a,a'}$ & $\id_{a \otimes a'}$  & $\id_{a} \otimes \id_{a'}$   \\
	 ${}_{r}\beta_{f}^{\sigma}$ & $f \circ \beta^{\sigma}$ & $\beta^{\bar{\sigma}} \circ f$ \\
	 ${}_{l}\beta_{f}^{\sigma}$  & $ \beta^{\sigma} \circ f$ & $  f \circ \beta^{\bar{\sigma}}$ \\
	 \hline
	 $X^{\sigma,\sigma'}$ & $(\beta^{\sigma} \ast 1) \circ \beta^{\sigma'}$ & $\beta^{\sigma \sigma'}$ \\
	 $X^e$ & $\text{id}$ & $\beta^e$ \\
	 \hline
    \end{tabular}
    \caption{Binary words in $\G_2$}
    \label{table:bw1G}
    \vspace{-0.2cm}
\end{table}
Moreover, $BW^{\uss}(\G_2)$ contains the inverses of symbols in Table \ref{table:bw1G} except for symbols containing $\beta^{\sigma}$'s. The set of \textit{pre-paragraphs} $\underline{PG^{\uss}}(\G_2)$ is constructed from $BW^{\uss}(\G_2)$ by adding compositions and monoidal products as above. Similar to $BS^{\uss}(\G_1)$ there are certain identifications on $\underline{PG^{\uss}}(\G_2)$ generated by $\phi^{\otimes}_{(\text{id}, f_2),(f_3,f_4)}= \text{id}$, $\phi^{\otimes}_{(f_1,f_2),(f_3,\text{id})}=\text{id}$, and $\phi^{\otimes}_{(a,a')}= \text{id}$ for all $f_1,f_2,f_3,f_4 \in B^{\uss}(\G_1)$ and $a,a' \in BW^{\uss}(\G_0)$. We consider the smallest equivalence relation $\sim$ on $\underline{PG^{\uss}}(\G_2)$ generated by these identifications along with identifications $\{ p \circ p^{-1} = \id_{t(p)}, \ p^{-1}\circ p =\id_{s(p)}\}_{p \in \underline{PG^{\uss}}(\G_2)}$ and those for $\beta$ and $X$ coming from the definition of unbiased semistrict symmetric monoidal $2$-category. The quotient set is denoted by $PG^{\uss}(\G_2)$ and called the set of \textit{paragraphs} in $\G_2$.
\end{definition}
\begin{definition}\label{presentation}
The set $\mathcal{R}$ of generating relations among $2$-morphisms for an unbiased semistrict symmetric monoidal $2$-computad $\P^{\G}=(\G_0,\G_1,\G_2,s,t)$ consists of pairs $(F,G)$ of paragraphs in $\G_2$ in $\F(\P^{\G})$ with $s(F)=s(G)$ and $t(F)=t(G)$. An \textit{unbiased semistrict presentation or unbiased semistrict symmetric monoidal $3$-computad $\P$} consists of an unbiased semistrict symmetric monoidal $2$-computad $\P^{\G}$ and a set $\mathcal{R}$ of generating relations among $2$-morphisms for $\P^{\G}$. 

The $2$-morphisms of the computadic unbiased semistrict symmetric monoidal $2$-category $\F_{\uss}(\P)$ are the $\wr$-equivalence classes of paragraphs in $\G_2$ where $\wr$ is the smallest equivalence relation on $PG^{\uss}(\G_2)$ such that $\wr$ is generated by $\mathcal{R}$ and closed under compositions and monoidal products. An unbiased semistrict symmetric monoidal $2$-category $(\mathcal{C},\beta,X)$ is called \textit{computadic} if there exists a strict symmetric monoidal equivalence $\mathcal{F}: \F_{\uss}(\P) \to \mathcal{C}$ for an unbiased semistrict presentation $\P$.
\end{definition}

In simpler terms, $\mathsf{F}_{\uss}(\P)$ can be described as follows. The objects of $\mathsf{F}_{\uss}(\P)$ are words in $\G_0$. There are two kinds of elementary $1$-morphisms which can be described as
\begin{enumerate}[(i)]
    \item $\beta^{\sigma}_{a,\sigma(a)}: a \to \sigma(a)$ where $\sigma \in S_n$ for $n \geq 0$ and $a$ is a word of length $n$,
    \item $\id_{a} \otimes f \otimes \id_{b}$ where $f \in \G_1$ and $a,b \in BW^{\uss}(\G_0)$,   
\end{enumerate}
so that nonidentity $1$-morphisms of $\mathsf{F}_{\uss}(\P)$ are given by compositions of elementary $1$-morphisms. The $2$-morphisms of $\mathsf{F}_{\uss}(\P)$ are the equivalence classes of string diagrams where two string diagrams are equivalent if they are related by finitely many (local) moves which come from generating relations $\mathcal{R}$, Figures \ref{fig:axiomsun} and \ref{fig:foamrels}, and the naturality of $\beta$ and $X$ with generating morphisms. Compositions of morphisms are given by horizontal and vertical concatenations of string diagrams while (cubical) monoidal product is given by stretching out diagrams from different horizontal directions and merging them (see Figure \ref{fig:composition}). 
\begin{example}\label{XPex}
Consider an unbiased semistrict presentation $\XP=(\mathcal{XG}_0, \mathcal{XG}_1,\mathcal{XG}_2, \allowbreak \mathcal{XR})$ whose generating sets are given as; $\mathcal{XG}_0= \{ \nokta,\noktab \}$, $\mathcal{XG}_1$ consists of $\{ \foldcapg, \foldcupg, \linartg, \lineksg \}$ i.e. $G$-linear diagrams of $\{ \leftelbow, \rightelbow ,\arti, \allowbreak \eksi \}_{g \in G}$ without chambering sets, and $\mathcal{XG}_2$ consists of $G$-planar diagrams without chambering graphs of generating $2$-morphisms in Figure \ref{fig:generators}. The set of relations $\mathcal{XR}$ consists of pairs of $G$-planar diagrams corresponding to pairs of $\<2\>$-X-surfaces in Figure \ref{fig:relations}. 

An object of $\mathsf{F}_{\uss}(\XP)$ is either $\one$ or words in $\nokta$ and $\noktab$. Each $1$-morphism is a composition of the following two types of elementary $1$-morphisms; $\beta^{\sigma}_{a,\sigma(a)}$ where $\sigma \in S_n$ and $a$ is a word of length $n$ and a $G$-linear diagram whose $1$-morphism is labeled with $\id_{a_1} \otimes \dots \otimes f \otimes \dots \otimes \id_{a_n}$ for some $0<k \leq n$ where $a_i$ is either $\nokta$ or $\noktab$, and $f \in \G_1$. Here we use $\id_{\nokta}=\linarte$, $\id_{\noktab}=\linekse$, and $f \otimes f' = (\id_b \otimes f') \circ (f \otimes \id_{a'})$. The $2$-morphisms of $\mathsf{F}_{\uss}(\XP)$ are equivalence classes of paragraphs $PG^{\uss}(\mathcal{XG}_2)$ where the equivalence relation is generated by the set of generating relations $\mathcal{XR}$ (see Figures \ref{fig:movie1}, \ref{fig:eskileri_relationlarin_yenilenmesi}, \ref{fig:new_relations}, and \ref{fig:relations}), and the string diagrams given in Figures \ref{fig:movies2}, \ref{fig:axiomsun}, and \ref{fig:foamrels}. 
\end{example}
\subsection{Proof of Theorem \ref{computadic}}\label{proof_of_computadic}
Note that the string diagram interpretation of elements of $PG^{\uss}(\mathcal{XG}_2)$ coincides with the string diagram interpretation of $2$-morphisms of $\XB^{\PD}$ except for (possible) black points on $G$-linear diagrams (see Figure \ref{fig:stringdiagex}). More precisely, the sets of labels for regions coincide, in both string diagrams there are two types of $1$-morphisms whose sets of labels and possible intersecting patterns coincide, and in both string diagrams there are three types of vertices whose sets of labels coincide for each type of vertex. Lastly, equivalence relations on both string diagrams are generated by the same local moves, equivalently by the same movie-moves. This observation suggests an isomorphism between $\mathsf{F}_{\uss}(\XP)$ and $\XB^{\PD}$ namely a symmetric monoidal equivalance preserving the unbiased semistrict symmetric monoidal structures. The following lemma shows that this is indeed the case and finishes the proof of Theorem \ref{computadic}.
\begin{lemma}
There exists a canonical isomorphism $\Theta: \mathsf{F}_{\uss}(\XP) \to \XB^{\PD}$ of unbiased semistrict symmetric monoidal $2$-categories. 
\end{lemma}
\begin{proof}
Comparing the descriptions of unbiased semistrict symmetric monoidal $2$-categories $\mathsf{F}_{\uss}(\XP)$ and $\XB^{\PD}$ given above, it is not hard to define the $2$-functor $\Theta$. On the level of objects $\Theta$ maps $\one$ to the empty set and words in set $\{\nokta, \noktab \}$ to the finite ordered oriented points given by the words.   

On $1$-morphisms it is enough to specify the images of elementary $1$-morphisms $\{\beta^{\sigma}, \id_{a_1} \otimes \dots \otimes f \otimes \dots \otimes \id_{a_n}\}$ where $\sigma \in S_n, n \geq 0, f \in \mathcal{XG}_1$, and $a_i \in \{ \nokta, \noktab\}$ for $1 \leq i \leq n$. For $\sigma \in S_n$ and a word $a$, the $1$-morphism $\Theta(\beta^{\sigma}_{a,\sigma(a)})$ is a $G$-linear diagram whose chambering set has only one element labeled by $\beta^{\sigma}$. The latter $1$-morphism is mapped to a $G$-linear diagram described in Example \ref{XPex}. Recall that $1$-morphisms of $\mathsf{F}_{\uss}(\XP)$ are equivalence classes determined by certain identifications and $1$-morphisms of $\XB^{\PD}$ are isotopy classes of $G$-linear diagrams. It is not hard to see that above assignments are well-defined on $1$-morphisms. 

The $2$-functor $\Theta$ maps the equivalence class $[P]$ of a paragraph $P \in PG^{\uss}(\mathcal{XG}_2)/\sim$ to the equivalence class of string diagram corresponding to $P$. This assignment makes sense because as mentioned in Example \ref{XPex} any representative string diagram can be interpreted in both $2$-categories. Since in both $2$-categories $\mathsf{F}_{\uss}(\XP)$ and $\XB^{\PD}$ string diagrams are considered up to the same lists of local moves (movie-moves), this assignment is well-defined. 

We use Whitehead theorem for symmetric monoidal bicategories (Theorem \ref{whiteadforsym}) to show that $\Theta$ is a symmetric monoidal equivalence. It is clear that $\Theta$ is essentially surjective on objects. We claim that $\Theta$ is essentially full on $1$-morphisms. To prove this, it is enough to show that every $G$-linear diagram is isomorphic to a composition of $1$-morphisms $\big\{ \Theta(\beta^{\sigma}_{a,\sigma(a)}),\Theta(\id_{a_1} \otimes \dots \otimes \id_{a_n}),\Theta(\id_{a_1} \otimes \dots \otimes f \otimes \dots \otimes \id_{a_n})\big\}$ for some $n \geq 0, \sigma \in S_n$, and $f \in \mathcal{XG}_1$. For a given $G$-linear diagram it is obvious to see how to write it as a composition of these diagrams except for extra black points. Recall that when we compose $G$-linear diagrams we do not remove black points along which two diagrams are concatenated. However, there are invertible $G$-planar diagrams which remove these points (see Figure \ref{fig:newdigrams2}). Therefore, up to invertible $2$-morphisms, every $G$-linear diagram can be written as compositions of above $1$-morphisms.

Recall that $G$-planar diagrams are formed using generic maps and the X-manifold data of cobordism type $\<2\>$-X-surfaces. Thus, any $G$-planar diagram can be obtained from generating $2$-morphisms in Figure \ref{fig:generators} under horizontal and vertical compositions, and symmetric monoidal product operation. This implies that for any $G$-planar diagram there exists a paragraph such that their equivalence classes are matched by $\Theta$. Consequently, $\theta$ is fully-faithful on $2$-morphisms. 

Hence, Whitehead theorem implies that $\Theta$ is an equivalence. By definition $\Theta$ preserves the unbiased semistrict symmetric monoidal structures. That is, $\XB^{\PD}$ is a computadic unbiased semistrict symmetric monoidal $2$-category. 
\end{proof}

\printbibliography
\Address

\end{document}